\documentclass[11pt]{amsart}
\usepackage{amssymb}
\usepackage{amsmath}
\usepackage{bbm}
\usepackage{graphicx} 
\usepackage[english]{babel} 
\usepackage{float}
\usepackage{mathtools}
\usepackage{hyperref}
\usepackage[noadjust]{cite}
\usepackage[margin=3.2cm]{geometry}
\hypersetup{
    colorlinks=true,       % false: boxed links; true: colored links
    linkcolor=blue,          % color of internal links
    citecolor=blue,        % color of links to bibliography
    filecolor=blue,      % color of file links
    urlcolor=blue           % color of external links
}

\usepackage[pdf]{pstricks} % Compile using XeLatex, View PDF
\usepackage{epsfig}
\usepackage{pst-grad} % For gradients
\usepackage{pst-plot} % For axes
\usepackage[space]{grffile} % For spaces in paths
\usepackage{etoolbox} % For spaces in paths
\makeatletter % For spaces in paths
\patchcmd\Gread@eps{\@inputcheck#1 }{\@inputcheck"#1"\relax}{}{}
\makeatother

\usepackage[all]{xy}
\usepackage{mathrsfs}
\usepackage{graphics}
\usepackage{amsthm}
\theoremstyle{plain}\newtheorem{theorem}{Theorem}[section]\newtheorem{Theorem}{Theorem}\newtheorem{proposition}[theorem]{Proposition}\newtheorem{lemma}[theorem]{Lemma}\newtheorem{corollary}[theorem]{Corollary}
\def\eq{\coloneqq}
\newtheorem*{claim*}{Claim}

\theoremstyle{definition}\newtheorem{definition}[theorem]{Definition}\newtheorem{example}[theorem]{Example}\newtheorem{notation}[theorem]{Notation}\newtheorem{remark}[theorem]{Remark}%\newtheorem{construction}[subsection]{Construction}

\def\C{\mathbb{C}}\def\N{\mathbb{N}}\def\Z{\mathbb{Z}}\def\R{\mathbb{R}}\def\K{\mathbb{K}}\def\T{\mathbb{T}}

\def\ZZ2{\mathbb{\Z/ 2\Z}}
\def\su{\subset}
\def\lb{\langle}\def\rb{\rangle}\def\disp{\displaystyle}\def\ot{\otimes}\def\t{\times}\def\sm{\setminus}
\def\c{\gamma}\def\v{\varphi}

\def\a{\alpha}\def\b{\beta}\def\d{\delta}\def\e{\epsilon}\def\s{\sigma}\def\De{\Delta}\def\La{\Lambda}\def\o{\omega}\def\la{\lambda}\def\la{\lambda}\def\n{\nabla}\def\p{\partial}\def\S{\Sigma}
\def\mf{\mathfrak}\def\mc{\mathcal}\def\ov{\overline}\def\wh{\widehat}\def\wt{\widetilde}

\def\gl{\mathfrak{gl}}
\def\sl2{\mathfrak{sl}_2}
\def\su2{\mathfrak{su}(2)}
\def\Aut{\text{Aut}}
\def\id{\text{id}}
\def\Hom{\text{Hom}}\def\rk{\text{rk}\,}
\def\sign{\text{sign}}\def\sgn{\text{sgn}}\def\aug{\text{aug}}
\def\inte{\text{int} \,}\def\Sym{\text{Sym}}
\def\deg{\text{deg}}\def\Spinc{\text{Spin}^c}

\def\TH3{\Theta_3^{H}}

\def\x{\textbf{x}}\def\y{\textbf{y}}
\def\aa{\boldsymbol\a}\def\bb{\boldsymbol\b}\def\cc{\boldsymbol\c}
\def\HD{(\S,\aa,\bb)}\def\HH{\mc{H}}
\def\HDzw{(\S,\aa,\bb,z,w)}
\def\Tab{\T_{\a}\cap\T_{\b}}\def\Ta{\T_{\a}}\def\Tb{\T_{\b}}

\def\ss{\mf{s}}

\def\L{\mc{L}}\def\cc{\boldsymbol\c}
\def\Uq{U_q(\mf{sl}_2)}\def\Usl2{U_q(\sl2)}\def\usl2{\wt{U}_q(\sl2)}\def\Uqgl11{U_q(\mathfrak{gl}(1|1))}\def\Uqsl11{U_q(\mathfrak{sl}(1|1))}\def\Uq{U_q(\mf{sl}_2)}

\def\II1{\text{II}_1}

\def\mod2{\ (mod \ 2)}
\def\L2{L^2}\def\l2{l^2(G)}\def\Cn2{C_n^{(2)}}\def\Hn2{H_n^{(2)}}\def\bn2{b_n^{(2)}}\def\pn2{\p_n^{(2)}}
 
\def\W1p{W^{1,p}}\def\Wd1p{W_{\d}^{1,p}}

\def\L{\mathcal{L}}

\def\ZZ{\mathcal{Z}}

\def\kk{\mathbb{K}}\def\Svect{\text{SVect}_{\kk}}\def\gl{\mathfrak{gl}}\def\gl11{\mathfrak{gl}(1|1)}

\def\Ka{\K[\a^{\pm 1}]}

\def\bolda{\boldsymbol{a}}\def\boldb{\boldsymbol{b}}
\def\aaa{\aa\cup\bolda}\def\bbb{\bb\cup\boldb}

\def\bstar{\b^*}

\def\aaa{\aa^e}\def\bbb{\bb^e}\def\aaap{\aa'^e}\def\bbbp{\bb'^e}

\def\waaa{w_{\aa^e}}\def\wbbb{w_{\bbb}}

\def\HDD{(\S,\aaa,\bbb)}\def\HDDp{(\S,\aaap,\bbbp)}

\def\dHH{\d_{\o}(\HH)}\def\dHHp{\d_{\o}(\HH')}
\def\Ua{U_{\a}}\def\Ub{U_{\b}}

\def\nequiv{\not\equiv}

\begin{document}

\title{Kuperberg invariants for balanced sutured 3-manifolds}
\author{Daniel L\'opez Neumann}
\address{Institut de Math\'ematiques de Jussieu - Paris Rive Gauche, Universit\'e de Paris, Paris, France}
\email{daniel.lopez@imj-prg.fr}

\maketitle

\begin{abstract}

We construct quantum invariants of balanced sutured 3-manifolds with a $\Spinc$ structure out of an involutive (possibly non-unimodular) Hopf superalgebra $H$. If $H$ is the Borel subalgebra of $\Uqgl11$, we show that our invariant is computed via Fox calculus and it is a normalization of Reidemeister torsion. The invariant is defined via a modification of a construction of G. Kuperberg, where we use the $\Spinc$ structure to take care of the non-unimodularity of $H$ or $H^*$.

%The invariant is a scalar of the form $Z_H^{\rho}(M,\c,\ss,\o)$ where $(M,\c)$ is a balanced sutured manifold, $\ss\in\Spinc(M,\c)$, $\o$ is an orientation of $H_*(M,R_-;\R)$ and $\rho: H_1(M)\to G$ is an abelian representation into a group that depends on $H$.
\end{abstract}

\tableofcontents

\section[Introduction]{Introduction}

%dim(H)\neq 0. Say just char zero.
% Write eq of coint rather than integral. 
% Don't say sutured complement of a link, but rather specify that we put two sutures.
% Say degree zero linear maps for def of Hopf superalgebra (possible confusion with antipode, since it is an algebra antihomomorphism).

Let $H=(H,m,\eta,\De,\e,S)$ be a finite dimensional Hopf algebra over a field $\kk$ of characteristic zero. When $H$ is involutive (i.e. $S\circ S=\id_H$), G. Kuperberg constructed in \cite{Kup1} a topological invariant $Z_H(Y)\in\kk$ for closed oriented 3-manifolds $Y$. This invariant is constructed directly from a Heegaard diagram of $Y$ using the structure morphisms of $H$ and the Hopf algebra integral and cointegral. Recall that a right cointegral is an element $c\in H$ such that
\begin{align}
\label{eqintro: right cointegral standard} 
cx=c\e(x)
\end{align}
for all $x\in H$ (here $\e:H\to \kk$ is the counit) and a right integral is an element $\mu\in H^*$ satisfying the dual property. Left cointegrals and integrals are defined in a similar way. The definition of the invariant $Z_H(Y)$ is quite simple as a consequence of the following fact \cite{LR:cosemisimplechar0}: if $H$ is an involutive finite dimensional Hopf algebra over a field $\kk$ of characteristic zero, then both $H$ and $H^*$ are semisimple and so they are {\em unimodular}, meaning that a right cointegral is also a left cointegral and similarly for integrals. However, there are not too many examples of involutive Hopf algebras apart from the group-algebra of a finite group $G$, in which case the invariant $Z_H(Y)$ becomes the number of homomorphisms from $\pi_1(Y)$ into $G$. In particular, the Hopf algebra $\Uq$ involved in the construction of the Witten-Reshetikhin-Turaev (WRT) 3-manifold invariants \cite{RT2} is non-involutive. More interesting involutive examples exist if we consider Hopf {\em superalgebras} instead, such as the quantum group $\Uqgl11$, but an involutive Hopf superalgebra needs not be unimodular, so the construction of \cite{Kup1} does not applies. An extension of the invariant $Z_H(Y)$ that is defined for any finite dimensional Hopf (super)algebra was given by Kuperberg himself in \cite{Kup2}, but it produces an invariant of 3-manifolds endowed with a framing, which is quite difficult to represent on a Heegaard diagram. Moreover, though it has been recently established that this extension coincides with the so called Hennings invariant of the Drinfeld double \cite{CC:ontwoinvariants}, the relation with WRT invariants is less understood (except for \cite{CKS:relation-WRT-Henn}) and it is not clear how any of these constructions is related to classical invariants such as Reidemeister torsion.
\medskip

In this paper, we propose a modification of Kuperberg's invariant in the case of a (possibly non-unimodular) involutive Hopf superalgebra $H$ and we generalize it to balanced sutured 3-manifolds $(M,\c)$. To take care of the non-unimodularity of $H$ or $H^*$ we use $\Spinc$ structures, which we represent in a very simple way using multipoints as in \cite{OS1, Juhasz:holomorphic}. When $H$ is the Borel subalgebra of $\Uqgl11$, we show that our invariant is computed via Fox calculus and is a normalization of the Reidemeister torsion of the pair $(M,R_-(\c))$, where $R_-(\c)$ is half of $\p M$.

\subsection[Main results]{Main results}

Let $H$ be an involutive Hopf superalgebra over a field $\kk$, not necessarily of finite dimension. We define a {\em right relative cointegral} as a tuple $(A,\pi_A,i_A,\iota)$ where $A$ is a  Hopf subalgebra of $H$, $i_A:A\to H$ is the inclusion, $\pi_A:H\to A$ is a (cocentral) Hopf morphism such that $\pi_Ai_A=\id_A$ and $\iota:A\to H$ is an $A$-colinear map satisfying an equation analogue to (\ref{eqintro: right cointegral standard}) (see Section \ref{section: relative integrals and cointegrals}). A relative right integral is a tuple $(B,i_B,\pi_B,\mu)$ defined in a dual way. The failure of $\iota$ (resp. $\mu$) being a two-sided cointegral is measured by a group-like $a^*\in G(A^*)$ (resp. $b\in G(B)$), we require these to satisfy the compatibility condition of Definition \ref{def: compatible integral cointegral}. Note that when $A$ is of infinite dimension, we define $G(A^*)$ as the group of algebra morphisms $A\to\kk$ (and similarly for $B$). Our first theorem is the following.

\begin{Theorem}
\label{thm: intro Z is an invariant}
Let $H$ be an involutive Hopf superalgebra over $\kk$ endowed with a compatible right relative cointegral and right relative integral. Then there is a topological invariant $$I_H^{\rho}(M,\c,\ss,\o)\in\kk$$ where $(M,\c)$ is a balanced sutured manifold, $\ss\in\Spinc(M,\c)$, $\o$ is an orientation of $H_*(M,R_-(\c);\R)$ and $\rho:H_1(M)\to G(A\ot B^*)$ is a group homomorphism.
\end{Theorem}

%This theorem is the combination of Theorem \ref{thm: Z is an invariant} (the case $R_-$ connected) and Proposition \ref{prop: Z for disconnected case}. 
%Note that the (abelian) sign-refined torsion of a three-manifold $M$ has the form $\tau^{\rho}(M,\ss,\o)$ where $\ss\in\Spinc(M)$, $\o$ is an orientation of $H_*(M;\R)$ and $\rho:H_1(M)\to\kk^{\t}$ is a group homomorphism, see \cite{Turaev: torsion}. %Note that $\Spinc$ structures and homology orientations are standard in the theory of sign-refined torsions (\cite{Turaev: torsion}). In the sutured case, the sutured torsion defined in \cite{FJR11} has the form $\tau^{\rho}(M,\c,\ss,\o)$ where $\ss,\o$ are exactly as above and $\rho:H_1(M)\to\kk^{\t}$ is any homomorphism into a field.
%Thus, the appearance of such $\ss,\o$ and $\rho$ in Theorem \ref{thm: intro Z is an invariant} is quite natural.
\medskip

We briefly explain the ingredients of this theorem. Sutured manifolds were introduced by Gabai in \cite{Gabai:foliations} in order to study foliations of 3-manifolds, and they provide a common generalization of closed 3-manifolds, link complements and Seifert surface complements. A sutured manifold is a pair $(M,\c)$ where $M$ is a compact oriented 3-manifold with non-empty boundary and $\c$ is a collection of pairwise disjoint annuli in $\p M$ that divides the boundary into two subsurfaces $R_-(\c)$ and $R_+(\c)$. The balanced condition amounts to $\chi(R_-(\c))=\chi(R_+(\c))$ and holds in most interesting examples. A $\Spinc$ structure on a sutured manifold $(M,\c)$ is an homology class of non-vanishing vector fields on $M$ with prescribed behavior over $\p M$. The set of $\Spinc$ structures over $(M,\c)$ forms an affine space over $H^2(M,\p M)$. In the closed case, $\Spinc$ structures are in bijection with Turaev's combinatorial Euler structures \cite{Turaev:Spinc, Turaev:BOOK2} and they are used, along with homology orientations, to remove the indeterminacies of Reidemeister torsion. In \cite{Juhasz:holomorphic}, Juh\'asz proves that sutured manifolds are described by Heegaard diagrams $\HH=\HD$ in which $\S$ has non-empty boundary (a {\em sutured} Heegaard diagram) and that $\Spinc$ structures on them can be easily understood in terms of $\HH$: if $\Tab$ denotes the set of multipoints of $\HH$ then there is a map $$s:\Tab\to\Spinc(M,\c)$$ with nice properties (this map was first defined by Ozsv\'ath and Szab\'o \cite{OS1} in the closed case). 
\medskip

The main steps in the construction of the invariant $I_H^{\rho}(M,\c,\ss,\o)$ (at least when $R_-(\c)$ is connected) are the following. First, we take an {\em extended Heegaard diagram} of $(M,\c)$: this is a tuple $\HH=(\S,\aa,\bolda,\bb,\boldb)$ consisting of a sutured Heegaard diagram $(\S,\aa,\bb)$ of $(M,\c)$ together with two collections $\bolda,\boldb$ of properly embedded pairwise disjoint arcs on $\S$ that give handlebody decompositions of $R_-,R_+$ respectively. If all the curves and arcs are oriented and each curve in $\aa\cup\bb$ has a basepoint, then the construction of Kuperberg \cite{Kup1} can be generalized to this setting and produces a scalar $Z_H^{\rho}(\HH)\in\kk$. More precisely, one assigns a tensor to each $\a\in\aa\cup\bolda$ and to each $\b\in\bb\cup\boldb$, and $Z_H^{\rho}(\HH)$ is obtained by contracting all these tensors using the combinatorics of the Heegaard diagram. The tensors corresponding to the closed $\a$-curves (resp. $\b$-curves) involve the cointegral $\iota$ (resp. $\mu$), while the tensors corresponding to the arcs in $\bolda$ (resp. $\boldb$) involve the Hopf morphism $i_A$ (resp. $\pi_B$), this is justified by Proposition \ref{prop: hsliding property for cointegral arc-arc, arc-curve, curve-curve}. However, the scalar $Z_H^{\rho}(\HH)$ is not a topological invariant if $\iota,\mu$ are not both two-sided (unless $\rho$ is trivial), due to some indeterminacies coming from the basepoints and orientations chosen on $\aa\cup\bb$. To correct this, we devise in Subsection \ref{subsection: basepoints} a special rule to put basepoints in $\aa\cup\bb$ once a multipoint $\x\in\Tab$ is given. If $Z^{\rho}_H(\HH,\x,\o)$ denotes the scalar above with the basepoints coming from $\x$ (and with a correction sign involving $\o$), then this depends on $\x$ in the same way as the $\Spinc$ structure $s(\x)$ does. Hence, if we fix $\ss\in\Spinc(M,\c)$ and set
$$\zeta_{\ss,\x}\eq\lb \rho(PD[\ss-s(\x)]), (a^*)^{-1}\ot b\rb$$
where $a^*\in G(A^*)$ and $b\in G(B)$ are the group-likes associated to $\iota$ and $\mu$, then
$$I_H^{\rho}(M,\c,\ss,\o)\eq\zeta_{\ss,\x} Z_H^{\rho}(\HH,\x,\o)$$
is independent of all the choices made in its construction, see Theorem \ref{thm: Z is an invariant}. The case when $R_-(\c)$ is disconnected is reduced to the connected case in Subsection \ref{subs: disconnected case}. 
\medskip

To state our second theorem, let $H_0$ be the Borel subalgebra of $\Uqgl11$ at generic $q$. Thus, $H_0$ is the superalgebra generated by two commuting elements $K,X$ with $|K|=0,|X|=1$ and $X^2=0$. The coproduct is given by
\begin{align}
\label{eqintro: coproduct of X}
\De(X)=K\ot X+X\ot 1
\end{align}
and we let $K$ be group-like. This (involutive) Hopf superalgebra does not quite satisfies the hypothesis needed in Theorem \ref{thm: intro Z is an invariant}, but each finite dimensional quotient $H_n\eq H_0/(K^n-1), n\geq 1$ does. This is enough to define an invariant $I_{H_0}(M,\c,\ss,\o)$ (where $(M,\c),\ss,\o$ are as above) which is a well-defined element of $\Z[H_1(M)]$ and which specializes to all the invariants $I^{\rho}_{H_n}$. % by the formula $$\rho(Z_{H_0}(M,\c,\ss,\o))=Z_{H_n}^{\rho}(M,\c,\ss,\o),$$ 
On the other hand, there is a relative Reidemeister torsion $\tau(M,R_-)$ which is an element of $\Z[H_1(M)]$ defined up to multiplication by an element of $\pm H_1(M)$ (see e.g. \cite{Turaev:BOOK2} or \cite{FJR11}).

\begin{Theorem}
\label{thm: intro Z recovers Reidemeister torsion}
Let $(M,\c)$ be a balanced sutured manifold, $\ss\in\Spinc(M,\c)$ and $\o$ an orientation of $H_*(M,R_-;\R)$. If $H_0$ is the Borel subalgebra of $\Uqgl11$ then $$I_{H_0}(M,\c,\ss,\o)\,\dot{=}\,\tau(M,R_-)$$ where $\dot{=}$ means equality up to multiplication by an element of $\pm H_1(M)$.
\end{Theorem}

This theorem follows from a Fox calculus expression of $I_{H_0}$, stated in Theorem \ref{thm: intro Z via Fox calculus} for the $Z_{H_n}^{\rho}$ case. The reason such an expression exists is the following. Recall that if $F$ is a free group on generators $x_1,\dots, x_d$ and $u,v\in F$ then the Fox derivatives $\p/\p x_i:\Z[F]\to \Z[F]$ satisfy
\begin{align}
\label{eqintro: Fox derivation property}
\frac{\p (uv)}{\p x_i}=u\frac{\p v}{\p x_i}+\frac{\p u}{\p x_i}.
\end{align}
Theorem \ref{thm: intro Z via Fox calculus} follows from the observation that formula (\ref{eqintro: Fox derivation property}) is very similar to formula (\ref{eqintro: coproduct of X}) for the coproduct of $X\in H_n$, which is the cointegral of $H_n$ modulo group-likes. One has to think that $X$ corresponds to derivating, $K$ corresponds to ``leaving the variable untouched" and $1$ corresponds to ``erasing the variable". Theorem \ref{thm: intro Z recovers Reidemeister torsion} follows from Theorem \ref{thm: intro Z via Fox calculus} since $\tau(M,R_-)$ is also computed via Fox calculus. 
\medskip

As a consequence of Theorem \ref{thm: intro Z recovers Reidemeister torsion} it follows that our extension of Kuperberg's invariant contains interesting topological information. For example, if $M$ denotes the complement of a link $L\subset S^3$ and $\c$ consists of two sutures on each boundary component of $M$, then $\tau(M,R_-(\c))$ recovers the multivariable Alexander polynomial $\De_L$ \cite{FJR11}, and so also does $I_{H_0}$ (see Corollary \ref{corollary: Z0 recovers multivariable Alexander polynomial}). The torsion $\tau(M,R_-(\c))$ has also been used to detect non-isotopic minimal genus Seifert surfaces of some knots \cite{Altman:Seifert}. However, if $M$ is the sutured manifold associated to a closed 3-manifold $Y$ (i.e. $M=Y\sm B$ where $B$ is a 3-ball with a single suture in $\p B$), then the invariant $I_{H_0}^{\rho}$ is essentially trivial (see Corollary \ref{corollary: Z of closed 3-mfld}). It is an interesting question whether there exist other Hopf superalgebras for which the invariant $I_H^{\rho}$ contains non-trivial topological information not detected by Reidemeister torsion, though we do not consider this in the present paper.

\subsection[Comparison to previous works]{Comparison to previous works}

Another generalization of Kuperberg's invariant was considered by Virelizier \cite{Virelizier:flat} using an involutive {\em Hopf group-coalgebra} $\{H_g\}_{g\in G}$. This gives a scalar invariant of the form $\tau^{\rho}(Y)$ where $Y$ is a closed 3-manifold and $\rho:\pi_1(Y)\to G$ is an arbitrary group homomorphism. This is similar to our invariant in the closed case, but no $\Spinc$ structures are involved since it is an unimodular (even semisimple) approach. 
\medskip

%Our invariant recovers that of \cite{Kup1} when $A=B=\kk$ and $H,H^*$ are unimodular and $M$ is a closed 3-manifold. Another generalization of Kuperberg's invariant was considered in \cite{Virelizier: Hopf-group} using {\em Hopf-group coalgebras}. This is still an unimodular situation, so no $\Spinc$ structures are involved.

%Other generalizations of (involutive) Kuperberg invariants appear in \cite{Virelizier: flat bundles} and \cite{KV: generalized}. The approach of \cite{Virelizier: flat bundles} uses an involutive {\em Hopf group-coalgebra} \cite{Virelizier: Hopf-group} $\{H_g\}_{g\in G}$ and gives an invariant $Z^{\rho}(M)$ of a closed 3-manifold together with a representation $\rho:\pi_1(M)\to G$. This is still an unimodular approach, so no $\Spinc$ structures are involved. In \cite{KV: generalized} a closed 3-manifold invariant is constructed out of any finite dimensional involutive Hopf superalgebra with {\em central} distinguished group-likes. They consider the same Hopf superalgebra $H_n$ as us, but they only recover $\pm |H_1(M;\Z)|$ which is our invariant at $\rho\equiv 1$.\medskip

Theorem \ref{thm: intro Z recovers Reidemeister torsion} has to be seen as a direct Hopf algebra approach to Reidemeister torsion, in the sense that we use no $R$-matrices or representation theory. It is well-known that torsion invariants of links and 3-manifolds can be obtained via the representation theory of some Hopf algebra. It was first shown by Reshetikhin \cite{Reshetikhin:supergroup} and Rozansky-Saleur \cite{RS:Alexander} that the Alexander polynomial could be obtained from the representation theory of $\Uqgl11$. A similar result was obtained by Jun Murakami \cite{Murakami:Alexander, Murakami:state} using the representation theory of $\Uq$ at $q=i$, see \cite{Viro:Alexander} for more references. More recently, Blanchet, Costantino, Geer and Patureau-Mirand showed in \cite{BCGP} that the (absolute) Reidemeister torsion of a closed 3-manifold can be obtained via the representation theory of an ``unrolled" version of $\Uq$ at $q=i$. Our approach has the advantage that it is defined for more general manifolds (e.g. complements of Seifert surfaces), nevertheless, we do not get invariants of links colored by a representation of the Hopf algebra as in the works mentioned above, neither we obtain the absolute Reidemeister torsion of a closed 3-manifold as in \cite{BCGP}. Note also that our approach makes clear the connection of the Borel of $\Uqgl11$ to Fox calculus (Theorem \ref{thm: intro Z via Fox calculus}), while the above papers are more related to the skein relation.
\medskip

%The usual approach goes through the representation theory of some Hopf algebra (e.g. \cite{CGP},\cite{Viro: quantum relatives}, see also \cite{Sartori: Alexander}). In \cite{Viro: quantum relatives},\cite{Sartori: Alexander}, for example, the Alexander polynomial is obtained from the representation theory of $\Uqgl11$. Note that $\Uqgl11$ (at a root of unity) is the Drinfeld double of $H_n$, so in view of the result of \cite{CC: on two invariants}, it should be possible to get the Alexander polynomial as a Kuperberg invariant of $H_n$. Our theorem makes this connection explicit and generalized to the sutured context. 

Finally, note that Friedl, Juh\'asz and Rasmussen have shown in \cite{FJR11} that another $\Spinc$ normalization of the Reidemeister torsion $\tau(M,R_-)$ equals the Euler characteristic of Juh\'asz sutured Floer homology $SFH(M,\c,\ss)$ \cite{Juhasz:holomorphic}, a powerful homological invariant defined for balanced sutured manifolds. It is an interesting question whether there is a deeper connection between Floer homology and the Borel subalgebra of $\Uqgl11$ or even whether Floer homology could be extended to categorify other Hopf algebra invariants.

\subsection[Structure of the paper]{Structure of the paper} Section \ref{section: Hopf superalgebras} contains the necessary material on Hopf algebra theory. This is mostly standard, except for Subsection \ref{section: relative integrals and cointegrals} where we introduce relative integrals and cointegrals. In Section \ref{section: Sutured manifolds and Heegaard diagrams} we define sutured manifolds, their Heegaard diagrams, $\Spinc$ structures and homology orientations. This is also standard, except for Subsection \ref{subs: Extended Heegaard diagrams} where we introduce extended Heegaard diagrams and extend the Reidemeister-Singer theorem to this context. In Section \ref{section: the invariant} we construct in detail the invariant $I_H^{\rho}(M,\c,\ss,\o)$. We begin in Subsection \ref{subs: tensors of HDs} by generalizing the construction of \cite{Kup1} to the sutured context using extended Heegaard diagrams, assuming $R_-(\c)$ connected. In Subsection \ref{subsection: basepoints} we explain how to put basepoints on the Heegaard diagram once a multipoint is chosen. The invariant $I_H^{\rho}(M,\c,\ss,\o)$ is defined in Subsection \ref{subs: Construction of Z_H}. In Subsection \ref{subs: disconnected case} we generalize $I_H^{\rho}$ to the case when $R_-(\c)$ is disconnected. Proof of invariance is devoted to Section \ref{section: proof of invariance}. Finally, we investigate the relation to Reidemeister torsion in Section \ref{section: The torsion for sutured manifolds} in which we prove Theorem \ref{thm: intro Z recovers Reidemeister torsion}.

\subsection[Acknowledgments]{Acknowledgments}

I would like to thank Christian Blanchet, my PhD supervisor, for his continuing support and for all the help provided during the course of my PhD. I also would like to thank Rinat Kashaev and Alexis Virelizier for helpful conversations. I am particularly grateful to Alexis Virelizier for all the time he spend to discuss about this topic and the many suggestions he gave. This project has received funding from the European Union's Horizon 2020 Research and Innovation Programme under the Marie Sklodowska-Curie Grant Agreement No. 665850.

\section[Algebraic preliminaries]{Algebraic preliminaries}
\label{section: Hopf superalgebras}

\def\SVect{\text{SVect}_{\kk}}

This section contains the necessary material on Hopf superalgebras. We begin by recalling some basic definitions and conventions, such as super-vector spaces, tensor networks, and Hopf superalgebras. In Subsection \ref{section: relative integrals and cointegrals} we introduce relative versions of the Hopf algebra integrals and cointegrals. Then in Subsection \ref{subs: The compatibility condition} we discuss a compatibility condition for such structures, which generalizes the properties of the distinguished group-likes for the standard Hopf algebra (co)integrals. An excellent reference for the theory of Hopf algebras is \cite{Radford:BOOK}.

\medskip

%To simplify the graphical notation, we will omit the source and target of a morphism whenever they are clear from the context.

\subsection[Super vector spaces]{Super vector spaces}

Let $\kk$ be a field. In what follows, all vector spaces are defined over $\kk$. A {\em super vector space} is a vector space $V$ endowed with a direct sum decomposition
\begin{align*}
V=V_0\oplus V_1.
\end{align*}
An element $v\in V_i$ is said to be {\em homogeneous} of degree $i$ and we denote its degree by $|v|= i\pmod 2$. If $V,W$ are super vector spaces, a linear map $f:V\to W$ has {\em degree $k \pmod 2 $} if $f(V_i)\subset W_{i+k}$ for $i=0,1$ and we denote $|f|= k\pmod 2$. The class of super vector spaces and linear maps of degree zero forms a category, which we denote by $\Svect$. %We work over the category $\Svect$ of supervector spaces over $\kk$, i.e. $\Z/2\Z$-graded $\kk$-vector spaces. If $V$ is such an object, we denote by $|v|\in \Z/2\Z$ the degree of an homogeneous element $v\in V$. One could take as morphisms of this category the degree preserving $\kk$-linear maps or just all $\kk$-linear maps. In the latter case, for any supervector spaces $V,W$, $\Hom(V,W)$ becomes a supervector space whose degree $i$ part consists of degree $i$ $\kk$-linear maps, $i\in\Z/2\Z$. Note that for any superobject $V$ one can shift degrees and get a new object $\ov{V}$ with the same underlying vector space. A degree one map $V\to W$ can be seen as a degree zero map $V\to \ov{W}$, so we consider $\Svect$ with morphisms the degree zero linear maps. 
This is a monoidal category if for each $i=0,1$ we let
\begin{align*}
(V\ot W)_i\eq \bigoplus_{j=0,1}V_j\ot W_{i-j}
\end{align*}
for super vector spaces $V,W$, and the unit object is the base field $\kk$ concentrated in degree zero. The tensor product of two morphisms $f,g$ of $\Svect$ is defined by
\begin{align*}
(f\ot g)(v\ot w)\eq f(v)\ot g(w).
\end{align*}
 %Note that if we want to consider degree one linear maps as morphisms in $\Svect$, then the tensor product $f\ot g$ has to be defined by\begin{align*}(f\ot g)(v\ot w)\eq (-1)^{|v||g|}f(v)\ot g(w).\end{align*}
The category $\Svect$ is a symmetric monoidal category with symmetry isomorphism given by 
\begin{align*}
\label{eq: symmetry isomorphism of supervect}
\tau_{V,W}:V\ot W&\to W \ot  V \\
 v\ot w &\mapsto (-1)^{|v||w|}w\ot v
\end{align*}
where $v\in V,w\in W$ are homogeneous elements. %The symmetry $c$ is natural since we took morphisms as degree zero linear maps. 
If $V$ is a super vector space, then the dual vector space $V^*$ is naturally a super vector space with $V^*_i\eq(V_i)^*$ for $i=0,1$. This way $\Svect$ is a rigid symmetric monoidal category. %Finally, note that for any supervector space $V$, there is a supervector space $\ov{V}$ obtained by shifting degrees, i.e. $\ov{V}_i\eq V_{i+1}$ for $i=0,1$. Thus any degree one linear map $f:V\to W$ can be seen as a degree zero map $f:V\to \ov{W}$, which is a morphism in $\Svect$.\medskip

%\begin{notation}\label{notation: representation of Sn via Svect}Let $V$ be a supervector space. For each $n\geq 1$ and $1\leq i\leq n-1$ we can define $c_i:V^{\ot n}\to V^{\ot n}$ by $c_i\eq \id_V^{\ot (i-1)}\ot c\ot \id_V^{\ot(n-i-1)}$ where $c$ is the above symmetry. These induce a representation $S_n\to GL_{\Svect}(V^{\ot n})$ (the latter denotes automorphisms in the category $\Svect$). We denote by $P_{\tau}$ the image of $\tau\in S_n$ under this map. We also denote by $P'_{\tau}$ the unsigned permutation (i.e. we use the symmetry $c$ of $\Vect$ instead). Note that if $v_1,\dots,v_n\in V$ have degree one, the\begin{align*}P_{\tau}(\ov{v})=\sign(\tau)v_{\tau(1)}\ot\dots\ot v_{\tau(n)}=\sign(\tau)P'_{\tau}(\ov{v})\end{align*}where $\ov{v}=v_1\ot\dots\ot v_n$.\end{notation}

\subsection{Graphical notation}

We will adopt the tensor network notation of \cite{Kup1, Kup2}. This notation is valid over any (symmetric) monoidal category, but we will work only over $\SVect$. In this notation, the inputs of a tensor are denoted as incoming arrows in counterclockwise direction, while the outputs are denoted as outcoming arrows in clockwise direction. For example, a linear map $T:V\ot W\to V\ot W$ is denoted by 

\begin{figure}[H]
\centering
\begin{pspicture}(0,-0.445)(2.65,0.445)
\rput[bl](1.15,-0.1){$T$}
\psline[linecolor=black, linewidth=0.018, arrowsize=0.05291667cm 2.0,arrowlength=0.8,arrowinset=0.2]{->}(0.55,0.305)(0.95,0.105)
\psline[linecolor=black, linewidth=0.018, arrowsize=0.05291667cm 2.0,arrowlength=0.8,arrowinset=0.2]{->}(0.55,-0.295)(0.95,-0.095)
\psline[linecolor=black, linewidth=0.018, arrowsize=0.05291667cm 2.0,arrowlength=0.8,arrowinset=0.2]{->}(1.55,0.105)(1.95,0.305)
\psline[linecolor=black, linewidth=0.018, arrowsize=0.05291667cm 2.0,arrowlength=0.8,arrowinset=0.2]{->}(1.55,-0.095)(1.95,-0.295)
\rput[bl](0.05,0.205){$V$}
\rput[bl](0.0,-0.445){$W$}
\rput[bl](2.15,0.205){$V$}
\rput[bl](2.1,-0.445){$W$}
\rput[bl](2.6,-0.1){.}
\end{pspicture}
\end{figure}

\noindent Usually, the domain and target of a linear map will be clear from the context so they will be dropped from the notation. Thus, we denote $T$ simply by
\begin{figure}[H]
\centering \begin{pspicture}(0,-0.37947297)(1.604469,0.37947297)
\rput[bl](0.604469,-0.029472962){$T$}
\psline[linecolor=black, linewidth=0.018, arrowsize=0.05291667cm 2.0,arrowlength=0.8,arrowinset=0.2]{->}(0.004468994,0.37052703)(0.40446898,0.17052704)
\psline[linecolor=black, linewidth=0.018, arrowsize=0.05291667cm 2.0,arrowlength=0.8,arrowinset=0.2]{->}(0.004468994,-0.22947297)(0.40446898,-0.029472962)
\psline[linecolor=black, linewidth=0.018, arrowsize=0.05291667cm 2.0,arrowlength=0.8,arrowinset=0.2]{->}(1.004469,0.17052704)(1.404469,0.37052703)
\psline[linecolor=black, linewidth=0.018, arrowsize=0.05291667cm 2.0,arrowlength=0.8,arrowinset=0.2]{->}(1.004469,-0.029472962)(1.404469,-0.22947297)
\rput[bl](1.554469,-0.029){.}
\end{pspicture}
\end{figure}

\noindent A vector $v\in V$, a functional $v^*\in V^*$ and the identity $\id_V$ are denoted respectively by:
\begin{figure}[H]
\centering \begin{pspicture}(0,-0.155)(3.97,0.155)
\psline[linecolor=black, linewidth=0.018, arrowsize=0.05291667cm 2.0,arrowlength=0.8,arrowinset=0.2]{->}(0.32,-0.005)(0.72,-0.005)
\psline[linecolor=black, linewidth=0.018, arrowsize=0.05291667cm 2.0,arrowlength=0.8,arrowinset=0.2]{->}(1.62,-0.005)(2.02,-0.005)
\rput[bl](0.0,-0.075){$v$}
\rput[bl](2.15,-0.085){$v^*$}
\psline[linecolor=black, linewidth=0.018, arrowsize=0.05291667cm 2.0,arrowlength=0.8,arrowinset=0.2]{->}(3.42,-0.005)(3.82,-0.005)
\rput[bl](0.77,-0.155){,}
\rput[bl](2.47,-0.155){,}
\rput[bl](3.92,-0.105){.}
\end{pspicture}

\end{figure}

\noindent Composition of tensors is denoted by joining the corresponding outcoming and incoming legs of the tensor networks while the tensor product of two tensors is denoted by stacking the corresponding tensor networks one over another. For example, if $T_1:V\ot V\to V\ot V$ and $T_2:V\ot V\to V$ are two tensors, then the composition $(T_2\ot \id_V)\circ (\id_V\ot T_1)$ and the tensor product $T_1\ot T_2$ are respectively denoted by

\begin{figure}[H]
\centering
\begin{pspicture}(0,-0.8089459)(5.7044683,0.8089459)
\rput[bl](0.6044684,-0.3){$T_1$}
\psline[linecolor=black, linewidth=0.018, arrowsize=0.05291667cm 2.0,arrowlength=0.8,arrowinset=0.2]{->}(0.0044683837,0.1)(0.4044684,-0.1)
\psline[linecolor=black, linewidth=0.018, arrowsize=0.05291667cm 2.0,arrowlength=0.8,arrowinset=0.2]{->}(0.0044683837,-0.5)(0.4044684,-0.3)
\psline[linecolor=black, linewidth=0.018, arrowsize=0.05291667cm 2.0,arrowlength=0.8,arrowinset=0.2]{->}(1.1044683,-0.1)(1.5044684,0.1)
\psline[linecolor=black, linewidth=0.018, arrowsize=0.05291667cm 2.0,arrowlength=0.8,arrowinset=0.2]{->}(1.1044683,-0.3)(1.5044684,-0.5)
\psline[linecolor=black, linewidth=0.018, arrowsize=0.05291667cm 2.0,arrowlength=0.8,arrowinset=0.2]{->}(1.1044683,0.5)(1.5044684,0.3)
\rput[bl](1.7044684,0.0){$T_2$}
\psline[linecolor=black, linewidth=0.018, arrowsize=0.05291667cm 2.0,arrowlength=0.8,arrowinset=0.2]{->}(2.2044685,0.2)(2.6044683,0.2)
\rput[bl](4.6044683,0.4){$T_1$}
\psline[linecolor=black, linewidth=0.018, arrowsize=0.05291667cm 2.0,arrowlength=0.8,arrowinset=0.2]{->}(4.0044684,0.8)(4.4044685,0.6)
\psline[linecolor=black, linewidth=0.018, arrowsize=0.05291667cm 2.0,arrowlength=0.8,arrowinset=0.2]{->}(4.0044684,0.2)(4.4044685,0.4)
\psline[linecolor=black, linewidth=0.018, arrowsize=0.05291667cm 2.0,arrowlength=0.8,arrowinset=0.2]{->}(5.1044683,0.6)(5.5044684,0.8)
\psline[linecolor=black, linewidth=0.018, arrowsize=0.05291667cm 2.0,arrowlength=0.8,arrowinset=0.2]{->}(5.1044683,0.4)(5.5044684,0.2)
\psline[linecolor=black, linewidth=0.018, arrowsize=0.05291667cm 2.0,arrowlength=0.8,arrowinset=0.2]{->}(4.0044684,-0.8)(4.4044685,-0.6)
\psline[linecolor=black, linewidth=0.018, arrowsize=0.05291667cm 2.0,arrowlength=0.8,arrowinset=0.2]{->}(4.0044684,-0.2)(4.4044685,-0.4)
\rput[bl](4.6044683,-0.7){$T_2$}
\psline[linecolor=black, linewidth=0.018, arrowsize=0.05291667cm 2.0,arrowlength=0.8,arrowinset=0.2]{->}(5.1044683,-0.5)(5.5044684,-0.5)
\rput[bl](5.6544685,-0.0){.}
\rput[bl](2.6944685,-0.0){,}
\end{pspicture}

\end{figure}

\noindent If $V,W$ are super vector spaces, the symmetry isomorphism $\tau_{V,W}$ is denoted by a crossing pair of arrows. For example, if $T:V\to V\ot V$ is a tensor, then $\tau_{V,V}\circ T$ is denoted by

\begin{figure}[H]
\centering
\begin{pspicture}(0,-0.28455383)(1.75,0.28455383)
\rput[bl](0.6,-0.10455383){$T$}
\psline[linecolor=black, linewidth=0.018, arrowsize=0.05291667cm 2.0,arrowlength=0.8,arrowinset=0.2]{->}(0.0,0.015446167)(0.4,0.015446167)
\psbezier[linecolor=black, linewidth=0.018, arrowsize=0.05291667cm 2.0,arrowlength=0.8,arrowinset=0.2]{->}(1.0,0.115446165)(1.3,0.115446165)(1.4,0.015446167)(1.5,-0.2845538330078125)
\rput[bl](1.7,-0.1){.}
\psbezier[linecolor=black, linewidth=0.018, arrowsize=0.05291667cm 2.0,arrowlength=0.8,arrowinset=0.2]{->}(1.0,-0.08455383)(1.3,-0.08455383)(1.4,0.015446167)(1.5,0.3154461669921875)
\end{pspicture}

\end{figure}

\subsection[Hopf superalgebras]{Hopf superalgebras} 
Recall that a superalgebra is a tuple $(A,m,\eta)$ consisting of a super-vector space together with tensors $m:A\ot A\to A$ and $\eta:\kk\to A$ satisfying the usual associativity and unitality conditions. Dually, a supercoalgebra is a tuple $(C,\De,\e)$ satisfying coassociativity and counitality conditions. In tensor network notation, the supercoalgebra axioms are:

\begin{figure}[H]
\centering
\begin{pspicture}(0,-0.5688811)(12.0,0.5688811)
\psline[linecolor=black, linewidth=0.018, arrowsize=0.05291667cm 2.0,arrowlength=0.8,arrowinset=0.2]{->}(8.9,-0.018881073)(9.3,-0.018881073)
\rput[bl](7.9,0.22111893){$\epsilon$}
\psline[linecolor=black, linewidth=0.018, arrowsize=0.05291667cm 2.0,arrowlength=0.8,arrowinset=0.2]{->}(0.0,0.28111893)(0.3,0.28111893)
\rput[bl](0.46,0.18111892){$\Delta$}
\psline[linecolor=black, linewidth=0.018, arrowsize=0.05291667cm 2.0,arrowlength=0.8,arrowinset=0.2]{->}(0.9,0.38111892)(1.2,0.58111894)
\psline[linecolor=black, linewidth=0.018, arrowsize=0.05291667cm 2.0,arrowlength=0.8,arrowinset=0.2]{->}(0.9,0.18111892)(1.2,-0.018881073)
\rput[bl](1.36,-0.31888106){$\Delta$}
\psline[linecolor=black, linewidth=0.018, arrowsize=0.05291667cm 2.0,arrowlength=0.8,arrowinset=0.2]{->}(1.8,-0.11888108)(2.1,0.08111893)
\psline[linecolor=black, linewidth=0.018, arrowsize=0.05291667cm 2.0,arrowlength=0.8,arrowinset=0.2]{->}(1.8,-0.31888106)(2.1,-0.5188811)
\psline[linecolor=black, linewidth=0.018, arrowsize=0.05291667cm 2.0,arrowlength=0.8,arrowinset=0.2]{->}(3.0,-0.21888107)(3.3,-0.21888107)
\rput[bl](3.46,-0.31888106){$\Delta$}
\psline[linecolor=black, linewidth=0.018, arrowsize=0.05291667cm 2.0,arrowlength=0.8,arrowinset=0.2]{->}(3.9,-0.11888108)(4.2,0.08111893)
\psline[linecolor=black, linewidth=0.018, arrowsize=0.05291667cm 2.0,arrowlength=0.8,arrowinset=0.2]{->}(3.9,-0.31888106)(4.2,-0.5188811)
\rput[bl](4.36,0.18111892){$\Delta$}
\psline[linecolor=black, linewidth=0.018, arrowsize=0.05291667cm 2.0,arrowlength=0.8,arrowinset=0.2]{->}(4.8,0.38111892)(5.1,0.58111894)
\psline[linecolor=black, linewidth=0.018, arrowsize=0.05291667cm 2.0,arrowlength=0.8,arrowinset=0.2]{->}(4.8,0.18111892)(5.1,-0.018881073)
\psline[linecolor=black, linewidth=0.018, arrowsize=0.05291667cm 2.0,arrowlength=0.8,arrowinset=0.2]{->}(6.5,-0.018881073)(6.8,-0.018881073)
\rput[bl](6.96,-0.11888108){$\Delta$}
\psline[linecolor=black, linewidth=0.018, arrowsize=0.05291667cm 2.0,arrowlength=0.8,arrowinset=0.2]{->}(7.4,0.08111893)(7.7,0.28111893)
\psline[linecolor=black, linewidth=0.018, arrowsize=0.05291667cm 2.0,arrowlength=0.8,arrowinset=0.2]{->}(7.4,-0.11888108)(7.7,-0.31888106)
\rput[bl](11.7,-0.37888107){$\epsilon$}
\psline[linecolor=black, linewidth=0.018, arrowsize=0.05291667cm 2.0,arrowlength=0.8,arrowinset=0.2]{->}(10.3,-0.018881073)(10.6,-0.018881073)
\rput[bl](10.76,-0.11888108){$\Delta$}
\psline[linecolor=black, linewidth=0.018, arrowsize=0.05291667cm 2.0,arrowlength=0.8,arrowinset=0.2]{->}(11.2,0.08111893)(11.5,0.28111893)
\psline[linecolor=black, linewidth=0.018, arrowsize=0.05291667cm 2.0,arrowlength=0.8,arrowinset=0.2]{->}(11.2,-0.11888108)(11.5,-0.31888106)
\rput[bl](8.35,-0.1){=}
\rput[bl](9.65,-0.1){=}
\rput[bl](2.45,-0.1){=}
\rput[bl](5.25,-0.1){,}
\rput[bl](11.95,-0.1){.}
\end{pspicture}

\end{figure}

\begin{definition}
A {\em Hopf superalgebra} over $\kk$ is a super-vector space $H$ (over $\kk$) endowed with degree zero linear maps
\begin{figure}[H]
\centering
\begin{pspicture}(0,-0.31162995)(10.105809,0.31162995)
\psline[linecolor=black, linewidth=0.018, arrowsize=0.05291667cm 2.0,arrowlength=0.8,arrowinset=0.2]{->}(4.3058095,0.0)(4.705809,0.0)
\rput[bl](4.8658094,-0.1){$\Delta$}
\psline[linecolor=black, linewidth=0.018, arrowsize=0.05291667cm 2.0,arrowlength=0.8,arrowinset=0.2]{->}(5.3058095,0.1)(5.705809,0.3)
\psline[linecolor=black, linewidth=0.018, arrowsize=0.05291667cm 2.0,arrowlength=0.8,arrowinset=0.2]{->}(5.3058095,-0.1)(5.705809,-0.3)
\psline[linecolor=black, linewidth=0.018, arrowsize=0.05291667cm 2.0,arrowlength=0.8,arrowinset=0.2]{->}(0.005809326,-0.3)(0.4058093,-0.1)
\psline[linecolor=black, linewidth=0.018, arrowsize=0.05291667cm 2.0,arrowlength=0.8,arrowinset=0.2]{->}(0.005809326,0.3)(0.4058093,0.1)
\rput[bl](0.5658093,-0.05){$m$}
\psline[linecolor=black, linewidth=0.018, arrowsize=0.05291667cm 2.0,arrowlength=0.8,arrowinset=0.2]{->}(1.0058093,0.0)(1.4058093,0.0)
\rput[bl](2.4458094,-0.1){$\eta$}
\psline[linecolor=black, linewidth=0.018, arrowsize=0.05291667cm 2.0,arrowlength=0.8,arrowinset=0.2]{->}(2.8058093,0.0)(3.2058094,0.0)
\psline[linecolor=black, linewidth=0.018, arrowsize=0.05291667cm 2.0,arrowlength=0.8,arrowinset=0.2]{->}(6.8058095,0.0)(7.205809,0.0)
\rput[bl](7.4058094,-0.06){$\epsilon$}
\rput[bl](9.145809,-0.1){$S$}
\psline[linecolor=black, linewidth=0.018, arrowsize=0.05291667cm 2.0,arrowlength=0.8,arrowinset=0.2]{->}(9.50581,0.0)(9.905809,0.0)
\psline[linecolor=black, linewidth=0.018, arrowsize=0.05291667cm 2.0,arrowlength=0.8,arrowinset=0.2]{->}(8.605809,0.0)(9.00581,0.0)
\rput[bl](1.5258093,-0.1){,}
\rput[bl](3.3258092,-0.1){,}
\rput[bl](5.725809,-0.1){,}
\rput[bl](7.625809,-0.1){,}
\rput[bl](10.055809,-0.1){,}
\end{pspicture}
\end{figure}
\noindent called multiplication, unit, comultiplication, counit and antipode respectively and subject to the following axioms. First, we require that $(H,m,\eta)$ is a superalgebra and that $(H,\De,\e)$ is a supercoalgebra. Second, $m,\De$ satisfy
\begin{figure}[H]
\centering
\begin{pspicture}(0,-0.42)(5.8058095,0.42)
\psline[linecolor=black, linewidth=0.018, arrowsize=0.05291667cm 2.0,arrowlength=0.8,arrowinset=0.2]{->}(0.005809326,-0.32)(0.4058093,-0.12)
\psline[linecolor=black, linewidth=0.018, arrowsize=0.05291667cm 2.0,arrowlength=0.8,arrowinset=0.2]{->}(0.005809326,0.28)(0.4058093,0.08)
\rput[bl](0.5558093,-0.1){$m$}
\psline[linecolor=black, linewidth=0.018, arrowsize=0.05291667cm 2.0,arrowlength=0.8,arrowinset=0.2]{->}(1.0058093,-0.02)(1.4058093,-0.02)
\rput[bl](1.5658094,-0.12){$\Delta$}
\psline[linecolor=black, linewidth=0.018, arrowsize=0.05291667cm 2.0,arrowlength=0.8,arrowinset=0.2]{->}(2.0058093,0.08)(2.4058094,0.28)
\psline[linecolor=black, linewidth=0.018, arrowsize=0.05291667cm 2.0,arrowlength=0.8,arrowinset=0.2]{->}(2.0058093,-0.12)(2.4058094,-0.32)
\rput[bl](2.6858094,-0.1){=}
\psline[linecolor=black, linewidth=0.018, arrowsize=0.05291667cm 2.0,arrowlength=0.8,arrowinset=0.2]{->}(3.2058094,0.28)(3.6058092,0.28)
\rput[bl](3.7658093,0.18){$\Delta$}
\psline[linecolor=black, linewidth=0.018, arrowsize=0.05291667cm 2.0,arrowlength=0.8,arrowinset=0.2]{->}(4.205809,0.28)(4.605809,0.28)
\rput[bl](4.7558093,0.2){$m$}
\psline[linecolor=black, linewidth=0.018, arrowsize=0.05291667cm 2.0,arrowlength=0.8,arrowinset=0.2]{->}(5.205809,0.28)(5.605809,0.28)
\psline[linecolor=black, linewidth=0.018, arrowsize=0.05291667cm 2.0,arrowlength=0.8,arrowinset=0.2]{->}(3.2058094,-0.32)(3.6058092,-0.32)
\rput[bl](3.7658093,-0.42){$\Delta$}
\psline[linecolor=black, linewidth=0.018, arrowsize=0.05291667cm 2.0,arrowlength=0.8,arrowinset=0.2]{->}(4.205809,-0.32)(4.605809,-0.32)
\rput[bl](4.7558093,-0.4){$m$}
\psline[linecolor=black, linewidth=0.018, arrowsize=0.05291667cm 2.0,arrowlength=0.8,arrowinset=0.2]{->}(5.205809,-0.32)(5.605809,-0.32)
\psline[linecolor=black, linewidth=0.018, arrowsize=0.05291667cm 2.0,arrowlength=0.8,arrowinset=0.2]{->}(4.205809,-0.22)(4.705809,0.18)
\psline[linecolor=black, linewidth=0.018, arrowsize=0.05291667cm 2.0,arrowlength=0.8,arrowinset=0.2]{->}(4.205809,0.18)(4.705809,-0.22)
\rput[bl](5.7558093,-0.1){.}
\end{pspicture}

\end{figure}

\noindent The crossing arrows in the above formula represent the symmetry $\tau_{H,H}$ of $\SVect$. Finally, we require the following axiom for the antipode:

\begin{figure}[H]
\centering
\begin{pspicture}(0,-0.43391204)(10.0,0.43391204)
\rput[bl](0.57,-0.10608795){$\Delta$}
\psline[linecolor=black, linewidth=0.018, arrowsize=0.05291667cm 2.0,arrowlength=0.8,arrowinset=0.2]{->}(1.0,0.09391205)(1.4,0.29391205)
\rput[bl](1.5,0.19391204){$S$}
\psbezier[linecolor=black, linewidth=0.018, arrowsize=0.05291667cm 2.0,arrowlength=0.8,arrowinset=0.2]{->}(1.0,-0.10608795)(1.3,-0.50608796)(1.9,-0.50608796)(2.2,-0.10608795166015625)
\psline[linecolor=black, linewidth=0.018, arrowsize=0.05291667cm 2.0,arrowlength=0.8,arrowinset=0.2]{->}(1.8,0.29391205)(2.2,0.09391205)
\rput[bl](2.34,-0.07608795){$m$}
\psline[linecolor=black, linewidth=0.018, arrowsize=0.05291667cm 2.0,arrowlength=0.8,arrowinset=0.2]{->}(0.0,-0.006087952)(0.4,-0.006087952)
\psline[linecolor=black, linewidth=0.018, arrowsize=0.05291667cm 2.0,arrowlength=0.8,arrowinset=0.2]{->}(2.8,-0.006087952)(3.2,-0.006087952)
\rput[bl](7.27,-0.10608795){$\Delta$}
\psline[linecolor=black, linewidth=0.018, arrowsize=0.05291667cm 2.0,arrowlength=0.8,arrowinset=0.2]{->}(8.5,-0.30608794)(8.9,-0.10608795)
\rput[bl](8.2,-0.40608796){$S$}
\psbezier[linecolor=black, linewidth=0.018, arrowsize=0.05291667cm 2.0,arrowlength=0.8,arrowinset=0.2]{->}(7.7,0.09391205)(8.0,0.49391204)(8.6,0.49391204)(8.9,0.09391204833984375)
\psline[linecolor=black, linewidth=0.018, arrowsize=0.05291667cm 2.0,arrowlength=0.8,arrowinset=0.2]{->}(7.7,-0.10608795)(8.1,-0.30608794)
\rput[bl](9.04,-0.07608795){$m$}
\psline[linecolor=black, linewidth=0.018, arrowsize=0.05291667cm 2.0,arrowlength=0.8,arrowinset=0.2]{->}(6.7,-0.006087952)(7.1,-0.006087952)
\psline[linecolor=black, linewidth=0.018, arrowsize=0.05291667cm 2.0,arrowlength=0.8,arrowinset=0.2]{->}(9.5,-0.006087952)(9.9,-0.006087952)
\rput[bl](3.48,-0.1){=}
\rput[bl](6.18,-0.1){=}
\rput[bl](5.14,-0.10608795){$\eta$}
\psline[linecolor=black, linewidth=0.018, arrowsize=0.05291667cm 2.0,arrowlength=0.8,arrowinset=0.2]{->}(5.5,-0.006087952)(5.9,-0.006087952)
\psline[linecolor=black, linewidth=0.018, arrowsize=0.05291667cm 2.0,arrowlength=0.8,arrowinset=0.2]{->}(4.0,-0.006087952)(4.4,-0.006087952)
\rput[bl](4.6,-0.066087954){$\epsilon$}
\rput[bl](9.95,-0.1){.}
\end{pspicture}

\end{figure}
\noindent A Hopf superalgebra is said to be {\em involutive} if $S^2=\id_H$.
\end{definition}

We will use the following abbreviation for iterated multiplication and comultiplication:
\begin{figure}[H]
\centering
\begin{pspicture}(0,-0.60919243)(11.959192,0.60919243)
\psline[linecolor=black, linewidth=0.018, arrowsize=0.05291667cm 2.0,arrowlength=0.8,arrowinset=0.2]{->}(8.209192,0.0)(8.609193,0.0)
\rput[bl](8.769193,-0.09999993){$\Delta$}
\psline[linecolor=black, linewidth=0.018, arrowsize=0.05291667cm 2.0,arrowlength=0.8,arrowinset=0.2]{->}(9.209192,0.0)(9.609193,0.0)
\psline[linecolor=black, linewidth=0.018, arrowsize=0.05291667cm 2.0,arrowlength=0.8,arrowinset=0.2]{->}(9.209192,-0.19999993)(9.609193,-0.39999992)
\psline[linecolor=black, linewidth=0.018, arrowsize=0.05291667cm 2.0,arrowlength=0.8,arrowinset=0.2]{->}(2.0091925,-0.39999992)(2.4091926,-0.19999993)
\psline[linecolor=black, linewidth=0.018, arrowsize=0.05291667cm 2.0,arrowlength=0.8,arrowinset=0.2]{->}(2.0091925,0.0)(2.4091926,0.0)
\rput[bl](2.5691924,-0.049999923){$m$}
\psline[linecolor=black, linewidth=0.018, arrowsize=0.05291667cm 2.0,arrowlength=0.8,arrowinset=0.2]{->}(3.0091925,0.0)(3.4091926,0.0)
\psline[linecolor=black, linewidth=0.018, arrowsize=0.05291667cm 2.0,arrowlength=0.8,arrowinset=0.2]{->}(6.2091923,0.0)(6.6091924,0.0)
\rput[bl](6.7691927,-0.09999993){$\Delta$}
\psline[linecolor=black, linewidth=0.018, arrowsize=0.05291667cm 2.0,arrowlength=0.8,arrowinset=0.2]{->}(7.2091923,0.20000008)(7.6091924,0.6000001)
\psline[linecolor=black, linewidth=0.018, arrowsize=0.05291667cm 2.0,arrowlength=0.8,arrowinset=0.2]{->}(7.2091923,-0.19999993)(7.6091924,-0.5999999)
\rput[bl](7.4691925,-0.23999992){$\vdots$}
\psline[linecolor=black, linewidth=0.018, arrowsize=0.05291667cm 2.0,arrowlength=0.8,arrowinset=0.2]{->}(0.009192505,-0.5999999)(0.4091925,-0.19999993)
\psline[linecolor=black, linewidth=0.018, arrowsize=0.05291667cm 2.0,arrowlength=0.8,arrowinset=0.2]{->}(0.009192505,0.6000001)(0.4091925,0.20000008)
\rput[bl](0.06919251,-0.23999992){$\vdots$}
\rput[bl](0.5691925,-0.049999923){$m$}
\psline[linecolor=black, linewidth=0.018, arrowsize=0.05291667cm 2.0,arrowlength=0.8,arrowinset=0.2]{->}(1.0091925,0.0)(1.4091926,0.0)
\psline[linecolor=black, linewidth=0.018, arrowsize=0.05291667cm 2.0,arrowlength=0.8,arrowinset=0.2]{->}(0.009192505,0.30000007)(0.4091925,0.100000076)
\psline[linecolor=black, linewidth=0.018, arrowsize=0.05291667cm 2.0,arrowlength=0.8,arrowinset=0.2]{->}(4.2091923,-0.39999992)(4.6091924,-0.19999993)
\psline[linecolor=black, linewidth=0.018, arrowsize=0.05291667cm 2.0,arrowlength=0.8,arrowinset=0.2]{->}(4.2091923,0.0)(4.6091924,0.0)
\rput[bl](4.7691927,-0.049999923){$m$}
\psline[linecolor=black, linewidth=0.018, arrowsize=0.05291667cm 2.0,arrowlength=0.8,arrowinset=0.2]{->}(5.2091923,0.0)(5.6091924,0.0)
\rput[bl](1.6091925,-0.039999925){=}
\rput[bl](3.6091926,-0.09999993){$\dots$}
\rput[bl](7.8091927,-0.039999925){=}
\psline[linecolor=black, linewidth=0.018, arrowsize=0.05291667cm 2.0,arrowlength=0.8,arrowinset=0.2]{->}(7.2091923,0.100000076)(7.6091924,0.30000007)
\psline[linecolor=black, linewidth=0.018, arrowsize=0.05291667cm 2.0,arrowlength=0.8,arrowinset=0.2]{->}(10.409192,0.0)(10.809193,0.0)
\rput[bl](10.9691925,-0.09999993){$\Delta$}
\psline[linecolor=black, linewidth=0.018, arrowsize=0.05291667cm 2.0,arrowlength=0.8,arrowinset=0.2]{->}(11.409192,0.0)(11.809193,0.0)
\psline[linecolor=black, linewidth=0.018, arrowsize=0.05291667cm 2.0,arrowlength=0.8,arrowinset=0.2]{->}(11.409192,-0.19999993)(11.809193,-0.39999992)
\rput[bl](9.809193,-0.09999993){$\dots$}
\rput[bl](5.7091923,-0.09999993){,}
\rput[bl](11.909192,-0.09999993){.}
\end{pspicture}
\end{figure}

\noindent We will also use the following shorthand for the opposite multiplication and comultiplication:
\begin{figure}[H]
\centering
\begin{pspicture}(0,-0.30805147)(10.254022,0.30805147)
\rput[bl](6.304022,-0.1){$\Delta^{\text{op}}$}
\rput[bl](0.5540216,-0.05){$m^{\text{op}}$}
\psline[linecolor=black, linewidth=0.018, arrowsize=0.05291667cm 2.0,arrowlength=0.8,arrowinset=0.2]{->}(1.3040216,0.0)(1.7040216,0.0)
\psline[linecolor=black, linewidth=0.018, arrowsize=0.05291667cm 2.0,arrowlength=0.8,arrowinset=0.2]{->}(0.0040216064,0.3)(0.40402162,0.1)
\psline[linecolor=black, linewidth=0.018, arrowsize=0.05291667cm 2.0,arrowlength=0.8,arrowinset=0.2]{->}(0.0040216064,-0.3)(0.40402162,-0.1)
\psline[linecolor=black, linewidth=0.018, arrowsize=0.05291667cm 2.0,arrowlength=0.8,arrowinset=0.2]{->}(7.0040216,-0.1)(7.4040217,-0.3)
\psline[linecolor=black, linewidth=0.018, arrowsize=0.05291667cm 2.0,arrowlength=0.8,arrowinset=0.2]{->}(7.0040216,0.1)(7.4040217,0.3)
\psline[linecolor=black, linewidth=0.018, arrowsize=0.05291667cm 2.0,arrowlength=0.8,arrowinset=0.2]{->}(5.7040215,0.0)(6.1040215,0.0)
\rput[bl](2.1540215,-0.1){=}
\rput[bl](3.4540217,-0.05){$m$}
\psline[linecolor=black, linewidth=0.018, arrowsize=0.05291667cm 2.0,arrowlength=0.8,arrowinset=0.2]{->}(3.9040215,0.0)(4.304022,0.0)
\psbezier[linecolor=black, linewidth=0.018, arrowsize=0.05291667cm 2.0,arrowlength=0.8,arrowinset=0.2]{->}(2.8040216,0.3)(3.0040216,0.1)(3.1040215,-0.2)(3.3040216,-0.1)
\rput[bl](9.104022,-0.1){$\Delta$}
\psline[linecolor=black, linewidth=0.018, arrowsize=0.05291667cm 2.0,arrowlength=0.8,arrowinset=0.2]{->}(8.504022,0.0)(8.904021,0.0)
\psbezier[linecolor=black, linewidth=0.018, arrowsize=0.05291667cm 2.0,arrowlength=0.8,arrowinset=0.2]{->}(9.504022,0.1)(9.804022,0.1)(9.904021,0.0)(10.004022,-0.3)
\rput[bl](4.3540215,-0.1){,}
\rput[bl](10.204021,-0.1){.}
\rput[bl](7.8540215,-0.1){=}
\psbezier[linecolor=black, linewidth=0.018, arrowsize=0.05291667cm 2.0,arrowlength=0.8,arrowinset=0.2]{->}(2.8040216,-0.3)(3.0040216,-0.1)(3.1040215,0.2)(3.3040216,0.1)
\psbezier[linecolor=black, linewidth=0.018, arrowsize=0.05291667cm 2.0,arrowlength=0.8,arrowinset=0.2]{->}(9.504022,-0.1)(9.804022,-0.1)(9.904021,0.0)(10.004022,0.3)
\end{pspicture}

\end{figure}

%Note that other authors (e.g. \cite{KV: generalized}) define $H^*$ as our $(H^*)^{op}$.

\def\LaX{\La_{\kk}}

\begin{example}
\label{example: exterior algebra} Let $\LaX$ be an exterior algebra on one generator, that is $\LaX\eq \kk[X]/(X^2)$. This is a superalgebra if we let $|X|=1$. It is an involutive Hopf superalgebra if we define
\begin{align*}
\De(X)&=1\ot X+ X\ot 1, & \e(X)&=0, & S(X)&=-X
\end{align*}
as one can readily check. %The Hopf algebra integral $\mu:\LaX\to \kk$ is given by $\mu(1)=0$ and $\mu(X)=1$ so it has degree one. 
\end{example} 

Our main example of an involutive Hopf superalgebra is the following.

\begin{definition}
\label{def: Hopf algebra Hn}
Let $n\in \N_{\geq 0}$. We denote by $H_n$ the superalgebra over $\kk$ generated by elements $X,K,K^{-1}$ subject to the relations
\begin{align*}
KK^{-1}=1 &=K^{-1}K, & XK&=KX, & X^2&=0, & K^n&=1
\end{align*}
and graded (mod 2) by letting $|K|=0$ and $|X|=1$. This is an involutive Hopf superalgebra with
\begin{align*}
\De(X)&=K\ot X+ X\ot 1, & \De(K^{\pm 1})&=K^{\pm 1}\ot K^{\pm 1} \\
\e(X)&=0, & \e(K^{\pm 1})&=1 \\
S(X)&=-K^{-1}X, & S(K^{\pm 1})&=K^{\mp 1}.
\end{align*}
\end{definition}

Equivalently, $H_n$ is the Borel subalgebra of $\Uqgl11$ at a root of unity $q$ of order $n$. We will not use the quantum group $\Uqgl11$ in this paper though, so we refer to \cite{Sartori:Alexander} for its definition.

\def\Ka{\K[\a^{\pm 1}]}

\begin{definition}
\label{def: dual Hopf H}
Let $(H,m,\eta,\De,\e,S)$ be a finite-dimensional Hopf superalgebra. Then its dual $H^*$ (consisting of all linear maps $H\to\kk$) is a Hopf superalgebra with grading $(H^*)_i\eq (H_i)^*$ and structure maps $\De^*,\e^*,m^*,\eta^*,S^*$. Thus, the multiplication and comultiplication are respectively given by

\begin{figure}[H]
\centering
\begin{pspicture}(0,-0.455)(10.0,0.455)
\psline[linecolor=black, linewidth=0.018, arrowsize=0.05291667cm 2.0,arrowlength=0.8,arrowinset=0.2]{->}(0.0,0.045)(0.4,0.045)
\rput[bl](0.65,-0.055){$f\cdot g$}
\rput[bl](1.7,-0.1){=}
\psline[linecolor=black, linewidth=0.018, arrowsize=0.05291667cm 2.0,arrowlength=0.8,arrowinset=0.2]{->}(3.3,-0.055)(3.7,-0.255)
\psline[linecolor=black, linewidth=0.018, arrowsize=0.05291667cm 2.0,arrowlength=0.8,arrowinset=0.2]{->}(3.3,0.145)(3.7,0.345)
\psline[linecolor=black, linewidth=0.018, arrowsize=0.05291667cm 2.0,arrowlength=0.8,arrowinset=0.2]{->}(2.3,0.045)(2.7,0.045)
\rput[bl](2.85,-0.055){$\Delta$}
\rput[bl](3.85,0.145){$f$}
\rput[bl](3.9,-0.405){$g$}
\psline[linecolor=black, linewidth=0.018, arrowsize=0.05291667cm 2.0,arrowlength=0.8,arrowinset=0.2]{->}(5.4,0.345)(5.8,0.145)
\psline[linecolor=black, linewidth=0.018, arrowsize=0.05291667cm 2.0,arrowlength=0.8,arrowinset=0.2]{->}(5.4,-0.255)(5.8,-0.055)
\rput[bl](6.0,-0.155){$\Delta_{H^*}(f)$}
\rput[bl](7.5,-0.1){=}
\psline[linecolor=black, linewidth=0.018, arrowsize=0.05291667cm 2.0,arrowlength=0.8,arrowinset=0.2]{->}(8.1,0.345)(8.5,0.145)
\psline[linecolor=black, linewidth=0.018, arrowsize=0.05291667cm 2.0,arrowlength=0.8,arrowinset=0.2]{->}(8.1,-0.255)(8.5,-0.055)
\rput[bl](8.65,-0.055){$m$}
\psline[linecolor=black, linewidth=0.018, arrowsize=0.05291667cm 2.0,arrowlength=0.8,arrowinset=0.2]{->}(9.1,0.045)(9.5,0.045)
\rput[bl](9.6,-0.155){$f$}
\rput[bl](4.2,-0.1){,}
\rput[bl](9.95,-0.1){.}
\end{pspicture}
\end{figure}
\end{definition}

\begin{definition}
\label{def: group-likes}
Let $H$ be a Hopf superalgebra over a field $\kk$. An element $g\in H$ is said to be {\em group-like} if it satisfies
\begin{figure}[H]
\centering
\begin{pspicture}(0,-0.305)(7.85,0.305)
\psline[linecolor=black, linewidth=0.018, arrowsize=0.05291667cm 2.0,arrowlength=0.8,arrowinset=0.2]{->}(0.35,-0.005)(0.75,-0.005)
\rput[bl](0.91,-0.105){$\Delta$}
\psline[linecolor=black, linewidth=0.018, arrowsize=0.05291667cm 2.0,arrowlength=0.8,arrowinset=0.2]{->}(1.35,0.095)(1.75,0.295)
\psline[linecolor=black, linewidth=0.018, arrowsize=0.05291667cm 2.0,arrowlength=0.8,arrowinset=0.2]{->}(1.35,-0.105)(1.75,-0.305)
\psline[linecolor=black, linewidth=0.018, arrowsize=0.05291667cm 2.0,arrowlength=0.8,arrowinset=0.2]{->}(5.55,-0.005)(5.95,-0.005)
\rput[bl](6.15,-0.065){$\epsilon$}
\rput[bl](3.57,-0.1){,}
\rput[bl](0.0,-0.105){$g$}
\rput[bl](5.2,-0.105){$g$}
\psline[linecolor=black, linewidth=0.018, arrowsize=0.05291667cm 2.0,arrowlength=0.8,arrowinset=0.2]{->}(3.05,0.195)(3.45,0.195)
\psline[linecolor=black, linewidth=0.018, arrowsize=0.05291667cm 2.0,arrowlength=0.8,arrowinset=0.2]{->}(3.05,-0.205)(3.45,-0.205)
\rput[bl](2.7,0.095){$g$}
\rput[bl](2.7,-0.305){$g$}
\rput[bl](2.05,-0.1){=}
\rput[bl](6.65,-0.1){=}
\rput[bl](7.15,-0.105){$\text{id}_{\mathbb{K}}$}
\rput[bl](7.8,-0.1){.}
\end{pspicture}
\end{figure}

We denote the set of group-likes of $H$ by $G(H)$. This is a group with the multiplication induced from $H$. Note that if $H$ is finite dimensional, then $G(H^*)$ consists of the algebra morphisms $H\to\kk$. If $H$ is infinite dimensional, we {\em define} $G(H^*)$ as $\Hom_{alg}(H,\kk)$.%A morphism $g:\kk\to H$ is said to be {\em group-like} if $\e_H\circ g=\id_{\kk}$ and $\De\circ g=g\ot g$. The set of group-likes of $H$ form a group denoted $G(H)$. Dually, we say a morphism $g:H\to \kk$ is group-like (in $H^*$) if $g\circ \eta_H=\id_{\kk}$ and $g\circ m_H=g\ot g$. These also form a group, denoted $G(H^*)$.
\end{definition}

\subsection[Relative integrals and cointegrals]{Relative integrals and cointegrals}\label{section: relative integrals and cointegrals}

We now define relative versions of the Hopf algebra integrals and cointegrals (see \cite[Chapter 10]{Radford:BOOK} for a detailed treatment of the usual non-relative case).  %and discuss the compatibility condition these need to satisfy in order to produce sutured 3-manifold invariants. We give an example of this structure on the Hopf superalgebra $H_n$ of Definition \ref{def: Hopf algebra Hn}. 
However, we do not explore the question of existence and uniqueness of these structures. In what follows we let $H$ be an involutive Hopf superalgebra over a field $\kk$. Since we consider various Hopf superalgebras, we denote the structure tensors of $H$ by $m_H,\eta_H,\De_H,\e_H,S_H$. We do not suppose $H$ is finite dimensional.

\begin{definition}
\label{def: relative integral}
A {\em relative right integral} in $H$ is a tuple $(B,i_B,\pi_B,\mu)$ where $B$ is a Hopf superalgebra, $i_B:B\to H$ and $\pi_B:H\to B$ are degree zero Hopf morphisms and $\mu:H\to B$ is a degree preserving linear map 
with the following properties:

\begin{enumerate}
\item \label{eq: centrality} $\pi_Bi_B=\id_B$ and $i_B(B)$ is a central subalgebra of $H$, that is, 
\begin{figure}[H]
\centering
\begin{pspicture}(0,-0.39374542)(6.016667,0.39374542)
\rput[bl](0.5,0.11374542){$i_B$}
\rput[bl](3.9,0.11374542){$i_B$}
\rput[bl](1.5,-0.18625458){$m_H$}
\rput[bl](2.85,-0.13625458){=}
\psline[linecolor=black, linewidth=0.018, arrowsize=0.05291667cm 2.0,arrowlength=0.8,arrowinset=0.2]{->}(0.0,0.21374542)(0.3,0.21374542)
\psline[linecolor=black, linewidth=0.018, arrowsize=0.05291667cm 2.0,arrowlength=0.8,arrowinset=0.2]{->}(1.0,0.21374542)(1.3,0.0137454225)
\psline[linecolor=black, linewidth=0.018, arrowsize=0.05291667cm 2.0,arrowlength=0.8,arrowinset=0.2]{->}(1.0,-0.38625458)(1.3,-0.18625458)
\psline[linecolor=black, linewidth=0.018, arrowsize=0.05291667cm 2.0,arrowlength=0.8,arrowinset=0.2]{->}(2.2,-0.086254574)(2.5,-0.086254574)
\rput[bl](4.87,-0.23625457){$m^{\text{op}}_H$}
\psline[linecolor=black, linewidth=0.018, arrowsize=0.05291667cm 2.0,arrowlength=0.8,arrowinset=0.2]{->}(5.6,-0.086254574)(5.9,-0.086254574)
\psline[linecolor=black, linewidth=0.018, arrowsize=0.05291667cm 2.0,arrowlength=0.8,arrowinset=0.2]{->}(3.4,0.21374542)(3.7,0.21374542)
\psline[linecolor=black, linewidth=0.018, arrowsize=0.05291667cm 2.0,arrowlength=0.8,arrowinset=0.2]{->}(4.4,0.21374542)(4.7,0.0137454225)
\psline[linecolor=black, linewidth=0.018, arrowsize=0.05291667cm 2.0,arrowlength=0.8,arrowinset=0.2]{->}(4.4,-0.38625458)(4.7,-0.18625458)
\rput[bl](5.9666667,-0.38625458){.}
\end{pspicture}

\end{figure}

\item \label{eq: relative integral B-linear}  $\mu$ is $B$-linear, that is
\begin{figure}[H]
\centering
\begin{pspicture}(0,-0.855)(5.9,0.855)
\psline[linecolor=black, linewidth=0.018, arrowsize=0.05291667cm 2.0,arrowlength=0.8,arrowinset=0.2]{->}(0.0,-0.745)(0.3,-0.745)
\rput[bl](0.54,-0.855){$m_H$}
\rput[bl](4.74,-0.355){$m_B$}
\psline[linecolor=black, linewidth=0.018, arrowsize=0.05291667cm 2.0,arrowlength=0.8,arrowinset=0.2]{<-}(5.0,0.055)(5.0,0.355)
\rput[bl](2.8,-0.085){=}
\psline[linecolor=black, linewidth=0.018, arrowsize=0.05291667cm 2.0,arrowlength=0.8,arrowinset=0.2]{->}(0.8,0.855)(0.8,0.555)
\rput[bl](0.61,0.055){$i_B$}
\psline[linecolor=black, linewidth=0.018, arrowsize=0.05291667cm 2.0,arrowlength=0.8,arrowinset=0.2]{->}(0.8,-0.145)(0.8,-0.445)
\psline[linecolor=black, linewidth=0.018, arrowsize=0.05291667cm 2.0,arrowlength=0.8,arrowinset=0.2]{->}(5.5,-0.245)(5.8,-0.245)
\psline[linecolor=black, linewidth=0.018, arrowsize=0.05291667cm 2.0,arrowlength=0.8,arrowinset=0.2]{->}(3.3,-0.245)(3.6,-0.245)
\rput[bl](3.81,-0.345){$\mu$}
\psline[linecolor=black, linewidth=0.018, arrowsize=0.05291667cm 2.0,arrowlength=0.8,arrowinset=0.2]{->}(4.2,-0.245)(4.5,-0.245)
\psline[linecolor=black, linewidth=0.018, arrowsize=0.05291667cm 2.0,arrowlength=0.8,arrowinset=0.2]{->}(1.3,-0.745)(1.6,-0.745)
\rput[bl](1.81,-0.845){$\mu$}
\psline[linecolor=black, linewidth=0.018, arrowsize=0.05291667cm 2.0,arrowlength=0.8,arrowinset=0.2]{->}(2.2,-0.745)(2.5,-0.745)
\rput[bl](5.85,-0.745){.}
\end{pspicture}

\end{figure}

\item \label{eq: relative integral relation} The map $\mu$ satisfies \begin{figure}[H]
\centering
\begin{pspicture}(0,-0.89)(5.9,0.89)
\psline[linecolor=black, linewidth=0.018, arrowsize=0.05291667cm 2.0,arrowlength=0.8,arrowinset=0.2]{->}(1.3,0.21)(1.6,0.21)
\rput[bl](0.54,0.1){$\Delta_H$}
\psline[linecolor=black, linewidth=0.018, arrowsize=0.05291667cm 2.0,arrowlength=0.8,arrowinset=0.2]{->}(0.0,0.21)(0.3,0.21)
\psline[linecolor=black, linewidth=0.018, arrowsize=0.05291667cm 2.0,arrowlength=0.8,arrowinset=0.2]{->}(0.7,-0.09)(0.7,-0.39)
\rput[bl](2.8,-0.03){=}
\rput[bl](1.81,0.11){$\mu$}
\psline[linecolor=black, linewidth=0.018, arrowsize=0.05291667cm 2.0,arrowlength=0.8,arrowinset=0.2]{->}(2.2,0.21)(2.5,0.21)
\rput[bl](4.74,0.6){$\Delta_B$}
\psline[linecolor=black, linewidth=0.018, arrowsize=0.05291667cm 2.0,arrowlength=0.8,arrowinset=0.2]{->}(5.0,0.41)(5.0,0.11)
\psline[linecolor=black, linewidth=0.018, arrowsize=0.05291667cm 2.0,arrowlength=0.8,arrowinset=0.2]{->}(5.5,0.71)(5.8,0.71)
\rput[bl](4.81,-0.39){$i_B$}
\psline[linecolor=black, linewidth=0.018, arrowsize=0.05291667cm 2.0,arrowlength=0.8,arrowinset=0.2]{->}(5.0,-0.59)(5.0,-0.89)
\psline[linecolor=black, linewidth=0.018, arrowsize=0.05291667cm 2.0,arrowlength=0.8,arrowinset=0.2]{->}(3.3,0.71)(3.6,0.71)
\rput[bl](3.81,0.61){$\mu$}
\psline[linecolor=black, linewidth=0.018, arrowsize=0.05291667cm 2.0,arrowlength=0.8,arrowinset=0.2]{->}(4.2,0.71)(4.5,0.71)
\rput[bl](5.85,-0.89){.}
\end{pspicture}
\end{figure}

\end{enumerate}

%\begin{equation}\label{eq: integral B-linearity}\mu\circ m_H\circ (i_B\ot \id_H)=(b^*\ot m_B)\circ (\De_B\ot \mu)\end{equation}for some central group-like $b^*\in G(B^*)$ and\begin{equation}\label{eq: integral handlesliding}(\mu\ot m_H)\circ (\De_H\ot\id_H)=(\id_B\ot m_H)\circ (\id_B\ot i_B\ot \id_H)\circ (\De_B\mu\ot \id_H).\end{equation}for all $h\in H$. \end{definition}
%\begin{equation}\label{eq: B linearity}\mu(i_B(b)h)=b\mu(h)\end{equation}and

\end{definition}

Note that for $B=\kk$ and $i_B=\eta_H, \pi_B=\e_H$ this definition reduces to the usual definition of Hopf algebra (right) integral. Since $\pi_Bi_B=\id_B$ we may as well consider $B$ as a central Hopf subalgebra of $H.$

\begin{definition}
\label{def: relative cointegral}
A {\em relative right cointegral} in $H$ is a tuple $(A,\pi_A,i_A,\iota)$ where $A$ is a Hopf superalgebra, $\pi_A:H\to A$ and $i_A:A\to H$ are degree zero Hopf morphisms and $\iota:A\to H$ is a degree preserving linear map with the following properties:
\begin{enumerate}
\item \label{eq: cocentrality} $\pi_Ai_A=\id_A$ and $\pi_A$ is cocentral, that is 
\begin{figure}[H]
\centering 
\begin{pspicture}(0,-0.34388107)(6.2166667,0.34388107)
\psline[linecolor=black, linewidth=0.018, arrowsize=0.05291667cm 2.0,arrowlength=0.8,arrowinset=0.2]{->}(0.0,-0.056118928)(0.3,-0.056118928)
\psline[linecolor=black, linewidth=0.018, arrowsize=0.05291667cm 2.0,arrowlength=0.8,arrowinset=0.2]{->}(1.2,-0.15611893)(1.5,-0.35611892)
\psline[linecolor=black, linewidth=0.018, arrowsize=0.05291667cm 2.0,arrowlength=0.8,arrowinset=0.2]{->}(1.2,0.043881074)(1.5,0.24388108)
\rput[bl](0.5,-0.15611893){$\Delta_H$}
\rput[bl](1.7,0.14388107){$\pi_A$}
\psline[linecolor=black, linewidth=0.018, arrowsize=0.05291667cm 2.0,arrowlength=0.8,arrowinset=0.2]{->}(2.3,0.24388108)(2.6,0.24388108)
\rput[bl](2.95,-0.106118925){=}
\psline[linecolor=black, linewidth=0.018, arrowsize=0.05291667cm 2.0,arrowlength=0.8,arrowinset=0.2]{->}(3.5,-0.056118928)(3.8,-0.056118928)
\psline[linecolor=black, linewidth=0.018, arrowsize=0.05291667cm 2.0,arrowlength=0.8,arrowinset=0.2]{->}(4.7,-0.15611893)(5.0,-0.35611892)
\psline[linecolor=black, linewidth=0.018, arrowsize=0.05291667cm 2.0,arrowlength=0.8,arrowinset=0.2]{->}(4.7,0.043881074)(5.0,0.24388108)
\rput[bl](4.0,-0.20611893){$\Delta^{\text{op}}_H$}
\rput[bl](5.2,0.14388107){$\pi_A$}
\psline[linecolor=black, linewidth=0.018, arrowsize=0.05291667cm 2.0,arrowlength=0.8,arrowinset=0.2]{->}(5.8,0.24388108)(6.1,0.24388108)
\rput[bl](6.1666665,-0.3227856){.}
\end{pspicture}
\end{figure}

\item \label{eq: A colinearity} $\iota$ is $A$-colinear, that is \begin{figure}[H]
\centering
\begin{pspicture}(0,-0.8372166)(5.4,0.8372166)
\psline[linecolor=black, linewidth=0.018, arrowsize=0.05291667cm 2.0,arrowlength=0.8,arrowinset=0.2]{->}(1.8,-0.72721666)(2.1,-0.72721666)
\rput[bl](1.14,-0.8372166){$\Delta_H$}
\rput[bl](3.54,-0.43721664){$\Delta_A$}
\psline[linecolor=black, linewidth=0.018, arrowsize=0.05291667cm 2.0,arrowlength=0.8,arrowinset=0.2]{->}(3.1,-0.32721666)(3.4,-0.32721666)
\psline[linecolor=black, linewidth=0.018, arrowsize=0.05291667cm 2.0,arrowlength=0.8,arrowinset=0.2]{->}(3.7,0.07278336)(3.7,0.37278336)
\rput[bl](2.5,-0.16721664){=}
\psline[linecolor=black, linewidth=0.018, arrowsize=0.05291667cm 2.0,arrowlength=0.8,arrowinset=0.2]{->}(0.0,-0.72721666)(0.3,-0.72721666)
\rput[bl](0.46,-0.79721665){$\iota$}
\psline[linecolor=black, linewidth=0.018, arrowsize=0.05291667cm 2.0,arrowlength=0.8,arrowinset=0.2]{->}(0.7,-0.72721666)(1.0,-0.72721666)
\psline[linecolor=black, linewidth=0.018, arrowsize=0.05291667cm 2.0,arrowlength=0.8,arrowinset=0.2]{->}(4.2,-0.32721666)(4.5,-0.32721666)
\rput[bl](4.66,-0.39721665){$\iota$}
\psline[linecolor=black, linewidth=0.018, arrowsize=0.05291667cm 2.0,arrowlength=0.8,arrowinset=0.2]{->}(4.9,-0.32721666)(5.2,-0.32721666)
\psline[linecolor=black, linewidth=0.018, arrowsize=0.05291667cm 2.0,arrowlength=0.8,arrowinset=0.2]{<-}(1.3,-0.027216645)(1.3,-0.32721666)
\psline[linecolor=black, linewidth=0.018, arrowsize=0.05291667cm 2.0,arrowlength=0.8,arrowinset=0.2]{<-}(1.3,0.87278336)(1.3,0.57278335)
\rput[bl](1.11,0.17278336){$\pi_A$}
\rput[bl](5.35,-0.67721665){.}
\end{pspicture}
\end{figure}

%The map $\iota$ is $A$-colinear, that is \begin{equation*}(\pi_A\ot \id_H)\circ \De_H\circ \iota=(\id_A\ot \iota)\circ \De_A.\end{equation*}

\item \label{eq: relative cointegral relation} The map $\iota$ satisfies \begin{figure}[H]
\centering
\begin{pspicture}(0,-0.82)(5.4,0.82)
\psline[linecolor=black, linewidth=0.018, arrowsize=0.05291667cm 2.0,arrowlength=0.8,arrowinset=0.2]{->}(1.8,0.33)(2.1,0.33)
\rput[bl](1.14,0.22){$m_H$}
\psline[linecolor=black, linewidth=0.018, arrowsize=0.05291667cm 2.0,arrowlength=0.8,arrowinset=0.2]{->}(0.0,0.33)(0.3,0.33)
\psline[linecolor=black, linewidth=0.018, arrowsize=0.05291667cm 2.0,arrowlength=0.8,arrowinset=0.2]{<-}(1.4,0.03)(1.4,-0.27)
\rput[bl](2.5,0.09){=}
\rput[bl](0.46,0.26){$\iota$}
\psline[linecolor=black, linewidth=0.018, arrowsize=0.05291667cm 2.0,arrowlength=0.8,arrowinset=0.2]{->}(0.7,0.33)(1.0,0.33)
\rput[bl](3.54,0.62){$m_A$}
\psline[linecolor=black, linewidth=0.018, arrowsize=0.05291667cm 2.0,arrowlength=0.8,arrowinset=0.2]{->}(3.1,0.73)(3.4,0.73)
\psline[linecolor=black, linewidth=0.018, arrowsize=0.05291667cm 2.0,arrowlength=0.8,arrowinset=0.2]{->}(4.2,0.73)(4.5,0.73)
\rput[bl](4.66,0.66){$\iota$}
\psline[linecolor=black, linewidth=0.018, arrowsize=0.05291667cm 2.0,arrowlength=0.8,arrowinset=0.2]{->}(4.9,0.73)(5.2,0.73)
\psline[linecolor=black, linewidth=0.018, arrowsize=0.05291667cm 2.0,arrowlength=0.8,arrowinset=0.2]{<-}(3.7,-0.47)(3.7,-0.77)
\psline[linecolor=black, linewidth=0.018, arrowsize=0.05291667cm 2.0,arrowlength=0.8,arrowinset=0.2]{<-}(3.7,0.43)(3.7,0.13)
\rput[bl](3.51,-0.27){$\pi_A$}
\rput[bl](5.35,-0.82){.}
\end{pspicture}

\end{figure}%One has\begin{equation*}m_H\circ (\iota\ot\id_H)=\iota\circ m_A\circ (\id_A\ot \pi_A).\end{equation*}
%that is \iota(a)h=\iota(a\pi_A(h)).

\end{enumerate}

% \begin{equation}\label{eq: cointegral A-comod}(\pi_A\ot \id_H)\circ \De_H\circ \iota=(m_A\ot\pi_A)\circ (a\ot \De_A)\end{equation}for some central $a\in G(A)$ and \begin{equation}\label{eq: cointegral handlesliding}(m_H\ot \id_H)\circ (\iota\ot\De_H)=(\iota m_A\ot \id_H(\id_A\ot \pi_A\ot \id_H)\circ (\id_A\ot \De_H).\end{equation}\end{definition}

\end{definition}

Since $\pi_Ai_A=\id_A$ we also think of $A$ as a Hopf subalgebra of $H$. The above definition also reduces to the usual one when $A=\kk$ and $\pi_A=\e_H, i_A=\eta_H$. This is precisely the dual notion of relative right integral when $H$ is finite dimensional: if $(B,i_B,\pi_B,\mu)$ is a relative right integral in $H$ then $(B^*,i_B^*,\pi_B^*,\mu^*)$ is a relative right cointegral in $H^*$.% (see Definition \ref{def: dual Hopf H} for our conventions for $H^*$). %Note that by evaluating at $a=1$ and $h=a\in A$ in condition (\ref{eq: relative cointegral relation}) we get $\iota(a)=\iota(1)a$.\medskip

%We will depict the morphisms of Definition \ref{def: relative cointegral} as follows:Here a red strand denotes $A$ and the black strand stands for $H$. \medskip

\begin{example}
\label{example: Hopf algebra Hn finite dim quotient}
Let $n\in \N_{\geq 1}$. Consider the Hopf superalgebra $H_n$ of Definition \ref{def: Hopf algebra Hn} and let $B$ be the subalgebra generated by $K,K^{-1}$. Then $H_n=\LaX\ot B$ as superalgebras. Define $i_B,\pi_B$ and $\mu$ by
\begin{align*}
 i_B&=\eta_{\LaX}\ot\id_B, & \pi_B&=\e_{\LaX}\ot\id_B, & \mu&=\mu_{\LaX}\ot \id_B.
\end{align*}
Here we denote by $\eta_{\LaX},\e_{\LaX}$ and $\mu_{\LaX}$ respectively the unit, counit and integral of the Hopf superalgebra $\LaX$ (see Example \ref{example: exterior algebra}), so $\mu_{\LaX}$ is given by $\mu_{\LaX}(1)=0$ and $\mu_{\LaX}(X)=1$. It is straightforward to check that these maps define a relative integral over $H_n$. %It is easy to see that both $i_B,\pi_B$ are Hopf algebra morphisms and that that $\mu$ is $B$-linear. To check (\ref{eq: relative integral relation}) note that both sides vanish when $h=K^i$. For $h=K^iX$ one has\begin{align*} \De(K^iX)&=K^{i+1}\ot K^iX+K^iX\ot K^i, \end{align*}hence\begin{equation*}\begin{split}\wtmu(K^{i+1})\ot K^iX+\wtmu(K^iX)\ot K^i&= \wtmu(K^iX)\ot K^i\\&= (K^i\mu_{\La_X}(X))\ot K^i\\&= \De_B(K^i \mu_{\La_X}(X))\\&= \De_B \wtmu(K^iX)\end{split}\end{equation*}so that $\mu$ is a right relative integral. 
As a relative cointegral we take the usual (two-sided) cointegral of $H_n$, that is, we let $A=\kk,\pi_A=\e_H, i_A=\eta_H$ and
\begin{align*}
\iota=\frac{1}{n}(1+K+\dots+K^{n-1})X
\end{align*}
where we identify $\iota$ with its image $\iota(1)$ since $A=\kk$. %Note that $\iota$ is just the usual cointegral of $H_n$ (which is finite dimensional since we supposed $n>0$). 

\end{example}

\begin{remark}
Though the above finite dimensional example will be the only one treated in the present paper, our definition of relative cointegral allows $H$ to be infinite dimensional. A general class of (possibly) infinite dimensional examples can be found via semidirect products, that is, $H=\kk[\Aut(J)]\ltimes J$ for some Hopf algebra $J$, with a relative right cointegral built over the subalgebra $A=\kk[\Aut(J)]$ and given by $\iota(\a)\eq c_r\a$ where $c_r\in J$ is a right cointegral and $\a\in \Aut(J)$. This example will be studied at length in a future paper. When $J=\LaX$, this is dual to the case of $H_n$ at $n=0$.
\end{remark}
\medskip

We now derive a simple but fundamental consequence of our definition. We state it for the relative cointegral, a dual statement holds for relative integrals. %For any Hopf superalgebra $(H,m_H,\eta,\De_H,\e, S_H)$ we denote $T_H\eq (m_H\ot\id_H)\circ(\id_H\ot\De_H)$.

\begin{proposition}
\label{prop: hsliding property for cointegral arc-arc, arc-curve, curve-curve}
Let $(A,\pi_A,i_A,\iota)$ be a relative right cointegral in $H$. If 
$(f_1,f_2)=(\iota,\iota), (\iota,i_A)$ or $(i_A,i_A)$ then one has
\begin{figure}[H]
\centering
\begin{pspicture}(0,-0.655)(6.0,0.655)
\psline[linecolor=black, linewidth=0.018, arrowsize=0.05291667cm 2.0,arrowlength=0.8,arrowinset=0.2]{->}(2.1,0.485)(2.4,0.485)
\rput[bl](1.34,0.375){$m_H$}
\psline[linecolor=black, linewidth=0.018, arrowsize=0.05291667cm 2.0,arrowlength=0.8,arrowinset=0.2]{->}(0.0,0.485)(0.3,0.485)
\psline[linecolor=black, linewidth=0.018, arrowsize=0.05291667cm 2.0,arrowlength=0.8,arrowinset=0.2]{<-}(1.6,0.185)(1.6,-0.115)
\rput[bl](0.46,0.345){$f_1$}
\psline[linecolor=black, linewidth=0.018, arrowsize=0.05291667cm 2.0,arrowlength=0.8,arrowinset=0.2]{->}(0.9,0.485)(1.2,0.485)
\psline[linecolor=black, linewidth=0.018, arrowsize=0.05291667cm 2.0,arrowlength=0.8,arrowinset=0.2]{->}(2.1,-0.515)(2.4,-0.515)
\rput[bl](1.43,-0.625){$\Delta_H$}
\psline[linecolor=black, linewidth=0.018, arrowsize=0.05291667cm 2.0,arrowlength=0.8,arrowinset=0.2]{->}(0.0,-0.515)(0.3,-0.515)
\rput[bl](0.46,-0.655){$f_2$}
\psline[linecolor=black, linewidth=0.018, arrowsize=0.05291667cm 2.0,arrowlength=0.8,arrowinset=0.2]{->}(0.9,-0.515)(1.2,-0.515)
\rput[bl](2.8,-0.1){=}
\rput[bl](3.84,0.375){$m_A$}
\psline[linecolor=black, linewidth=0.018, arrowsize=0.05291667cm 2.0,arrowlength=0.8,arrowinset=0.2]{<-}(4.1,0.185)(4.1,-0.115)
\psline[linecolor=black, linewidth=0.018, arrowsize=0.05291667cm 2.0,arrowlength=0.8,arrowinset=0.2]{->}(3.4,0.485)(3.7,0.485)
\rput[bl](3.93,-0.615){$\Delta_A$}
\psline[linecolor=black, linewidth=0.018, arrowsize=0.05291667cm 2.0,arrowlength=0.8,arrowinset=0.2]{->}(3.4,-0.515)(3.7,-0.515)
\psline[linecolor=black, linewidth=0.018, arrowsize=0.05291667cm 2.0,arrowlength=0.8,arrowinset=0.2]{->}(4.6,0.485)(4.9,0.485)
\rput[bl](5.06,0.345){$f_1$}
\psline[linecolor=black, linewidth=0.018, arrowsize=0.05291667cm 2.0,arrowlength=0.8,arrowinset=0.2]{->}(5.5,0.485)(5.8,0.485)
\psline[linecolor=black, linewidth=0.018, arrowsize=0.05291667cm 2.0,arrowlength=0.8,arrowinset=0.2]{->}(4.6,-0.515)(4.9,-0.515)
\rput[bl](5.06,-0.655){$f_2$}
\psline[linecolor=black, linewidth=0.018, arrowsize=0.05291667cm 2.0,arrowlength=0.8,arrowinset=0.2]{->}(5.5,-0.515)(5.8,-0.515)
\rput[bl](5.95,-0.1){.}
\end{pspicture}
\end{figure}

\end{proposition}
\begin{proof}
Consider the case when $(f_1,f_2)=(\iota, i_A)$. We have
\begin{figure}[H]
\centering
\begin{pspicture}(0,-0.85)(12.326559,0.85)
\psline[linecolor=black, linewidth=0.017, arrowsize=0.05291667cm 2.0,arrowlength=1.2,arrowinset=0.2]{->}(4.7,0.76)(5.0,0.76)
\rput[bl](4.04,0.65){$m_A$}
\psline[linecolor=black, linewidth=0.017, arrowsize=0.05291667cm 2.0,arrowlength=1.2,arrowinset=0.2]{<-}(4.3,0.56)(4.3,0.26)
\psline[linecolor=black, linewidth=0.017, arrowsize=0.05291667cm 2.0,arrowlength=1.2,arrowinset=0.2]{->}(3.6,0.76)(3.9,0.76)
\rput[bl](5.16,0.69){$\iota$}
\psline[linecolor=black, linewidth=0.017, arrowsize=0.05291667cm 2.0,arrowlength=1.2,arrowinset=0.2]{->}(5.4,0.76)(5.7,0.76)
\rput[bl](4.11,-0.04){$\pi_A$}
\psline[linecolor=black, linewidth=0.017, arrowsize=0.05291667cm 2.0,arrowlength=1.2,arrowinset=0.2]{<-}(4.3,-0.14)(4.3,-0.44)
\psline[linecolor=black, linewidth=0.017, arrowsize=0.05291667cm 2.0,arrowlength=1.2,arrowinset=0.2]{->}(1.8,0.46)(2.1,0.46)
\rput[bl](1.14,0.35){$m_H$}
\psline[linecolor=black, linewidth=0.017, arrowsize=0.05291667cm 2.0,arrowlength=1.2,arrowinset=0.2]{->}(0.0,0.46)(0.3,0.46)
\psline[linecolor=black, linewidth=0.017, arrowsize=0.05291667cm 2.0,arrowlength=1.2,arrowinset=0.2]{<-}(1.4,0.26)(1.4,-0.04)
\rput[bl](0.46,0.39){$\iota$}
\psline[linecolor=black, linewidth=0.017, arrowsize=0.05291667cm 2.0,arrowlength=1.2,arrowinset=0.2]{->}(0.7,0.46)(1.0,0.46)
\psline[linecolor=black, linewidth=0.017, arrowsize=0.05291667cm 2.0,arrowlength=1.2,arrowinset=0.2]{->}(1.8,-0.34)(2.1,-0.34)
\rput[bl](1.24,-0.45){$\Delta_H$}
\psline[linecolor=black, linewidth=0.017, arrowsize=0.05291667cm 2.0,arrowlength=1.2,arrowinset=0.2]{->}(0.0,-0.34)(0.3,-0.34)
\rput[bl](0.36,-0.41){$i_A$}
\psline[linecolor=black, linewidth=0.017, arrowsize=0.05291667cm 2.0,arrowlength=1.2,arrowinset=0.2]{->}(0.775,-0.34)(1.1,-0.34)
\psline[linecolor=black, linewidth=0.017, arrowsize=0.05291667cm 2.0,arrowlength=1.2,arrowinset=0.2]{->}(4.8,-0.74)(5.1,-0.74)
\rput[bl](4.14,-0.85){$\Delta_H$}
\psline[linecolor=black, linewidth=0.017, arrowsize=0.05291667cm 2.0,arrowlength=1.2,arrowinset=0.2]{->}(2.9,-0.74)(3.2,-0.74)
\rput[bl](3.26,-0.81){$i_A$}
\psline[linecolor=black, linewidth=0.017, arrowsize=0.05291667cm 2.0,arrowlength=1.2,arrowinset=0.2]{->}(3.7,-0.74)(4.0,-0.74)
\rput[bl](2.4,-0.14){=}
\psline[linecolor=black, linewidth=0.017, arrowsize=0.05291667cm 2.0,arrowlength=1.2,arrowinset=0.2]{->}(8.6,0.76)(8.9,0.76)
\rput[bl](7.94,0.65){$m_A$}
\psline[linecolor=black, linewidth=0.017, arrowsize=0.05291667cm 2.0,arrowlength=1.2,arrowinset=0.2]{<-}(8.2,0.56)(8.2,0.26)
\psline[linecolor=black, linewidth=0.017, arrowsize=0.05291667cm 2.0,arrowlength=1.2,arrowinset=0.2]{->}(7.5,0.76)(7.8,0.76)
\rput[bl](9.06,0.69){$\iota$}
\psline[linecolor=black, linewidth=0.017, arrowsize=0.05291667cm 2.0,arrowlength=1.2,arrowinset=0.2]{->}(9.3,0.76)(9.6,0.76)
\rput[bl](8.01,-0.04){$\pi_A$}
\psline[linecolor=black, linewidth=0.017, arrowsize=0.05291667cm 2.0,arrowlength=1.2,arrowinset=0.2]{->}(7.5,0.06)(7.8,0.06)
\rput[bl](7.06,-0.01){$i_A$}
\psline[linecolor=black, linewidth=0.017, arrowsize=0.05291667cm 2.0,arrowlength=1.2,arrowinset=0.2]{<-}(7.1,-0.14)(7.1,-0.44)
\psline[linecolor=black, linewidth=0.017, arrowsize=0.05291667cm 2.0,arrowlength=1.2,arrowinset=0.2]{->}(7.6,-0.74)(7.9,-0.74)
\rput[bl](6.94,-0.85){$\Delta_A$}
\psline[linecolor=black, linewidth=0.017, arrowsize=0.05291667cm 2.0,arrowlength=1.2,arrowinset=0.2]{->}(6.5,-0.74)(6.8,-0.74)
\rput[bl](7.96,-0.81){$i_A$}
\psline[linecolor=black, linewidth=0.017, arrowsize=0.05291667cm 2.0,arrowlength=1.2,arrowinset=0.2]{->}(8.4,-0.74)(8.7,-0.74)
\psline[linecolor=black, linewidth=0.017, arrowsize=0.05291667cm 2.0,arrowlength=1.2,arrowinset=0.2]{->}(11.4,0.36)(11.7,0.36)
\rput[bl](10.74,0.25){$m_A$}
\psline[linecolor=black, linewidth=0.017, arrowsize=0.05291667cm 2.0,arrowlength=1.2,arrowinset=0.2]{<-}(11.0,0.16)(11.0,-0.14)
\psline[linecolor=black, linewidth=0.017, arrowsize=0.05291667cm 2.0,arrowlength=1.2,arrowinset=0.2]{->}(10.4,0.36)(10.7,0.36)
\rput[bl](11.86,0.29){$\iota$}
\psline[linecolor=black, linewidth=0.017, arrowsize=0.05291667cm 2.0,arrowlength=1.2,arrowinset=0.2]{->}(12.1,0.36)(12.4,0.36)
\psline[linecolor=black, linewidth=0.017, arrowsize=0.05291667cm 2.0,arrowlength=1.2,arrowinset=0.2]{->}(11.4,-0.44)(11.7,-0.44)
\rput[bl](10.84,-0.55){$\Delta_A$}
\psline[linecolor=black, linewidth=0.017, arrowsize=0.05291667cm 2.0,arrowlength=1.2,arrowinset=0.2]{->}(10.4,-0.44)(10.7,-0.44)
\rput[bl](11.76,-0.51){$i_A$}
\psline[linecolor=black, linewidth=0.017, arrowsize=0.05291667cm 2.0,arrowlength=1.2,arrowinset=0.2]{->}(12.1,-0.44)(12.4,-0.44)
\rput[bl](6.0,-0.14){=}
\rput[bl](9.9,-0.14){=}
\rput[bl](12.75,-0.14){.}
\end{pspicture}

\end{figure}
\noindent Here the first equality follows from condition (\ref{eq: relative cointegral relation}) for $\iota$,  the second follows from $i_A$ being a coalgebra morphism and the last follows from $\pi_Ai_A=\id_A$. The case $(f_1,f_2)=(i_A,i_A)$ is proved in a similar way, using that $i_A$ is both an algebra and a coalgebra morphism. The case $(f_1,f_2)=(\iota,\iota)$ follows by combining conditions (\ref{eq: A colinearity}) and (\ref{eq: relative cointegral relation}) in Definition \ref{def: relative cointegral}.

\end{proof}

\begin{remark}
\label{remark: three handlesliding moves algebraically}
The three equalities of the above proposition correspond to the three handlesliding moves of Proposition \ref{prop: extended RS thm}. More precisely, the cases $(\iota,\iota),(\iota,i_A),(i_A,i_A)$ of the above proposition will respectively imply invariance of $I^{\rho}_H(M,\c,\ss,\o)$ under curve-curve, arc-curve and arc-arc handlesliding. This justifies our use of the Hopf morphisms $\pi_B$ and $i_A$ in the definitions of relative integral and cointegral.
\end{remark}

\subsection{The compatibility condition}
\label{subs: The compatibility condition}  Let $H$ be an involutive Hopf superalgebra over $\kk$ together with a relative right integral $(B,i_B,\pi_B,\mu)$ and relative right cointegral $(A,\pi_A,i_A,\iota)$. 

\begin{notation}
\label{notation: multiplication by b and cop by astar}
For any $b\in G(B)$ we denote by $m_b:H\to H$ left multiplication by $b$, that is,
\begin{figure}[H]
 \centering 
\begin{pspicture}(0,-0.39374545)(5.1,0.39374545)
\psline[linecolor=black, linewidth=0.018, arrowsize=0.05291667cm 2.0,arrowlength=0.8,arrowinset=0.2]{->}(0.0,-0.08625454)(0.3,-0.08625454)
\rput[bl](0.5,-0.18625455){$m_b$}
\psline[linecolor=black, linewidth=0.018, arrowsize=0.05291667cm 2.0,arrowlength=0.8,arrowinset=0.2]{->}(1.1,-0.08625454)(1.4,-0.08625454)
\rput[bl](1.7,-0.13625453){=}
\psline[linecolor=black, linewidth=0.018, arrowsize=0.05291667cm 2.0,arrowlength=0.8,arrowinset=0.2]{->}(3.6,-0.38625455)(3.9,-0.18625455)
\psline[linecolor=black, linewidth=0.018, arrowsize=0.05291667cm 2.0,arrowlength=0.8,arrowinset=0.2]{->}(3.6,0.21374546)(3.9,0.013745461)
\rput[bl](4.05,-0.18625455){$m_H$}
\psline[linecolor=black, linewidth=0.018, arrowsize=0.05291667cm 2.0,arrowlength=0.8,arrowinset=0.2]{->}(4.7,-0.08625454)(5.0,-0.08625454)
\rput[bl](2.25,0.11374546){$b$}
\rput[bl](3.1,0.11374546){$i_B$}
\psline[linecolor=black, linewidth=0.018, arrowsize=0.05291667cm 2.0,arrowlength=0.8,arrowinset=0.2]{->}(2.6,0.21374546)(2.9,0.21374546)
\rput[bl](5.05,-0.38625455){.}
\end{pspicture}
 \end{figure}
 
Similarly, if $a^*\in G(A^*) (=\Hom_{alg}(A,\kk))$, we denote by $\De_{a^*}:H\to H$ the map defined by
\begin{figure}[H]
\centering
\begin{pspicture}(0,-0.363881)(5.3,0.363881)
\rput[bl](0.45,-0.176119){$\Delta_{a^*}$}
\psline[linecolor=black, linewidth=0.018, arrowsize=0.05291667cm 2.0,arrowlength=0.8,arrowinset=0.2]{->}(0.0,-0.076119006)(0.3,-0.076119006)
\psline[linecolor=black, linewidth=0.018, arrowsize=0.05291667cm 2.0,arrowlength=0.8,arrowinset=0.2]{->}(1.2,-0.076119006)(1.5,-0.076119006)
\rput[bl](1.8,-0.126119){=}
\psline[linecolor=black, linewidth=0.018, arrowsize=0.05291667cm 2.0,arrowlength=0.8,arrowinset=0.2]{->}(3.4,0.023880996)(3.7,0.22388099)
\psline[linecolor=black, linewidth=0.018, arrowsize=0.05291667cm 2.0,arrowlength=0.8,arrowinset=0.2]{->}(3.4,-0.176119)(3.7,-0.37611902)
\rput[bl](2.75,-0.226119){$\Delta_H$}
\psline[linecolor=black, linewidth=0.018, arrowsize=0.05291667cm 2.0,arrowlength=0.8,arrowinset=0.2]{->}(2.3,-0.076119006)(2.6,-0.076119006)
\rput[bl](3.85,0.123881){$\pi_A$}
\psline[linecolor=black, linewidth=0.018, arrowsize=0.05291667cm 2.0,arrowlength=0.8,arrowinset=0.2]{->}(4.4,0.22388099)(4.711111,0.22388099)
\rput[bl](4.927778,0.123881){$a^*$}
\rput[bl](5.25,-0.326119){.}
\end{pspicture}
\end{figure}

Note that by centrality of $B$ in $H$, it doesn't matter whether we use left or right multiplication to define $m_b$ (and similarly for $\De_{a^*}$ by cocentrality of $\pi_A$).
\end{notation}

\begin{definition}
\label{def: compatible integral cointegral}
We say $\mu$ and $\iota$ are {\em compatible} if there exist group-likes $b\in G(B)$ and $a^*\in G(A^*)$ satisfying the following properties:\\

\begin{enumerate}
\item\hspace{3cm}
\begin{pspicture}(0,-0.195)(7.7,0.195)
\psline[linecolor=black, linewidth=0.018, arrowsize=0.05291667cm 2.0,arrowlength=0.8,arrowinset=0.2]{->}(0.0,0.005)(0.3,0.005)
\rput[bl](0.5,-0.095){$m_b$}
\rput[bl](1.55,-0.095){$\mu$}
\rput[bl](4.75,-0.095){$S_H$}
\rput[bl](6.7,-0.095){$S_B$}
\psline[linecolor=black, linewidth=0.018, arrowsize=0.05291667cm 2.0,arrowlength=0.8,arrowinset=0.2]{->}(6.2,0.005)(6.5,0.005)
\psline[linecolor=black, linewidth=0.018, arrowsize=0.05291667cm 2.0,arrowlength=0.8,arrowinset=0.2]{->}(7.3,0.005)(7.6,0.005)
\rput[bl](5.85,-0.095){$\mu$}
\psline[linecolor=black, linewidth=0.018, arrowsize=0.05291667cm 2.0,arrowlength=0.8,arrowinset=0.2]{->}(5.4,0.005)(5.7,0.005)
\psline[linecolor=black, linewidth=0.018, arrowsize=0.05291667cm 2.0,arrowlength=0.8,arrowinset=0.2]{->}(4.3,0.005)(4.6,0.005)
\rput[bl](3.0,-0.195){$(-1)^{|\mu|}$}
\psline[linecolor=black, linewidth=0.018, arrowsize=0.05291667cm 2.0,arrowlength=0.8,arrowinset=0.2]{->}(1.9,0.005)(2.2,0.005)
\rput[bl](2.5,-0.045){=}
\rput[bl](7.65,-0.145){,}
\psline[linecolor=black, linewidth=0.018, arrowsize=0.05291667cm 2.0,arrowlength=0.8,arrowinset=0.2]{->}(1.1,0.005)(1.4,0.005)
 \label{eq: integral inv of orientation} 
\end{pspicture}

\vspace{1cm}

\item \hspace{3cm}
\begin{pspicture}(0,-0.195)(7.8,0.195)
\psline[linecolor=black, linewidth=0.018, arrowsize=0.05291667cm 2.0,arrowlength=0.8,arrowinset=0.2]{->}(0.0,0.005)(0.3,0.005)
\rput[bl](0.46,-0.095){$\iota$}
\rput[bl](1.25,-0.095){$\Delta_{a^*}$}
\rput[bl](6.75,-0.095){$S_H$}
\rput[bl](4.9,-0.095){$S_A$}
\psline[linecolor=black, linewidth=0.018, arrowsize=0.05291667cm 2.0,arrowlength=0.8,arrowinset=0.2]{->}(2.0,0.005)(2.3,0.005)
\psline[linecolor=black, linewidth=0.018, arrowsize=0.05291667cm 2.0,arrowlength=0.8,arrowinset=0.2]{->}(6.3,0.005)(6.6,0.005)
\psline[linecolor=black, linewidth=0.018, arrowsize=0.05291667cm 2.0,arrowlength=0.8,arrowinset=0.2]{->}(7.4,0.005)(7.7,0.005)
\psline[linecolor=black, linewidth=0.018, arrowsize=0.05291667cm 2.0,arrowlength=0.8,arrowinset=0.2]{->}(5.5,0.005)(5.8,0.005)
\psline[linecolor=black, linewidth=0.018, arrowsize=0.05291667cm 2.0,arrowlength=0.8,arrowinset=0.2]{->}(4.4,0.005)(4.7,0.005)
\rput[bl](6.0,-0.095){$\iota$}
\rput[bl](2.6,-0.045){=}
\rput[bl](3.1,-0.195){$(-1)^{|\iota|}$}
\rput[bl](7.75,-0.145){,}
\psline[linecolor=black, linewidth=0.018, arrowsize=0.05291667cm 2.0,arrowlength=0.8,arrowinset=0.2]{->}(0.8,0.005)(1.1,0.005)
\label{eq: cointegral inv of orientation}
\end{pspicture}

\vspace{1cm}

\item\hspace{3.5cm}\begin{pspicture}(0,-0.35374543)(6.7049885,0.35374543)
\rput[bl](0.4649884,-0.07374542){$m^{\text{op}}_H$}
\rput[bl](3.5549884,-0.35374543){$\Delta_{a^*}$}
\psline[linecolor=black, linewidth=0.018, arrowsize=0.05291667cm 2.0,arrowlength=0.8,arrowinset=0.2]{->}(3.1049883,-0.25374544)(3.4049883,-0.25374544)
\psline[linecolor=black, linewidth=0.018, arrowsize=0.05291667cm 2.0,arrowlength=0.8,arrowinset=0.2]{->}(1.2049884,0.04625458)(1.5049884,0.04625458)
\rput[bl](1.6549884,-0.053745423){$\mu$}
\psline[linecolor=black, linewidth=0.018, arrowsize=0.05291667cm 2.0,arrowlength=0.8,arrowinset=0.2]{->}(2.0049884,0.04625458)(2.3049884,0.04625458)
\psline[linecolor=black, linewidth=0.018, arrowsize=0.05291667cm 2.0,arrowlength=0.8,arrowinset=0.2]{->}(4.3049884,-0.25374544)(4.6049886,-0.053745423)
\psline[linecolor=black, linewidth=0.018, arrowsize=0.05291667cm 2.0,arrowlength=0.8,arrowinset=0.2]{->}(4.3049884,0.3462546)(4.6049886,0.14625458)
\rput[bl](4.754988,-0.053745423){$m_H$}
\psline[linecolor=black, linewidth=0.018, arrowsize=0.05291667cm 2.0,arrowlength=0.8,arrowinset=0.2]{->}(5.504988,0.04625458)(5.8049884,0.04625458)
\rput[bl](5.9549885,-0.053745423){$\mu$}
\psline[linecolor=black, linewidth=0.018, arrowsize=0.05291667cm 2.0,arrowlength=0.8,arrowinset=0.2]{->}(6.3049884,0.04625458)(6.6049886,0.04625458)
\rput[bl](2.6049883,-0.0037454225){=}
\rput[bl](6.6549883,-0.30374542){,}
\psline[linecolor=black, linewidth=0.018, arrowsize=0.05291667cm 2.0,arrowlength=0.8,arrowinset=0.2]{->}(0.0049884035,-0.25374544)(0.3049884,-0.053745423)
\psline[linecolor=black, linewidth=0.018, arrowsize=0.05291667cm 2.0,arrowlength=0.8,arrowinset=0.2]{->}(0.0049884035,0.3462546)(0.3049884,0.14625458)
 \label{eq: integral trace property}
\end{pspicture}

\vspace{1cm}
\item\hspace{3.5cm}  
\begin{pspicture}(0,-0.34388107)(6.5857143,0.34388107)
\psline[linecolor=black, linewidth=0.018, arrowsize=0.05291667cm 2.0,arrowlength=0.8,arrowinset=0.2]{->}(4.985714,0.15611893)(5.285714,0.35611892)
\psline[linecolor=black, linewidth=0.018, arrowsize=0.05291667cm 2.0,arrowlength=0.8,arrowinset=0.2]{->}(4.985714,-0.043881074)(5.285714,-0.24388108)
\rput[bl](4.3357143,-0.09388107){$\Delta_H$}
\rput[bl](5.485714,-0.34388107){$m_b$}
\psline[linecolor=black, linewidth=0.018, arrowsize=0.05291667cm 2.0,arrowlength=0.8,arrowinset=0.2]{->}(6.0857143,-0.24388108)(6.385714,-0.24388108)
\rput[bl](1.2357141,-0.09388107){$\Delta^{\text{op}}_H$}
\rput[bl](2.585714,0.006118927){=}
\psline[linecolor=black, linewidth=0.018, arrowsize=0.05291667cm 2.0,arrowlength=0.8,arrowinset=0.2]{->}(0.0,0.056118928)(0.28571412,0.056118928)
\rput[bl](0.44571412,-0.043881074){$\iota$}
\psline[linecolor=black, linewidth=0.018, arrowsize=0.05291667cm 2.0,arrowlength=0.8,arrowinset=0.2]{->}(0.7857141,0.056118928)(1.0857141,0.056118928)
\psline[linecolor=black, linewidth=0.018, arrowsize=0.05291667cm 2.0,arrowlength=0.8,arrowinset=0.2]{->}(3.085714,0.056118928)(3.385714,0.056118928)
\rput[bl](3.5457141,-0.043881074){$\iota$}
\psline[linecolor=black, linewidth=0.018, arrowsize=0.05291667cm 2.0,arrowlength=0.8,arrowinset=0.2]{->}(3.885714,0.056118928)(4.1857142,0.056118928)
\rput[bl](6.535714,-0.34388107){.}
\psline[linecolor=black, linewidth=0.018, arrowsize=0.05291667cm 2.0,arrowlength=0.8,arrowinset=0.2]{->}(1.9857141,0.15611893)(2.2857141,0.35611892)
\psline[linecolor=black, linewidth=0.018, arrowsize=0.05291667cm 2.0,arrowlength=0.8,arrowinset=0.2]{->}(1.9857141,-0.043881074)(2.2857141,-0.24388108)
\label{eq: cointegral cotrace property} 
\end{pspicture}

\vspace{1cm}

\noindent \hspace{-0.6cm} We will further require the following conditions:\\

\item \hspace{3.5cm}
\begin{pspicture}(0,-0.14)(5.6,0.14)
\rput[bl](0.5,-0.14){$i_A$}
\rput[bl](1.45,-0.14){$\pi_B$}
\psline[linecolor=black, linewidth=0.018, arrowsize=0.05291667cm 2.0,arrowlength=0.8,arrowinset=0.2]{->}(1.0,-0.04)(1.3,-0.04)
\psline[linecolor=black, linewidth=0.018, arrowsize=0.05291667cm 2.0,arrowlength=0.8,arrowinset=0.2]{->}(2.0,-0.04)(2.3,-0.04)
\psline[linecolor=black, linewidth=0.018, arrowsize=0.05291667cm 2.0,arrowlength=0.8,arrowinset=0.2]{->}(3.1,-0.04)(3.4,-0.04)
\rput[bl](3.6,-0.14){$\epsilon_A$}
\rput[bl](4.65,-0.14){$\eta_B$}
\psline[linecolor=black, linewidth=0.018, arrowsize=0.05291667cm 2.0,arrowlength=0.8,arrowinset=0.2]{->}(5.2,-0.04)(5.5,-0.04)
\rput[bl](2.6,-0.09){=}
\rput[bl](5.55,-0.14){,}
\psline[linecolor=black, linewidth=0.018, arrowsize=0.05291667cm 2.0,arrowlength=0.8,arrowinset=0.2]{->}(0.0,-0.04)(0.3,-0.04)
\label{eq: invariance under bdry isotopy of arc}
\end{pspicture}

\vspace{1cm}
\item\hspace{3.5cm}
\begin{pspicture}(0,-0.145)(4.45,0.145)
\rput[bl](1.8,-0.145){$\mu$}
\psline[linecolor=black, linewidth=0.018, arrowsize=0.05291667cm 2.0,arrowlength=0.8,arrowinset=0.2]{->}(2.15,-0.045)(2.45,-0.045)
\rput[bl](2.65,-0.145){$\epsilon_B$}
\rput[bl](0.0,-0.145){$\eta_A$}
\rput[bl](3.35,-0.095){=}
\psline[linecolor=black, linewidth=0.018, arrowsize=0.05291667cm 2.0,arrowlength=0.8,arrowinset=0.2]{->}(0.55,-0.045)(0.8357143,-0.045)
\rput[bl](1.01,-0.145){$\iota$}
\psline[linecolor=black, linewidth=0.018, arrowsize=0.05291667cm 2.0,arrowlength=0.8,arrowinset=0.2]{->}(1.35,-0.045)(1.65,-0.045)
\rput[bl](3.8,-0.145){$\text{id}_{\mathbb{K}}$}
\rput[bl](4.4,-0.145){.}
\label{eq: stabilization}
\end{pspicture}

\end{enumerate}
\end{definition}
\bigskip

We pause a bit to explain this definition. Recall that if $H$ is any finite-dimensional (ungraded) Hopf algebra and $\mu\in H^*, c\in H$ are respectively the right integral and right cointegral, then there is a group-like $b\in G(H)$ satisfying
\begin{align*}
\mu(bx)=\mu S_H(x) 
\end{align*}
for all $x\in H$ and
\begin{align*}
c_{(2)}\ot c_{(1)}=c_{(1)}\ot bS^{-2}(c_{(2)})
\end{align*}
where $\De_H(c)=\sum c_{(1)}\ot c_{(2)}$ in Sweedler's notation, see for example \cite[Thm. 10.5.4]{Radford:BOOK}. This is very similar to conditions (\ref{eq: integral inv of orientation}) and (\ref{eq: cointegral cotrace property}) above ($S^2$ does not appears there since we assumed involutivity). However, it will be essential in our construction that ``$b$ goes out of $\mu$", that is $\mu(bx)=b\mu(x)$, see Lemma \ref{lemma: multiply by b on lower handlebody does not affect result} and Proposition \ref{prop: changing basepoints and Z}. Thus, we really need $b$ to belong to the subalgebra $B$ of $H$ over which $\mu$ takes its values. Hence if $A=B=\kk$ and $\mu,\iota$ are the usual Hopf algebra integral and cointegral, then they will be compatible in our sense only if $b=1_H$ and $a^*=\e_H$, i.e., when both $H$ and $H^*$ are unimodular.

%It will also be essential in the proof of topological invariance of $Z_H$ that the group-like $b$ appearing in (\ref{eq: integral inv of orientation}) is the same that appears in (\ref{eq: cointegral cotrace property}).

%\begin{remark}If one wants to generalize the above axioms to non-involutive Hopf algebras, then the square of the antipode comes into play.\end{remark}

% NOTE: The integral $\mu$ is not a morphism in $\Svect$! So which is the underlying symmetric category for the Hopf algebra $\La_X$?

\begin{example}
Consider the relative integral and cointegral of $H_n$ defined in Example \ref{example: Hopf algebra Hn finite dim quotient}. These satisfy the compatibility condition of Definition \ref{def: compatible integral cointegral} where the group-likes $b\in G(B)$ and $a^*\in G(A^*)$ are given by $b=K$ and $a^*=\e_A$. % Indeed, one has \begin{equation*}\begin{split}-S_B\mu S_H(K^iX)=-S_B\mu(-K^{-i-1}X)=S_B(K^{-i-1})&=K^{i+1}\\&=\mu(K^{i+1}X)\end{split}\end{equation*}while both sides of (\ref{eq: integral inv of orientation}) vanish for $h=K^i$, so that (\ref{eq: integral inv of orientation}) holds. It is also easy to check $S(\iota)=-\iota$ so that (\ref{eq: cointegral inv of orientation}) holds. The trace property (\ref{eq: integral trace property}) holds since $a^*=\e$and $H_n$ is commutative while
For example, we have
\begin{equation*}
\label{eqexample: cotrace property}
\begin{split}
\iota_{(2)}\ot \iota_{(1)}&=\left(\sum_{i=0}^{n-1}K^i\ot K^i\right)(X\ot K+1\ot X)\\
&=\left(\sum_{i=0}^{n-1}K^i\ot K^i\right)(1\ot K)(X\ot 1+(K^{-1}\ot K^{-1})K\ot X)\\
&= (1\ot K)\iota_{(1)}\ot \iota_{(2)}
\end{split}
\end{equation*}
so condition (\ref{eq: cointegral cotrace property}) holds. 

%Finally, equation (\ref{eq: astar over b equals one}) is trivial.

%Note that since $b^*=\e_B$, $\mu$ is $B$-linear in the usual sense (where $H$ is a $B$-module via $i_B$) but $\pi_A$ is not an $A$-comodule map since $a\neq 1$. 

\end{example}

\section[Sutured manifolds and Heegaard diagrams]{Sutured manifolds and Heegaard diagrams}
\label{section: Sutured manifolds and Heegaard diagrams}

In this section we recall some notions of sutured manifold theory. We start by defining sutured manifolds and give some examples. In Subsection \ref{subs: Heegaard diagrams} we explain how these manifolds are represented using Heegaard diagrams, following \cite{Juhasz:holomorphic} and \cite{JTZ:naturality}. In Subsection \ref{subs: Extended Heegaard diagrams} we introduce extended Heegaard diagrams and we prove a version of the Reidemeister-Singer theorem for such diagrams. In Subsection \ref{subsection: dual curves} we define some canonical homology classes associated to a Heegaard diagram, under the additional assumption that $R_-(\c)$ is connected. In Subsection \ref{subs: Spinc structures} we define $\Spinc$ structures and recall the construction of the map $s$ of \cite{OS1, Juhasz:holomorphic}. Finally, we discuss homology orientations following \cite{FJR11} in Subsection \ref{subsection: homology orientation}. In what follows, all 3-manifolds will be assumed to be compact and oriented unless explicitly stated. We use the letter $Y$ to denote a closed 3-manifold, and $M$ for manifolds with non-empty boundary.

%Some questions arising: \begin{enumerate}\item The structure on the boundary of a sutured manifold is like a Heegaard splitting of the boundary, so the "Heegaard submanifold" is fixed and there is no stabilization move analogue. Our cut systems are Heegaard diagrams of these splittings (well, kind of...the cut system of $R_-$ is internal to it and does not quite relates to that of $R_+$, only near $\p\c=\p \S$. A real Heegaard diagram of a surface should be understood on a 1-dimensional object). Thus we are actually constructing invariants for surfaces. Formalize all this.\item Is it possible we could define invariants for bordered manifolds containing our sutured manifold invariant?\item A surgery presentation of $M$ also gives a presentation of $\pi_1(M)$ (or the complement of a link). We should be able to compute torsion via Fox calculus in this context. Do this means Hennings invariant can be enhanced with a character as we did with Kuperberg's invariants?\item Sutured manifolds are not quite convenient for knots with framings. But bordered Floer seems so. Indeed "bordered structure" at the bdry = a framing! Which is the bordered diagram associated to trefoil? Or knots in general...the doubly pointed diagrams do not work.\end{enumerate}

\subsection[Sutured manifolds]{Sutured manifolds} Sutured manifolds were introduced by Gabai in \cite{Gabai:foliations}. We use a slightly less general definition, as in \cite{Juhasz:holomorphic, JTZ:naturality}.

\begin{definition}
\label{def: sutured manifold}
A {\em sutured manifold} is a pair $(M,\c)$ where $M$ is a 3-manifold with boundary and $\c$ is a collection of pairwise disjoint annuli 
%$A(\c)$ and pairwise disjoint tori $T(\c)$ 
contained in $\p M$. Each annuli $A$ in $\c$ is supposed to be the tubular neighborhood of a single oriented simple closed curve called a {\em suture}, and the set of such sutures is denoted by $s(\c)$. We further suppose that each component of $R(\c)\eq\p M\sm \inte(\c)$ is oriented and we require that each (oriented) component of $\p R(\c)$ represents the same homology class in $H_1(\c)$ as some suture. We denote by $R_+(\c)$ (resp. $R_-(\c)$) the union of the components of $R(\c)$ whose orientation coincides (resp. is opposite) with the induced orientation of $\p M$. We say that $(M,\c)$ is {\em balanced} if $M$ has no closed components, $\chi(R_-(\c))=\chi(R_+(\c))$ and every component of $\p M$ has at least one suture.
\end{definition}

%We will restrict to a simpler class of sutured manifolds.

%\begin{definition}A {\em balanced sutured manifold} is a sutured manifold $(M,\c)$ where $M$ has no closed components, $\chi(R_-(\c))=\chi(R_+(\c))$ and every component of $\p M$ has at least one suture.\end{definition}

If $M$ has no closed components and every component of $\p M$ contains a suture (a {\em proper} sutured manifold as in \cite{JTZ:naturality}), then $s(\c)$ determines $R_{\pm}(\c)$. Thus, we only need to specify $s(\c)$ in order to define a balanced sutured manifold. 
%We now give some examples of balanced sutured manifolds.

%A sutured manifold can be thought of as a 3-manifold with boundary, endowed with a Heegaard decomposition of the boundary. Note that this is weaker than requiring a parametrization of the boundary, so there is not a well-defined gluing along the boundary operation between sutured manifolds. What is interesting about these manifolds, is that there is a good notion of decomposition along surfaces, see \cite{Juhasz: holomorphic}.\medskip

%NOTE: well, there are some gluing operations, see the paper by Grigsby-Wehrli.

\begin{example}
\label{example: sutured from closed}
Pointed closed 3-manifolds: consider a closed 3-manifold $Y$ together with a basepoint $p\in Y$. Then if $B$ is an open ball neighborhood of $p$ in $Y$, the complement $Y\sm B$ becomes a balanced sutured 3-manifold if we let $s(\c)$ be a single oriented simple closed curve on $\p B$. Both surfaces $R_-$ and $R_+$ are disks in this case. We denote this sutured manifold by $Y(1)$. Note that a pointed diffeomorphism between pointed closed 3-manifolds is equivalent to a diffeomorphism between the associated sutured 3-manifolds.
\end{example}

\begin{example} 
\label{example: sutured from link complement}
Link complements: let $L$ be a link in a closed 3-manifold $Y$ and let $N(L)$ be a closed tubular neighborhood of $L$. Then $Y\sm \inte (N(L))$ is a balanced sutured manifold if we put two oppositely oriented meridians on each component of $\p N(L)$. We denote this sutured 3-manifold by $Y(L)$. Here $R_+$ consist of one annulus component for each component of $L$ and similarly for $R_-$.
\end{example}

%\begin{example}\label{example: sutured from Seifert surface}Seifert surface complements: Let $S$ be a compact oriented surface-with-boundary with no closed components which is embedded in a closed 3-manifold $Y$ and let $N(S)\cong S\t [-1,1]$ be a closed tubular neighborhood of $S$ in $Y$. Then $Y\sm \inte (N(S))$ is a balanced sutured manifold if we let $s(\c)=\p S\t \{0\}, \c=\p S\t [-1,1], R_+=S\t \{1\}$ and $R_-=S\t \{-1\}$. %We denote this sutured manifold by $Y(S)$.\end{example}

%NOTE: Grigsby-Wehrli's examples are also nice.
%NOTE: Sutured link complement can be made bordered in a canonical way. But what about surface complements? How is this related to the fact that bordered Floer categorifies the Seifert matrix?\medskip

\subsection[Heegaard diagrams]{Heegaard diagrams}
\label{subs: Heegaard diagrams}

%\begin{enumerate}\item More references: Gabai defined sutured manifolds, Juhasz describes HDs for these.\item More detail on ``compressing $\S$ along $\aa$".\end{enumerate}

We now describe how to represent sutured manifolds via sutured Heegaard diagrams as in \cite{Juhasz:holomorphic}, \cite{JTZ:naturality}. %These generalize the pointed Heegaard diagrams of closed 3-manifolds of \cite{OS1} and the doubly-pointed Heegaard diagrams of knots of \cite{OS:knot}. 
We denote by $I$ the interval $[-1,1]$.

\begin{definition}
\label{def: Heegaard diagram}
Let $(M,\c)$ be a sutured manifold. An (embedded) {\em sutured Heegaard diagram} of $(M,\c)$ is a tuple $\HH=\HD$ consisting of the following data: \begin{enumerate}
\item An embedded oriented surface-with-boundary $\S\subset M$ such that 
$\p\S=s(\c)$ as oriented 1-manifolds,
\item A set $\aa=\{\a_1,\dots,\a_n\}$ of pairwise disjoint embedded circles in $\inte(\S)$ bounding disjoint disks to the negative side of $\S$,
\item A set $\bb=\{\b_1,\dots,\b_m\}$ of pairwise disjoint embedded circles in $\inte(\S)$ bounding disjoint disks to the positive side of $\S$.
\end{enumerate}
We further require that if $\S$ is compressed along $\aa$ (resp. $\bb$), inside $M$, we get a surface isotopic to $R_-$ (resp. $R_+$) relative to $\c$. Thus, $M$ can be written as $M=\Ua\cup\Ub$ with $\Ua\cap\Ub=\S$ where $\Ua$ (resp. $\Ub$) is homeomorphic to the sutured manifold obtained from $R_-\t I$ (resp. $R_+\t I$) by gluing 1-handles to $R_-\t \{1\}$ (resp. $R_+\t \{0\}$) with belt circles the $\a$ curves (resp. $\b$ curves). We say that $\Ua$ (resp. $\Ub$) is the lower (resp. upper) compression body corresponding to the Heegaard diagram. We say that $\HD$ is {\em balanced} if $|\aa|=|\bb|$ and every component of $\S\sm \aa$ contains a component of $\p \S$ (and similarly for $\S\sm \bb$).
\end{definition}

It can be shown that any sutured 3-manifold admits an embedded Heegaard diagram \cite[Prop. 2.14]{Juhasz:holomorphic}. Moreover, a sutured manifold is balanced if and only if any Heegaard diagram of it is balanced \cite[Prop. 2.9]{Juhasz:holomorphic}.

\begin{remark}\label{remark: Heegaard diagram gives handle decomposition}
Suppose that $\HD$ is an embedded Heegaard diagram of $(M,\c)$. Then $(M,\c)$ is diffeomorphic to the sutured manifold constructed in the following way: take $R_-\t I$ and attach $n$ 3-dimensional 1-handles to $R_-\t I$ along $R_-\t\{1\}$ so that the $\a_i$ become the belt circles of these 1-handles. The upper boundary of the manifold thus obtained can be identified with $\S$. Then, attach $m$ 3-dimensional 2-handles along the curves $\b_i\subset \S$. This defines a sutured 3-manifold with sutures $\p\S\t I$ and is clearly diffeomorphic to $(M,\c)$.
%In other words, a Heegaard diagram specifies a handlebody decomposition of $M_{\HH}$ relative to $R_-\t I$, where the handles are attached in increasing order according to their index. Thus, it corresponds to a {\em self-indexing Morse function} (see \cite{Milnor:h-cobordism}) on $M_{\HH}$.
\end{remark}

%If $\HH=\HD$ is an embedded diagram of $(M,\c)$, then of course it is also an abstract diagram. Conversely, an abstract diagram $\HH$ is an embedded diagram of $M_{\HH}$. Thus, from now on we assume all Heegaard diagrams are embedded diagrams (of some sutured manifold). The fact that $\HH$ is embedded in $M$ implies that there is a {\em canonical} homeomorphism (up to isotopy) $d:M\to M_{\HH}$. Usually, one says that an abstract diagram $\HH$ is a diagram for $(M,\c)$ if {\em there is} an homeomorphism $M\cong M_{\HH}$, but this is not enough to define the homology classes (in $H_1(M)$) of Subsection \ref{subsection: dual curves}.
%NOTE: moreover now Heegaard diagrams for knots have sense!\medskip

%The following theorem is proved in \cite[Prop. 2.14]{Juhasz:holomorphic}. \begin{theorem}Any (balanced) sutured 3-manifold admits a (balanced) sutured Heegaard diagram. \end{theorem}

%Note that from this theorem sutured manifolds become a very natural object: they are the 3-manifolds described by Heegaard diagrams where the Heegaard surface $\S$ has boundary.

For the following two examples, see \cite[Section 9]{Juhasz:holomorphic}.

\begin{example}
\label{example: sutured HD for closed manifold}
Let $\HD$ be a Heegaard diagram of a closed 3-manifold $Y$ in the usual sense (so $\S$ is a closed oriented surface). If $z\in \S\sm (\aa\cup\bb)$ is a basepoint then let $\S'=\S\sm D$ where $D$ is a small open disk centered at $z$. Then $(\S',\aa,\bb)$ is a Heegaard diagram of the sutured manifold $Y(1)$.
\end{example}

\begin{example}
\label{example: sutured HD for link complement}
Let $\HDzw$ be a doubly pointed Heegaard diagram of a knot $K$ in a closed 3-manifold $Y$, that is, $\HD$ is an embedded Heegaard diagram of $Y$, $K$ is disjoint from the disks bounded by the $\a$ and $\b$ curves and $K\cap \S$ consists of two points $z,w\in\S\sm(\aa\cup\bb)$ (see \cite{OS:knot}). An example of a doubly pointed Heegaard diagram for the left-handed trefoil is given in Figure \ref{figure: doubly pointed HD of trefoil}. Let $\S'=\S\sm (D_z\cup D_w)$ where $D_z,D_w$ are small disks centered at $z,w$ respectively. Then $(\S',\aa,\bb)$ is a Heegaard diagram for the sutured manifold $Y(K)$. %Note that if $K$ is a knot in $S^3$, represented by a planar projection $D$ with $c$ crossings, then there is a doubly pointed Heegaard diagram for $(S^3,K)$ associated to $D$ with $g(\S)=c+1$, cf. \cite[Example 3.3]{Manolescu:intro}. The diagram of Figure \ref{figure: doubly pointed HD of trefoil} does not comes from this procedure.
\end{example}

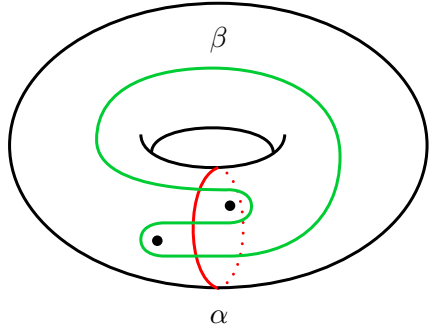
\begin{figure}[h]
\centering
\begin{pspicture}(0,-2.14)(5.590509,2.14)
\definecolor{colour0}{rgb}{0.0,0.8,0.2}
\psbezier[linecolor=red, linewidth=0.04, linestyle=dotted, dotsep=0.10583334cm](2.7952542,-1.6740508)(3.2366102,-1.5269322)(3.2366102,-0.2138983)(2.7952542,-0.0667796610169512)
\psellipse[linecolor=black, linewidth=0.04, dimen=outer](2.7952542,0.22745761)(2.7952542,1.9125423)
\psbezier[linecolor=black, linewidth=0.04](1.7654237,0.37457627)(1.7654237,-0.2138983)(3.677966,-0.2138983)(3.677966,0.3745762711864398)
\psbezier[linecolor=black, linewidth=0.04](1.9076384,0.13183051)(1.9076384,0.57318646)(3.5259435,0.57318646)(3.5259435,0.13183050847457536)
\psbezier[linecolor=red, linewidth=0.04](2.7952542,-0.06677966)(2.3538983,-0.2138983)(2.3465424,-1.5232543)(2.7878983,-1.670372881355929)
\psbezier[linecolor=colour0, linewidth=0.04](1.1769491,0.30101696)(1.1769491,-0.43457627)(2.9423728,-0.36101696)(2.9423728,-0.361016949152542)(2.9423728,-0.36101696)(3.2366102,-0.36101696)(3.2366102,-0.5816949)(3.2366102,-0.8023729)(2.9423728,-0.8023729)(2.9423728,-0.8023729)(2.9423728,-0.8023729)(2.059661,-0.8023729)(2.059661,-0.8023729)(2.059661,-0.8023729)(1.7654237,-0.8023729)(1.7654237,-1.0230509)(1.7654237,-1.2437288)(2.059661,-1.2437288)(2.059661,-1.2437288)(2.059661,-1.2437288)(2.1332202,-1.2437288)(2.9423728,-1.2437288)(3.7515254,-1.2437288)(4.4135594,-0.8023729)(4.4135594,0.080338985)(4.4135594,0.96305084)(3.457288,1.2572881)(2.721695,1.2572881)(1.9861016,1.2572881)(1.1769491,1.0366101)(1.1769491,0.30101696)
\rput[bl](2.6844068,-2.14){$\alpha$}
\rput[bl](2.6844068,1.46){$\beta$}
\psdots[linecolor=black, dotsize=0.14](2.9510734,-0.5733333)
\psdots[linecolor=black, dotsize=0.14](1.9844067,-1.04)
\end{pspicture}
\medskip

\caption{A doubly pointed Heegaard diagram of a left-handed trefoil. Removing two small disks centered at $z,w$ (represented by black dots) produces a sutured Heegaard diagram of the associated sutured manifold}
\label{figure: doubly pointed HD of trefoil}

\end{figure}

We will use an embedded version of the Reidemeister-Singer theorem. For this, we make a few definitions. If $\HH_1=(\S_1,\aa_1,\bb_1),\HH_2=(\S_2,\aa_2,\bb_2)$ are two Heegaard diagrams, a diffeomorphism $d:\HH_1\to \HH_2$ consists of an orientation-preserving diffeomorphism $d:\S_1\to\S_2$ such that $d(\aa_1)=\aa_2$ and $d(\bb_1)=\bb_2$.
\medskip

%NOTE: Compare the abstract/embedded moves below.

\begin{definition}[{\cite[Definition 2.34]{JTZ:naturality}}]
\label{def: diffeo isotopic to identity}
Let $\HH_1=(\S_1,\aa_1,\bb_1), \HH_2=(\S_2,\aa_2,\bb_2)$ be two Heegaard diagrams of $(M,\c)$ and denote by $j_i:\S_i\to M$ the inclusion map, for $i=1,2$. A diffeomorphism $d:\HH_1\to \HH_2$ is {\em isotopic to the identity in $M$} 
%(\cite[Definition 2.34]{JTZ: naturality})  such that $d(\aa_1)=\aa_2, d(\bb_1)=\bb_2$ and 
if $j_2\circ d:\S_1\to M$ is isotopic to $j_1:\S_1\to M$ relative to $s(\c)$.
\end{definition}

\begin{definition}
Let $\HH=\HD$ be a Heegaard diagram. Let $\d$ be an arc embedded in $\inte(\S)$ connecting a point of a curve $\a_j$ to a point of a curve $\a_i$ and such that $\inte (\d)\cap \aa=\emptyset$. There is a neighborhood of $\a_j\cup\d\cup\a_i$ which is a pair of pants embedded in $\S$ and whose boundary consists of the curves $\a_j,\a_i$ and a curve $\a'_j$. We say that $\a'_j$ is obtained by {\em handlesliding} the curve $\a_j$ over $\a_i$ along the arc $\d$. Similarly, we can handleslide a $\b$ curve over another along an arc $\d\subset \inte(\S)$ such that $\inte(\d)\cap\bb=\emptyset$.
\end{definition}

\begin{definition}
Let $\HH=\HD$ be a Heegaard diagram of a sutured manifold $(M,\c)$.  Let $D\subset \inte(\S)\sm (\aa\cup\bb)$ be a disk. Let $\S'$ be a connected sum $\S'=\S\# T$ along $D$, where $T$ is a torus. Let $\a',\b'$ be two curves in $T$ intersecting transversely in one point and let $\aa'=\aa\cup \{\a'\}$, $\bb'=\bb\cup\{\b'\}$. Then $\HH'=(\S',\aa',\bb')$ is a Heegaard diagram of $(M,\c)$ which we say is obtained by a {\em stabilization} of the diagram $\HH$. We also say that $\HH$ is obtained by destabilization of $\HH'$.
\end{definition}

\begin{theorem}[{\cite[Prop. 2.36]{JTZ:naturality}}]
\label{thm: embedded RS thm of JTZ}
Any two embedded Heegaard diagrams of a sutured 3-manifold $(M,\c)$ are related by a finite sequence of the following moves:
\begin{enumerate}
\item Isotopy of $\aa$ (or $\bb$) in $\inte(\S)$.
\item Diffeomorphisms isotopic to the identity in $M$.
\item Handlesliding an $\a$ curve (resp. $\b$ curve) over another $\a$ curve (resp. $\b$ curve).
\item Stabilization.
\end{enumerate}
\end{theorem}

Note that this theorem is stronger than the usual Reidemeister-Singer theorem \cite[Prop. 2.15]{Juhasz:holomorphic} which states that a {\em diffeomorphism class} of sutured manifolds is specified by a Heegaard diagram up to isotopy, handlesliding, stabilization and diagram diffeomorphism. 

%NOTE: I still need to say something with respect to the $\Spinc$ structure $\x$, $\x\in\Tab$ when doing Heegaard moves. The point is, when we prove invariance of Heegaard moves, I assume there is a multipoint before and after the move. We have to say that the $\Spinc$ structure $s_{\HH}(\x)=s_{\HH'}(\x)$.

\subsection[Extended Heegaard diagrams]{Extended Heegaard diagrams}
\label{subs: Extended Heegaard diagrams}

%\begin{enumerate}\item I still need a proof of RS thm for extended Heegaard diagrams. Note that a cut system of $R_-$ determines a diffeomorphism (up to isotopy... and more?) with a standard surface. Maybe we can write a diffeomorphism in terms of Dehn twists and realize each such in terms of our moves.\end{enumerate}

In what follows, we let $R$ be a compact orientable surface with boundary. We suppose $R$ has no closed components. 

\begin{definition}
A {\em cut system} of $R$ is a collection $\bolda=\{a_1,\dots, a_l\}$ of pairwise disjoint arcs properly embedded in $R$ such that for each component $R'$ of $R$, the complement $R'\sm R'\cap N(\bolda)$ is homeomorphic to a closed disk, where $N(\bolda)$ is an open tubular neighborhood of $\bolda$.
\end{definition}

\begin{remark}
\label{remark: cut systems in terms of homology}
If $\bolda=\{a_1,\dots,a_l\}$ is a collection of pairwise disjoint properly embedded arcs in $R$, then $\bolda$ is a cut system if and only if $[\bolda]=\{[a_1],\dots,[a_l]\}$ is a basis of $H_1(R,\p R)$. We don't prove this here as we don't need it, but see Lemma \ref{lemma: compression disks give homology basis} for a similar statement.

 %To see this, it suffices to suppose $R$ is connected. By excision, we have \begin{align*}H_*(\p R\cup\bolda, \p R)\cong H_*(\bolda,\p\bolda)\end{align*}and also\begin{align*}H_*(R,\p R\cup\bolda)\cong H_*(R',\p R')\end{align*}where $R'=R\sm N(\bolda)$. The exact sequence of the triple $(R,\p R\cup\bolda, \p R)$ becomes\begin{align*} 0\to \Z\to H_2(R',\p R')\to  H_1(\bolda,\p\bolda)\to H_1(R,\p R)\to H_1(R',\p R')\to 0.\end{align*}Thus, $[\bolda]$ is a basis of $H_1(R,\p R)$ if and only if $H_2(R',\p R')=\Z$ and $H_1(R',\p R')=0$, i.e., when $R'$ is a disk. % NOTE: If $\bolda$ was l.i. and maximal, then $H_1(R',\p R')$ would be finite. But R' is orientable, so this H_1 has to be zero!
\end{remark}

A cut system on $R$ always exists: take a handlebody decomposition of $R$ with a single 0-handle on each component of $R$. Then the cocores of the 1-handles define a cut system. Of course, a cut system is not unique: we can isotope the arcs (we allow their endpoints to be isotoped along $\p R$) and handleslide them. 

\begin{definition}
Let $\bolda=\{a_1,\dots, a_l\}$ be a cut system of $R$ and suppose that an arc $a_j$ has an endpoint on the same component $C$ of $\p R$ as another arc $a_i$. Suppose there is an arc $\d$ in $C$ connecting these two endpoints such that no other endpoint of an arc of $\bolda$ lies on $\d$. Then there is a neighborhood of $a_j\cup\d\cup a_i$ which is a disk $D$ embedded in $R$ and whose boundary $[\p D]\in H_1(R,\p R),$ consists of the arcs $a_j,a_i$ and a new arc $a'_j$. %Take an arc $\c$ in $\S\sm(\aa\cup \bolda)$ joining a point in the interior of $a_j$ to a point in the interior of $a_i$, and perform a connected sum $a_j\# a_i$ along that arc. This gives two new arcs: one whose endpoints lie on a same component of $\p \S$ and another arc $a'_j$. We discard the first arc, and say that $a'_j$ is obtained by handlesliding $a_j$ over $a_i$ (see Figure \ref{figure: handlesliding arc over arc}). It is clear that $(\bolda\sm \{a_j\}) \cup \{a'_j\}$ is a new cut system of $R_-$. 
We say that $a'_j$ is obtained by {\em handlesliding} $a_j$ over $a_i$. It is clear that $(\bolda\sm \{a_j\}) \cup \{a'_j\}$ also defines a cut system of $R$. See Figure \ref{figure: handlesliding arc over arc}.
\end{definition}

\begin{figure}[h]
\centering
\begin{pspicture}(0,-2.2)(6.65,2.2)
\psellipse[linecolor=black, linewidth=0.04, dimen=outer](3.2,0.2)(3.2,2.0)
\psline[linecolor=red, linewidth=0.04](2.0,-0.2)(2.0,-1.6133333)
\psline[linecolor=red, linewidth=0.04](4.4,-0.2)(4.4,-1.6133333)
\psbezier[linecolor=red, linewidth=0.04](2.2,-0.11333334)(2.2,-0.61333334)(2.2,-0.9)(2.2,-1.3)(2.2,-1.7)(4.2,-1.7)(4.2,-1.3)(4.2,-0.9)(4.2,-0.38)(4.2,-0.11333334)
\pscircle[linecolor=black, linewidth=0.04, dimen=outer](2.0,0.2){0.4}
\pscircle[linecolor=black, linewidth=0.04, dimen=outer](4.4,0.2){0.4}
\rput[bl](1.8666667,-2.2){$a_j$}
\rput[bl](4.3,-2.1666667){$a_i$}
\rput[bl](3.0666666,-1.3){$a'_j$}
%\rput[bl](6.6,-2.1){.}
\end{pspicture}

\caption{An arc $a_j$ is handleslided over an arc $a_i$}
\label{figure: handlesliding arc over arc}
\end{figure}
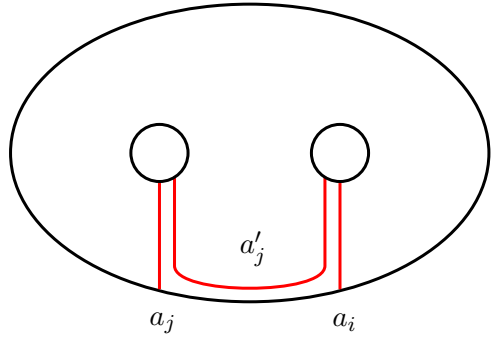

We now prove that these moves suffice to relate any two cut systems over $R$. In the following lemma, we assume all arcs are properly embedded (see \cite[Prop. 2.4]{OS1} for a similar statement).

\begin{lemma}
\label{lemma: making disjoint cut systems with hsliding}
Let $\bolda=\{a_1,\dots,a_l\}$ be a cut system of $R$. 
\begin{enumerate}
\item If $a'$ is an arc in $R$ disjoint from $\bolda$ and $[a']\neq 0$ in $H_1(R, \p R)$, then there is an $i$ such that $a'$ is isotopic to an arc obtained by handlesliding $a_i$ over some of the $a_j$ with $j\neq i$.
\item If $\bolda'$ is another cut system of $R$, then one can make $\bolda\cap\bolda'=\emptyset$ after sufficiently many handleslides and isotopies inside $\bolda$.
\end{enumerate}

\end{lemma}
\begin{proof}
It suffices to suppose $R$ is connected. For both proofs we think of $R$ as obtained from a disk $D$ by attaching $l$ one-handles, where each handle is attached along two points $p_i,q_i\in \p D$ and has cocore $a_i$ for $i=1,\dots,l$. We isotope each $a_i$ to a small arc contained in $D$ around $p_i$. Then, handlesliding $a_i$ over the other arcs of $\bolda$ corresponds to moving the endpoints of $a_i$ through $\p D\sm \{q_i\}$. So let $a'$ be an arc as in $(1)$. Since $a'$ is disjoint from $\bolda$, we can suppose it is contained in $D$, so it separates the points of $F=\cup_{i=1}^l\{p_i,q_i\}$ into two disjoint sets, call them $X,Y$. Now since $[a']\neq 0$ in $H_1(R,\p R)$, there must exist an $i$ such that $p_i\in X$ and $q_i\in Y$. Then one obtains a curve isotopic to $a'$ by handlesliding $a_i$ along all arcs $a_j$ for which $p_j\in X$ or $q_j\in X$ (if both $p_j,q_j\in X$ then we need to handleslide $a_i$ over $a_j$ twice). Now let $\bolda'$ be another cut system. We suppose that all the endpoints of the arcs of $\bolda'$ lie over $\p D\sm F$. We prove $(2)$ by induction on $N=\sum_{i,j}|a_i\cap a'_j|$. If $N=0$, then we are done. If $N>0$, then there is an arc, say $a'_1$, that intersects an arc of $\bolda$, say $a_1$. Let $x_0\in\p D\sm F$ be one of the endpoints of $a'_1$ and let $x_1\in a'_1\cap a_1$ so that $a'_1|_{[x_0,x_1]}\cap\bolda=a'_1|_{[x_0,x_1]}\cap a_1=\{x_1\}$ (here $a'_1|_{[x_0,x_1]}$ denotes the subarc of $a'_1$ with endpoints $x_0,x_1$). As in the proof of $(1)$, we can handleslide $a_1$ across $\bolda\sm\{a_1\}$ to place its endpoint right next to $x_0$ in such a way that the intersection $a'_1|_{[x_0,x_1]}\cap a_1$ is removed. This decreases $N$ at least by one, and we are done by the induction hypothesis.
\end{proof}

\begin{proposition}
\label{prop: RS thm for cut systems}
Any two cut systems of a compact orientable surface-with-boundary $R$ (without closed components) are related by isotopy and handlesliding.
\end{proposition}
\begin{proof}
It suffices to suppose $R$ connected. Let $\bolda,\bolda'$ be two cut systems over $R$. By Lemma \ref{lemma: making disjoint cut systems with hsliding} $(2)$ we may suppose $\bolda\cap\bolda'=\emptyset$ and by $(1)$ of the same lemma we may suppose $a_1=a'_1$. Then we can cut $R$ along $a_1$ to get two cut systems on a surface $R'$ with $\rk H_1(R')=\rk H_1(R)-1$. The proof then follows by induction.
\end{proof}

\begin{definition}
Let $\HH=\HD$ be a balanced sutured Heegaard diagram. A {\em cut system} of $(\S,\aa)$ is a collection $\bolda=\{a_1,\dots,a_l\}$ of properly embedded arcs in $\S$ where:
\begin{enumerate}
\item The sets $\aa$ and $\bolda$ are disjoint in $\S$.
\item The collection $\bolda$ is a cut system in $\S[\aa]$, the surface obtained from $\S$ by doing surgery along $\aa$.
\end{enumerate}
\end{definition}

Cut systems of $(\S,\bb)$ are defined in a similar way. If $\bolda,\bolda'$ are two cut systems of $(\S,\aa)$, then they are related by isotopy and handlesliding in $\S[\aa]$ by Proposition \ref{prop: RS thm for cut systems}. Note that isotopies may pass through the traces of the surgery, this corresponds to handlesliding an arc in $\bolda$ over a curve of $\aa$. 

\begin{corollary}
\label{corollary: moves for cut systems of Heegaards diagram}
Let $\HD$ be a balanced sutured Heegaard diagram. Any two cut systems of $(\S,\aa)$ are related by isotopy, arc-arc handlesliding and arc-curve handlesliding.
\end{corollary}

\begin{definition}
\label{def: extended HD}
An {\em extended} sutured Heegaard diagram is a tuple $\HH=(\S,\aa,\bolda,\bb,\boldb)$ where $\HH=\HD$ is a Heegaard diagram and $\bolda$ (resp. $\boldb$) is a cut system of $(\S,\aa)$ (resp. $(\S,\bb)$). If $d=|\aa|$ and $l=|\bolda|$, we will usually note $\aaa=\{\a_1,\dots,\a_{d+l}\}$ where $\aa=\{\a_1,\dots,\a_d\}$ and $\bolda=\{\a_{d+1},\dots,\a_{d+l}\}$. We use similar notation for $\bbb=\bb\cup\boldb$. Thus, we note a general extended Heegaard diagram as $\HH=\HDD$.
\end{definition}

\begin{example}
\label{example: extended diagram for left trefoil}
The Heegaard diagram of the left trefoil given in Example \ref{example: sutured HD for link complement} can be extended by letting $\bolda=\{a\},\boldb=\{b\}$ as in Figure \ref{figure: cut systems for left trefoil}.
\end{example}
\begin{figure}[t]
\centering
\begin{pspicture}(0,-2.14)(12.334407,2.14)
\definecolor{colour0}{rgb}{0.0,0.8,0.2}
\psbezier[linecolor=red, linewidth=0.04](9.268475,-0.06677966)(8.827119,-0.2138983)(8.819762,-1.509921)(9.261119,-1.6570395480225977)
\psbezier[linecolor=red, linewidth=0.04, linestyle=dotted, dotsep=0.10583334cm](9.268475,-1.6607175)(9.70983,-1.5135989)(9.70983,-0.2138983)(9.268475,-0.0667796610169512)
\psbezier[linecolor=black, linewidth=0.04](8.380858,0.13183051)(8.380858,0.57318646)(9.999164,0.57318646)(9.999164,0.13183050847457536)
\psbezier[linecolor=colour0, linewidth=0.04](7.6501694,0.30101696)(7.6501694,-0.43457627)(9.415593,-0.36101696)(9.415593,-0.361016949152542)(9.415593,-0.36101696)(9.70983,-0.36101696)(9.70983,-0.5816949)(9.70983,-0.8023729)(9.415593,-0.8023729)(9.415593,-0.8023729)(9.415593,-0.8023729)(8.532882,-0.8023729)(8.532882,-0.8023729)(8.532882,-0.8023729)(8.238644,-0.8023729)(8.238644,-1.0230509)(8.238644,-1.2437288)(8.532882,-1.2437288)(8.532882,-1.2437288)(8.532882,-1.2437288)(8.606441,-1.2437288)(9.415593,-1.2437288)(10.224746,-1.2437288)(10.88678,-0.8023729)(10.88678,0.080338985)(10.88678,0.96305084)(9.930509,1.2572881)(9.194915,1.2572881)(8.459322,1.2572881)(7.6501694,1.0366101)(7.6501694,0.30101696)
%\rput[bl](12.284407,-1.9057627){.}
\pscircle[linecolor=black, linewidth=0.04, dimen=outer](8.606441,-1.0230509){0.14711864}
\psbezier[linecolor=blue, linewidth=0.04](8.753796,-1.0208095)(9.34227,-1.0208095)(9.949366,-1.082188)(9.977411,-0.713238885788187)(10.005457,-0.34428972)(9.7377405,-0.086666666)(9.950124,0.09014689)
\psbezier[linecolor=red, linewidth=0.04, linestyle=dotted, dotsep=0.10583334cm](2.7952542,-1.6740508)(3.2366102,-1.5269322)(3.2366102,-0.2138983)(2.7952542,-0.0667796610169512)
\psellipse[linecolor=black, linewidth=0.04, dimen=outer](2.7952542,0.22745761)(2.7952542,1.9125423)
\psbezier[linecolor=black, linewidth=0.04](1.7654237,0.37457627)(1.7654237,-0.2138983)(3.677966,-0.2138983)(3.677966,0.3745762711864398)
\psbezier[linecolor=black, linewidth=0.04](1.9076384,0.13183051)(1.9076384,0.57318646)(3.5259435,0.57318646)(3.5259435,0.13183050847457536)
\pscircle[linecolor=black, linewidth=0.04, dimen=outer](2.1332202,-1.0230509){0.14711864}
\pscircle[linecolor=black, linewidth=0.04, dimen=outer](2.8688135,-0.5816949){0.14711864}
\psbezier[linecolor=red, linewidth=0.04](1.9861016,-1.0230509)(1.3976271,-1.0230509)(0.5149152,-0.43457627)(0.5149152,0.44813559322033714)(0.5149152,1.3308475)(1.6918644,1.7722034)(2.721695,1.7722034)(3.7515254,1.7722034)(5.075593,1.3308475)(5.075593,0.44813558)(5.075593,-0.43457627)(3.7515254,-0.5816949)(3.015932,-0.5816949)
\psbezier[linecolor=blue, linewidth=0.04, linestyle=dotted, dotsep=0.10583334cm](10.984406,-1.24)(10.684406,-1.14)(10.571864,-1.0288136)(10.484406,-0.74)(10.396949,-0.45118645)(10.365198,0.033446327)(10.004067,0.11466666)
\psbezier[linecolor=black, linewidth=0.04](8.238644,0.37457627)(8.238644,-0.2138983)(10.151186,-0.2138983)(10.151186,0.3745762711864398)
\psbezier[linecolor=red, linewidth=0.04](2.7952542,-0.06677966)(2.3538983,-0.2138983)(2.3465424,-1.5232543)(2.7878983,-1.670372881355929)
\psbezier[linecolor=colour0, linewidth=0.04](1.1769491,0.30101696)(1.1769491,-0.43457627)(2.9423728,-0.36101696)(2.9423728,-0.361016949152542)(2.9423728,-0.36101696)(3.2366102,-0.36101696)(3.2366102,-0.5816949)(3.2366102,-0.8023729)(2.9423728,-0.8023729)(2.9423728,-0.8023729)(2.9423728,-0.8023729)(2.059661,-0.8023729)(2.059661,-0.8023729)(2.059661,-0.8023729)(1.7654237,-0.8023729)(1.7654237,-1.0230509)(1.7654237,-1.2437288)(2.059661,-1.2437288)(2.059661,-1.2437288)(2.059661,-1.2437288)(2.1332202,-1.2437288)(2.9423728,-1.2437288)(3.7515254,-1.2437288)(4.4135594,-0.8023729)(4.4135594,0.080338985)(4.4135594,0.96305084)(3.457288,1.2572881)(2.721695,1.2572881)(1.9861016,1.2572881)(1.1769491,1.0366101)(1.1769491,0.30101696)
\rput[bl](2.6844068,-2.14){$\alpha$}
\rput[bl](9.184406,-2.04){$\alpha$}
\rput[bl](2.6844068,0.76){$\beta$}
\rput[bl](9.184406,0.76){$\beta$}
\rput[bl](4.984407,-0.34){$a$}
\rput[bl](7.6844068,-0.94){$b$}
\psbezier[linecolor=blue, linewidth=0.04](9.209407,-0.58666664)(8.92274,-0.58666664)(8.51774,-0.5933333)(8.264407,-0.56)(8.011073,-0.52666664)(7.1844068,-0.44)(7.1844068,0.26)(7.1844068,0.96)(7.8044066,1.68)(9.184406,1.66)(10.564406,1.64)(11.284407,0.96)(11.384407,0.26)(11.484406,-0.44)(11.404407,-1.045)(11.004407,-1.245)
\psellipse[linecolor=black, linewidth=0.04, dimen=outer](9.268475,0.22745761)(2.7952542,1.9125423)
\pscircle[linecolor=black, linewidth=0.04, dimen=outer](9.342033,-0.5816949){0.14711864}
\end{pspicture}
\medskip

\caption{Cut systems for the diagram of the left trefoil complement of Example \ref{example: sutured HD for link complement}. Taken together they define an extended Heegaard diagram. On the right we depict the arc $b$ in blue to avoid confusion}
\label{figure: cut systems for left trefoil}
\end{figure}
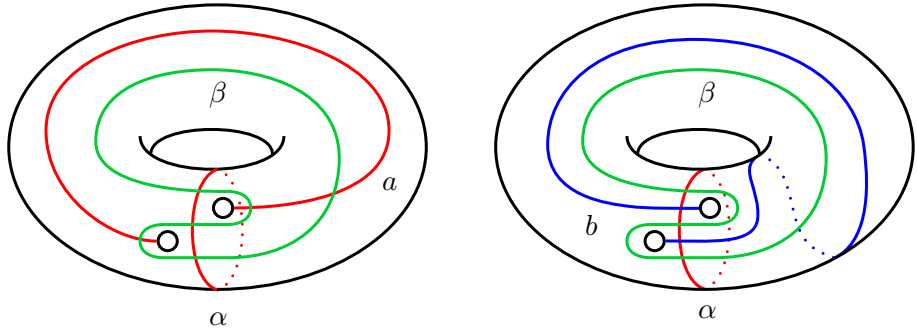

\begin{definition}
\label{def: extended Heegaard move}
Let $\HH=\HDD$ be an extended Heegaard diagram of $(M,\c)$. An {\em extended Heegaard move} consist of the Heegaard moves of $\HD$ and the moves between cut systems of $(\S,\aa)$ and $(\S,\bb)$ (i.e. arc-arc and arc-curve handlesliding) in which all isotopies and handleslidings of $\aaa$ are performed in the complement of $\aaa$ in $\S$ (and similarly for $\bbb$).
\end{definition}

\begin{proposition}
\label{prop: extended RS thm}
Any two extended Heegaard diagrams of a sutured 3-manifold $(M,\c)$ are related by extended Heegaard moves.
\end{proposition}
\begin{proof}
Let $\HH=\HDD$ and $\HH'=\HDDp$ be two extended Heegaard diagrams of $(M,\c)$. By Proposition \ref{prop: RS thm for cut systems} it suffices to prove that whenever the diagrams are related by one of the moves of Theorem \ref{thm: embedded RS thm of JTZ}, then the latter move can be performed in the complement of the cut systems. This is obvious for all moves except when handlesliding closed curves. So suppose $\HH'$ is obtained from $\HH$ by handlesliding a curve $\a_j$ over a curve $\a_i$ along an arc $\d$. Suppose the arc $\d$ intersects some arcs in $\bolda$. Then we can successively handleslide these arcs along $\a_i$ to get a new cut system $\bolda''$ for $\HH$ which is disjoint of $\d$. Hence, $\bolda''$ is also a cut system for $\HH'$ and so is related to $\bolda'$ by the moves of Corollary \ref{corollary: moves for cut systems of Heegaards diagram}. Thus we can pass from $\HH$ to $\HH'$ using extended Heegaard moves, as desired.\end{proof}

%NOTE: a cut system seems to give a "bordered structure" to $R_-$. Invariance of cut systems moves would mean independence of parametrization. So it seems our invariant actually is an evaluation of a bordered invariant.

%NOTE: Which is the stronger version of RS thm for bordered manifolds? (analogue to the version of JTZ). Is there such a version in the literature, maybe we can use such a thing and then prove our invariant is independent of the parametrization.

\subsection[Dual curves]{Dual curves} 
\label{subsection: dual curves}

%NOTE: This subsection should be placed just after the definition of embedded HD.\medskip

%NOTE: For the moment I assume $R_-$ is connected. Maybe we can band-connect-sum but I shall prove then independence of the band chosen. Another possibility is to remember that in the link case there is no arc-arc sliding, so no need to apply projection and $\psi$.\medskip

Let $(M,\c)$ be a balanced sutured 3-manifold. We will now require $R=R_-(\c)$ to be {\em connected}. 

\begin{definition}
\label{def: oriented HD}
An extended Heegaard diagram $\HH=\HDD$ of $(M,\c)$ is said to be {\em oriented} if each curve or arc in $\aaa\cup\bbb$ is oriented.
\end{definition}

Let $\HH=\HDD$ be an extended oriented diagram of $(M,\c)$. Let $\aa=\{\a_1,\dots,\a_d\}$ and $\bolda=\{\a_{d+1},\dots,\a_{d+l}\}$. 
We will denote by $\Ua$ and $\Ub$ the (oriented) compression bodies corresponding to $\HD$. Thus, for each $i=1,\dots,d$ there is a disk $D_i$ embedded in $\Ua$ so that $\p D_i=\a_i$ and the given orientation of $\a_i$ determines an orientation of $D_i$. Now think of $\Ua$ as built from $R_-\t I$ by attaching one-handles to $R_-\t\{1\}$ (with belt circles the $\a$ curves), see Remark \ref{remark: Heegaard diagram gives handle decomposition}. For each $i=1,\dots,l$ let $D_{d+i}=\a_{d+i}\t I\subset R_-\t I$ and give $D_{d+i}$ the {\em opposite} of the product orientation (this choice of orientation is explained in Remark \ref{remark: handlesliding alpha curves} below). This way, the oriented disks $D_1,\dots,D_{d+l}$ define homology classes in $H_2(\Ua,\p \Ua)$.

\begin{lemma}
\label{lemma: compression disks give homology basis}
Let $\Ua$ be the lower compression body of $(M,\c)$ corresponding to $\HD$. Then, the homology classes $[D_1],\dots,[D_{d+l}]$ defined above constitute a basis of $H_2(\Ua,\p\Ua)$.
\end{lemma}
\begin{proof}
Let $V=D_1\cup\dots\cup D_{d+l}$. By excision, we have $$H_*(\Ua,V\cup\p\Ua)\cong H_*(B,\p B)$$ where $B=\Ua\sm N(V)$ and $N(V)$ is an open tubular neighborhood of $V$. Similarly, by excising $\p U_{\a}\sm N(V)\cap \p U_{\a}$ we get $$H_*(V\cup\p\Ua, \p\Ua)\cong H_*(V,\p V).$$ Thus, the long exact sequence of the triple $(\Ua,V\cup\p\Ua, \p\Ua)$ becomes 
\begin{align*}
0\to H_3(\Ua,\p\Ua)\to H_3(B,\p B)\to H_2(V,\p V)\to H_2(\Ua,\p\Ua)\to H_2(B,\p B).
\end{align*}
Now, since $R_-$ is connected, $B$ is homeomorphic to a 3-ball. This implies that the fourth arrow above is an isomorphism, which is equivalent to the statement of the lemma.
\end{proof}

Now consider the Poincar\'e duality pairing
\begin{align*}
H_2(\Ua,\p\Ua)\t H_1(\Ua)&\to \Z \\
(x,y)\hspace{1cm}&\mapsto x\cdot y.
\end{align*}
Since this pairing is non-degenerate, there is a basis $\c_1,\dots,\c_{d+l}$ of $H_1(\Ua)$ characterized by
\begin{align*}
[D_i]\cdot \c_j=\d_{ij}
\end{align*}
for all $i,j=1,\dots,d+l$.

\begin{definition}
\label{def: dual homology classes}
Let $\HDD$ be an oriented extended Heegaard diagram of a sutured 3-manifold $(M,\c)$ with connected $R_-$. Let $j:\Ua\to M$ be the inclusion. We will denote $\a^*_i=j(\c_i)\in H_1(M)$ for each $i=1,\dots,d+l$ where $\c_1,\dots,\c_{d+l}$ is the Poincar\'e dual basis of $[D_1],\dots,[D_{d+l}]$ as above.
\end{definition}

For each $i=1,\dots,d$, the class $\a^*_i$ can be represented by an oriented simple closed curve that goes through the one-handle corresponding to $\a_i$ once and disjoint from the disks $D_j$ for $j=1,\dots, d+l, j\neq i$. It can even be represented as an oriented simple closed curve in $\S$ satisfying $\a_i\cdot \a^*_i=+1$ (where the intersection is taken with respect to the orientation of $\S$) and $\a^*_i\cap (\aaa\sm\{\a_i\})=\emptyset$. Similarly, for each $i=1,\dots,l$, the class $\a^*_{d+i}$ can be represented as an oriented simple closed curve in $\S$ intersecting $\a_{d+i}$ positively in one point and disjoint from all curves or arcs in $\aaa$. The orientation of $D_{d+i}$ chosen above guarantees that $\a_{d+i}\cdot \a^*_{d+i}=+1$ over $\S$ for each $i=1,\dots,l$.
\medskip

Of course, one can also take the Poincar\'e dual basis of the disks corresponding to $\bbb$ and then take their image in $H_1(M)$. This gives homology classes $\b^*_1,\dots,\b^*_{d+l}$ in $H_1(M)$ where $\bbb=\{\b_1,\dots,\b_{d+l}\}$. If these classes are represented in the surface $\S$, then they are oriented so that $\b^*_i\cdot\b^*_j=\d_{ij}$ for $i,j=1,\dots, d+l$. 

\begin{remark}
\label{remark: handlesliding alpha curves} 
Suppose a closed curve $\a'_j$ is obtained by handlesliding a curve $\a_j$ over a curve $\a_i$. Let $\HH'=(\S,\aaap,\bbb)$ be the Heegaard diagram thus obtained. We denote by $\a'_i$ the curve $\a_i$ seen as a curve in $\HH'$. Suppose the curves $\a_i,\a_j,\a'_j$ are oriented so that
\begin{align*}
\p P=\a_i\cup\a_j\cup (-\a'_j)
\end{align*}
as oriented 1-manifolds, where $P$ is the handlesliding region. Let $D_i,D_j,D'_j$ be the disks corresponding to $\a_i,\a_j,\a'_j$. Then one has
\begin{align*}
\p B= -D_i\cup -D_j\cup D'_j
\end{align*}
as oriented surfaces, where $B$ is the three-ball where the handlesliding occurs. This implies $[D'_j]=[D_j]+[D_i]$ in $H_2(\Ua,\p\Ua)$ and so
\begin{align*}
(\a'_i)^*=\a^*_i(\a^*_j)^{-1}
\end{align*}
while $(\a'_j)^*=\a^*_j$, where we use multiplicative notation for $H_1(M)$. By our orientation conventions for the disks $D_{d+i}$ for $i=1,\dots,l$, a similar assertion holds if we handleslide an arc over a curve or an arc over an arc. Of course, similar statements hold for $\bbb$.
\end{remark}

We introduce a last piece of notation which we will use repeatedly.

\begin{definition}
\label{def: ov l for l arc}
\label{notation: waaa and wbbb for immersed curve in Sigma}
\def\waaa{w_{\aa,\bolda}}\def\wbbb{w_{\bb,\boldb}} Let $l$ be an oriented immersed circle or arc in $\S$, transversal to $\aaa$. If $l$ is a circle, we assume it has a basepoint. We construct an element $\ov{l}$ in $H_1(M)$ as follows:
\begin{enumerate}
\item At each intersection point $x\in \a\cap l$ where $\a\in\aaa$ write $(\a^*)^{m(x)}$ where $m(x)\in\{\pm 1\}$ denotes the sign of intersection at $x$.
\item Multiply all the elements encountered from left to right, following the orientation of $l$ and starting from the basepoint of $l$ (if $l$ is an arc we take its starting point as basepoint).
\end{enumerate}
If $l$ is transversal to $\bbb$ then there is an element in $H_1(M)$ defined in a similar way (but the signs are taken at $x\in l\cap\b$ and not $\b\cap l$). We also denote this by $\ov{l}$, whenever there is no confusion.
\end{definition}

\begin{remark}
\label{remark: ovl equals homology class of l if cycle}
If $l$ is a 1-cycle in $\S$ transversal to $\aaa$, then its homology class $[l]$ in $H_1(M)$ coincides with $\ov{l}$. Note that if $l$ is an arc in $\S$ which is not properly embedded (e.g. a subarc of an $\a$ curve), then $\ov{l}$ is by no means an isotopy invariant of $l$.
\end{remark}

\subsection[$\Spinc$ structures]{$\Spinc$ structures}
\label{subs: Spinc structures}

%\begin{enumerate}\item NOTE: why do we have such a height function? I mean, we could Dehn twist along $s(\c)\subset\p M$.\end{enumerate}

%There are many equivalent definitions of $\Spinc$ structures on closed 3-manifolds. We will take the standard definition used in Heegaard Floer theory, thus, we follow \cite[Sect. 2.6]{OS1} closely.

Let $(M,\c)$ be a connected sutured manifold. Fix a nowhere vanishing vector field $v_0$ on $\p M$ with the following properties:
\begin{enumerate}
\item It points into $M$ along $\inte R_-$,
\item It points out of $M$ along $\inte R_+$,
\item It is given by the gradient of the height function $\c=s(\c)\t [-1,1]\to [-1,1]$ on $\c$. 
\end{enumerate}

\begin{definition}
\label{def: spinc structure as homology class of vector fields}
Let $v$ and $w$ be two non-vanishing vector fields on $M$ such that $v|_{\p M}=v_0=w|_{\p M}$. We say that $v$ and $w$ are {\em homologous} if they are homotopic rel $\p M$ in the complement of an open 3-ball embedded in $\inte(M)$ where the homotopy is through non-vanishing vector fields. A {\em $\Spinc$ structure} is an homology class of such non-vanishing vector fields on $M$. We denote the set of $\Spinc$ structures on $M$ by $\Spinc(M,\c)$.
\end{definition}

The space of boundary vector fields $v_0$ with the above properties is contractible. This implies that there is a canonical identification between the set of $\Spinc$ structures coming from different boundary vector fields. Thus, we make no further reference to $v_0$.
\medskip

\begin{proposition}[{\cite[Prop. 3.6]{Juhasz:polytope}}]
\label{prop: spinc is affine space over Htwo}
Let $M$ be a connected sutured manifold. Then $\Spinc(M,\c)$ is non-empty if and only if $M$ is balanced. In such a case, the group $H^2(M,\p M)$ acts freely and transitively over $\Spinc(M,\c)$.
\end{proposition}

%Let $v,w$ be two non-vanishing vector fields on a closed 3-manifold $M$. Provide $M$ with a CW structure. Then, since $\pi_1(S^2)=0$ there is no obstruction to homotope $v$ to $w$ over the 1-skeleton on $M$ among non-vanishing vector fields. To homotope along the 2-skeleton, there is an obstruction lying in $H^2(M;\Z)$. But then, since $H^3(M\sm B^3)=0$, there is no obstruction to homotope along the 3-skeleton of $M\sm B^3$. This implies that the set $\Spinc(M)$ is an {\em affine} space over $H^2(M;\Z)$: there is a free transitive action of $H^2(M;\Z)$ onto $\Spinc(M)$, but there is no canonical bijection between these two sets. 

We denote the action of $H^2(M,\p M)$ over $\Spinc(M,\c)$ by 
$(h,\ss)\mapsto \ss+h$. If $\ss_1,\ss_2$ denote two $\Spinc$ structures on $M$, we denote by $\ss_1-\ss_2$ the element $h\in H^2(M,\p M)$ such that $\ss_1=\ss_2+h$.

\medskip

We now study how $\Spinc$ structures of $(M,\c)$ can be understood directly from a Heegaard diagram. For this, we need the following definition.

\begin{definition}
Let $\HH=\HD$ be a balanced Heegaard diagram, where 
$\aa=\{\a_1,\dots,\a_d\}$ and $\bb=\{\b_1,\dots,\b_d\}$. A {\em multipoint} in $\HH$ is an unordered set $\x=\{x_1,\dots,x_d\}$ where $x_i\in\a_i\cap \b_{\s(i)}$ for each $i=1,\dots,d$ and $\s$ is some permutation in $S_d$. The set of multipoints of $\HH$ is denoted by $\Tab$.
\end{definition}

We use the notation $\Tab$ for the set of multipoints by the following reason. If we let $\Sym^d(\S)\eq\S^d/S_d$ where the symmetric group $S_d$ acts on $\S^d=\S\t\dots\t\S$ by permuting the factors, then the Heegaard diagram induces two tori 
$\Ta\eq \a_1\t\dots\t\a_d$ and $\Tb\eq \b_1\t\dots\t \b_d$ contained in $\Sym^d(\S)$. A multipoint $\x=\{x_1,\dots,x_d\}$ corresponds to an intersection point $\x\in\Tab$.

\medskip

Given a balanced Heegaard diagram $\HH=\HD$ of $(M,\c)$ with $d=|\aa|=|\bb|$, one can construct a map
\begin{align*}
s:\Tab\to \Spinc(M,\c)
\end{align*}
(see \cite[Section 2.6]{OS1} for the closed case and \cite[Section 4]{Juhasz:holomorphic} for the sutured extension). To do this, we first fix a Riemannian metric on $M$. Now, take a Morse function $f:M\to [-1,4]$ satisfying the following conditions:
\begin{enumerate}
\item $f(R_-)=-1, f(R_+)=4$ and $f|_{\c}$ is the height function 
$\c=s(\c)\t[-1,4]\to [-1,4]$. Here we choose a diffeomorphism $\c=s(\c)\t [-1,4]$ such that $s(\c)$ corresponds to $s(\c)\t \{3/2\}$.
\item For $i=1,2$, $f$ has $d$ index $i$ critical points and has value $i$ on these points. These lie in $\inte(M)$ and there are no other critical points.
\item One has $\S=f^{-1}(3/2)$, the $\a$ curves coincide with the intersection of the unstable manifolds of the index one critical points with $\S$ and the $\b$ curves coincide with the intersection of the stable manifolds of the index two critical points with $\S$.
\end{enumerate}
Such a Morse function always exists (see e.g. \cite[Prop. 6.17]{JTZ:naturality}). By the first condition, $\n f|_{\p M}$ satisfies the properties of the vector field $v_0$ in Definition \ref{def: spinc structure as homology class of vector fields}. Note that the only singularities of $\n f$ are the index one and index two critical points of $f$, denote them by $P_1,\dots, P_d$ and $Q_1,\dots, Q_d$ respectively. Then the last condition above means
\begin{align*}
W^u(P_i)\cap \S=\a_i  \hspace{0.4cm} \text{ and }  \hspace{0.4cm} W^s(Q_i)\cap\S=\b_i
\end{align*}
for each $i$. Here $W^u$ and $W^s$ denote respectively the unstable and stable submanifolds of $\n f$ at the corresponding critical point. Thus, an intersection point $x\in\a_i\cap\b_j$ corresponds to a trajectory of $\n f$ starting at $P_i$ and ending at $Q_j$. In particular, a multipoint $\x\in\Tab$ corresponds to a $d$-tuple $\c_{\x}$ of trajectories of $\n f$ connecting all the index one critical points to all the index two critical points. 

\begin{definition}
\label{def: sx of OSz}
Let $\HD$ be a balanced Heegaard diagram of $(M,\c)$ and $\x\in\Tab$. Let $f$ be a Morse function adapted to the Heegaard diagram as above. We define a $\Spinc$ structure $s(\x)$ as follows: let $N$ be a tubular neighborhood of $\c_{\x}$ in $M$ homeomorphic to a disjoint union of $d$ 3-balls. Then $\n f$ is a non-vanishing vector field over $M\sm N$. Since the critical points of $f$ have complementary indices on each component of $N$, one can extend $\n f|_{M\sm N}$ to a non-vanishing vector field over all of $M$. We let $s(\x)$ be the homology class of this vector field.
\end{definition}

\begin{definition}
\label{def: exy}
Let $\x,\y\in\Tab$ be two multipoints say $x_i\in\a_i\cap \b_i$ and $y_i\in\a_i\cap \b_{\s(i)}$ for each $i$. Let $c_i$ be an arc joining $x_i$ to $y_i$ along $\a_i$ and $d_i$ be an arc joining $y_{\s^{-1}(i)}$ to $x_i$ along $\b_i$. Then 
\begin{align*}
\disp\sum_{i=1}^d c_i+\sum_{i=1}^d d_i
\end{align*}
is a cycle in $\S$. We denote by $\e(\x,\y)$ the element of $H_1(M)$ induced by the cycle above.
\end{definition}

Note that for each $i$ there are two choices for an arc $c_i$ joining $x_i$ and $y_i$ along $\a_i$ (similarly for $d_i$) but the class $\e(\x,\y)\in H_1(M)$ is independent of which arc is chosen.
\medskip

\begin{lemma}[{\cite[Lemma 4.7]{Juhasz:holomorphic}}]
\label{lemma: spinc and exy}
For any $\x,\y\in\Tab$ we have
$$
PD[s(\x)-s(\y)]=\e(\x,\y)
$$
where $PD:H^2(M, \p M)\to H_1(M)$ is Poincar\'e duality.
\end{lemma}

\medskip

%\subsection[$Spinc$ structures and Heegaard moves]{$Spinc$ structures and Heegaard moves}
\def\Tabprime{\mathbb{T}_{\a'}\cap\mathbb{T}_{\b'}}
Now lets study what happens with the map $s$ when doing Heegaard moves. Let $\HH=\HD,\HH'=(\S',\aa',\bb')$ be two balanced Heegaard diagrams of $(M,\c)$. Then we have two maps
\begin{align*}
s:\Tab\to\Spinc(M,\c) \hspace{0.4cm} \text{ and } \hspace{0.4cm}  s':\Tabprime\to \Spinc(M,\c).
\end{align*}
Suppose $\HH'$ is obtained from $\HH$ by one of the moves of Theorem \ref{thm: embedded RS thm of JTZ}. In the case of isotopy, we suppose that $\HH'$ is obtained by an isotopy of an $\a$ or $\b$ curve that adds just two new intersection points. For all such moves there is an obvious map
\begin{align*}
j:\Tab\to \Tabprime.
\end{align*}
In the case of isotopies (that increase the number of intersection points) or handlesliding, it is clear that $\Tab\subset\Tabprime$ so we let $j$ be the inclusion. For diffeomorphisms isotopic to the identity in $M$ we just let $j$ be the bijection between multipoints induced by the diffeomorphism. For stabilization, if $\a_{d+1},\b_{d+1}$ are the stabilized curves intersecting in a point $x_{d+1}\in \S'$ then we let $j$ be the bijection $\x\mapsto \x\cup \{x_{d+1}\}$.

\begin{proposition}
\label{prop: Spinc map is preserved under Heegaard moves}
If $\HH'$ is obtained from $\HH$ by a Heegaard move and if the map $j$ is defined as above, then 
\begin{align*}
s'(j(\x))=s(\x)
\end{align*}
for all $\x\in\Tab$.
%\label{eq: Spinc map is preserved under Heegaard moves}
\end{proposition}
\begin{proof}
This is obvious for isotopies and diffeomorphisms isotopic to the identity in $M$. Suppose $\HH'$ is obtained from $\HH$ by stabilization. The union of the newly attached one-handle/two-handle pair is a 3-ball and it is clear that in the complement of this ball, the vector fields representing $s(\x)$ and $s'(j(\x))$ coincide so $s(\x)=s'(j(\x))$ by definition. Now suppose $\HH'$ is obtained from $\HH$ by handlesliding a curve $\a_1$ over $\a_2$ and let $\a'_1=\a_1\#\a_2$. Let $U_{\a}$ be the lower handlebody and let $D_1,D_2,D'_1\subset U_{\a}$ be the compressing disks corresponding to $\a_1,\a_2,\a'_1$ respectively. The complement of these three disks in $U_{\a}$ has two components: one contains the index zero critical point and the other is homeomorphic to a 3-ball $B$. Note that the boundary of $B$ is the union of $D_1,D_2,D'_1$ and the pair of pants bounded by $\a_1,\a_2,\a'_1$. One can pick Morse functions $f,f'$ adapted to $\HH,\HH'$ respectively such that in the complement of $B$ in $M$, the vector fields $-\n f$ and $-\n f'$ coincide. Moreover, the trajectories of $f$ associated to $\x$ coincide with the trajectories of $f'$ associated to $j(\x)$, so the vector fields representing $s(\x)$ and $s'(j(\x))$ coincide in the complement of $d+1$ 3-balls, which implies that $s(\x)=s'(j(\x))$. 
\end{proof}

\subsection[Homology orientations]{Homology orientations}
\label{subsection: homology orientation}

Let $\HD$ be a balanced Heegaard diagram of $(M,\c)$ and $d=|\aa|=|\bb|$. Denote $R_-=R_-(\c)$. Let $A\subset H_1(\S;\R)$ (resp. $B$) be the subspace spanned by $\aa$ (resp. $\bb$). These subspaces have dimension $d$ by \cite[Lemma 2.10]{Juhasz:holomorphic}. There is a bijection $o$ between orientations of the vector space $H_*(M,R_-;\R)$ and orientations of the vector space $\La^d(A)\ot \La^d(B)$ \cite[Sect. 2.4]{FJR11}. This is seen as follows. The Heegaard diagram specifies a handle decomposition of $(M,\c)$ relative to $R_-\t I$ with no handles of index zero or three. There are $d$ handles of index one and two, so the handlebody complex $C_*=C_*(M,R_-\t I;\R)$ is just $C_1\oplus C_2$ where both $C_1,C_2$ have dimension $d$. Now let $\o$ be an orientation of $H_*(M,R_-;\R)$ and let $h^1_1,\dots,h^1_m,h^2_1,\dots,h^2_m$ be an ordered basis of $H_*(M,R_-;\R)$ compatible with $\o$, where $h^i_j\in H_i(M,R_-;\R)$. Let $c^1_1,\dots,c^1_m,c^2_1,\dots,c^2_m\in C_*$ be chains representing this basis, where $c^i_j\in C_i$. Then, for any $b_1,\dots,b_{d-m}\in C_2$ such that $c^2_1,\dots,c^2_m,b_1,\dots,b_{d-m}$ is a basis of $C_2$ the collection\begin{align*}
c^1_1,\dots,c^1_m,\p b_1,\dots,\p b_{d-m},c^2_1,\dots,c^2_m,b_1,\dots,b_{d-m}
\end{align*}
is a basis of $C_*$ whose orientation $\o'$ depends only on $\o$. Now, an orientation of $C_*$ is specified by an ordering and orientation of the handles of index one and two. This is the same as an ordering and orientation of the curves in $\aa\cup\bb$ or equivalently, an orientation of $\La^d(A)\ot \La^d(B)$. This way, the orientation $\o$ induces an orientation of $\La^d(A)\ot \La^d(B)$ via $\o'$. We denote this orientation of $\La^d(A)\ot \La^d(B)$ by $o(\o)$.

\begin{remark}
\label{remark: canonical homology orientation when H is zero}
When $H_*(M,R_-;\R)=0$, the orientation $\o'$ of $C_*(M,R_-;\R)$ constructed above is {\em canonical}. In this case, an ordering and orientation of the curves of $\aa\cup\bb$ corresponds to the canonical orientation if and only if
\begin{align*}
\det(\a_i\cdot\b_j)>0.
\end{align*}

\end{remark}

\section{The invariant}
\label{section: the invariant}

In this section we define the sutured 3-manifold invariant $I_H^{\rho}(M,\c,\ss,\o)$. We begin by extending Kuperberg's construction \cite{Kup1} to extended Heegaard diagrams, at least provided $R_-(\c)$ is connected. This leads to a scalar $Z_H^{\rho}(\HH,\o)$ which is a topological invariant only when both the relative integral and cointegral are two-sided. In Subsection \ref{subsection: basepoints} we devise a way to put basepoints on a Heegaard diagram, once a multipoint is chosen. The actual invariant is then defined in Subsection \ref{subs: Construction of Z_H}, using the previous rule to put basepoints and the relation between multipoints and $\Spinc$ structures. The case when $R_-(\c)$ is disconnected can be reduced to the connected case, this is explained in Subsection \ref{subs: disconnected case}. The most general definition of the invariant $I_H$ is given in Definition \ref{def: Z for disconnected R}. 
\medskip

During the whole section we let $H=(H,m_H,\eta_H,\De_H,\e_H,S_H)$ be an involutive Hopf superalgebra over a field $\kk$ endowed with a relative right integral $(B,i_B,\pi_B,\mu)$ and a compatible relative right cointegral $(A,\pi_A,i_A,\iota)$. We let $b\in G(B)$ and $a^*\in G(A^*)$ be the associated group-likes (see Definition \ref{def: compatible integral cointegral}).
\medskip

We also let $(M,\c)$ be a connected, oriented, balanced sutured 3-manifold with {\em connected $R_-(\c)$}, let $\o$ be an orientation of $H_*(M,R_-;\R)$ and $\rho:H_1(M)\to G(A\ot B^*)$ be a group homomorphism. All homology groups will be taken with integral coefficients unless explicitly stated.
%In this section we use our notions of relative integrals and cointegrals to construct an invariant $Z(M,\ss,\v,\psi)$ where $M$ is a closed 3-manifold, $\ss\in\Spinc(M)$ and $\v:\pi_1(M)\to G(A), \psi:\pi_1(M)\to G(B^*)$ are group homomorphisms. 

\subsection[Tensors associated to Heegaard diagrams]{Tensors associated to Heegaard diagrams}
\label{subs: tensors of HDs} Let $\HH=\HDD$ be an extended Heegaard diagram of $(M,\c)$. Recall that $\aaa=\aa\cup\bolda$ where $\aa$ consists of closed curves and $\bolda$ consists of properly embedded arcs in $\S$ (and similarly for $\bbb=\bb\cup\boldb$). We let $d=|\aa|=|\bb|$ and $l=|\bolda|=|\boldb|$ and note $\aa=\{\a_1,\dots,\a_d\}, \bolda=\{\a_{d+1},\dots,\a_{d+l}\}$. Similarly, we note $\bb=\{\b_1,\dots,\b_d\}, \boldb=\{\b_{d+1},\dots,\b_{d+l}\}$. We will endow our diagrams with some additional structure.

\begin{definition}
We say that $\HH$ is {\em ordered} if each of the sets $\aa,\bolda,\bb,\boldb$ is linearly ordered. We extend the linear orders to $\aaa$ (resp. $\bbb$) by declaring each curve of $\aa$  (resp. $\bb$) to come before each arc of $\bolda$ (resp. $\boldb$). %We will always assume the curves are numbered in increasing order, that is, $$\a_1<\dots<\a_d<\a_{d+1}<\dots<\a_{d+l}$$ and similarly for $\bbb$. 
We say that $\HH$ is {\em based} if for each $i=1,\dots, d$ there is a basepoint $p_i\in\a_i$ disjoint from $\bbb$ and a basepoint $q_i\in\b_i$ disjoint from $\aaa$. Whenever the arcs of $\HH$ are oriented, we will take the beginning point of each arc as its basepoint.
\end{definition}

\begin{definition}
Let $\HH$ be an oriented extended Heegaard diagram, that is, each curve or arc is oriented. We call an intersection point $x\in \aaa\cap\bbb$ a {\em crossing} of $\HH$ and we say it is {\em positive} if the tangent vectors of $\aaa$ and $\bbb$ at $x$ define an oriented basis of $T_x\S$ and {\em negative} otherwise. At each crossing $x$ we define a number $\e_x\in\{0,1\}$ by $\e_x=0$ if the crossing is positive and $\e_x=1$ if it is negative. In other words, $(-1)^{\e_x}=m_x$ where $m_x$ is the intersection sign at $x$.
\end{definition}

Let $\HH=\HDD$ be an ordered, oriented, based extended Heegaard diagram of a sutured manifold $(M,\c)$ as above.  We define a scalar 
\begin{align*}
Z^{\rho}_H(\HH)\in\kk
\end{align*}
as follows. First, write $\rho=(\v,\psi)$ under the isomorphism $G(A\ot B^*)\cong G(A)\ot G(B^*)$. Since we suppose $R_-(\c)$ is connected and $\HH$ is oriented, there is an element $\a^*\in H_1(M)$ associated to each $\a\in\aaa$ (Definition \ref{def: dual homology classes}), and thus also an element $\v(\a^*)\in G(A)$. We associate to each $\a\in\aaa$ the tensor

%We will always suppose that $\aaa$ and $\bbb$ are transversal as submanifolds of $\S$.

\begin{figure}[H]
\centering

\begin{pspicture}(0,-0.5786306)(9.77,0.5786306)
\psline[linecolor=black, linewidth=0.018, arrowsize=0.05291667cm 2.0,arrowlength=0.8,arrowinset=0.2]{->}(2.02,0.0213694)(2.42,0.0213694)
\rput[bl](2.59,-0.1186306){$\Delta_H$}
\psline[linecolor=black, linewidth=0.018, arrowsize=0.05291667cm 2.0,arrowlength=0.8,arrowinset=0.2]{->}(3.22,0.2213694)(3.62,0.6213694)
\psline[linecolor=black, linewidth=0.018, arrowsize=0.05291667cm 2.0,arrowlength=0.8,arrowinset=0.2]{->}(3.22,-0.1786306)(3.62,-0.5786306)
\rput[bl](3.48,-0.2186306){$\vdots$}
\rput[bl](0.0,-0.1586306){$\varphi(\alpha^*)$}
\psline[linecolor=black, linewidth=0.018, arrowsize=0.05291667cm 2.0,arrowlength=0.8,arrowinset=0.2]{->}(3.22,0.1213694)(3.62,0.3213694)
\psline[linecolor=black, linewidth=0.018, arrowsize=0.05291667cm 2.0,arrowlength=0.8,arrowinset=0.2]{->}(1.02,0.0213694)(1.42,0.0213694)
\rput[bl](1.66,-0.0386306){$\iota$}
\psline[linecolor=black, linewidth=0.018, arrowsize=0.05291667cm 2.0,arrowlength=0.8,arrowinset=0.2]{->}(7.82,0.0213694)(8.22,0.0213694)
\rput[bl](8.4,-0.1186306){$\Delta_H$}
\psline[linecolor=black, linewidth=0.018, arrowsize=0.05291667cm 2.0,arrowlength=0.8,arrowinset=0.2]{->}(9.02,0.2213694)(9.42,0.6213694)
\psline[linecolor=black, linewidth=0.018, arrowsize=0.05291667cm 2.0,arrowlength=0.8,arrowinset=0.2]{->}(9.02,-0.1786306)(9.42,-0.5786306)
\rput[bl](9.28,-0.2186306){$\vdots$}
\rput[bl](5.8,-0.1586306){$\varphi(\alpha^*)$}
\psline[linecolor=black, linewidth=0.018, arrowsize=0.05291667cm 2.0,arrowlength=0.8,arrowinset=0.2]{->}(9.02,0.1213694)(9.42,0.3213694)
\psline[linecolor=black, linewidth=0.018, arrowsize=0.05291667cm 2.0,arrowlength=0.8,arrowinset=0.2]{->}(6.82,0.0213694)(7.22,0.0213694)
\rput[bl](7.36,-0.1286306){$i_A$}
\rput[bl](4.62,-0.0786306){\text{or}}
%\rput[bl](9.72,-0.5786306){.}
\end{pspicture}
\end{figure}

\noindent depending on whether $\a$ is a closed curve (left) or $\a$ is an arc (right). The iterated coproduct has as many legs as crossings through $\a$, and from top to bottom the legs correspond to the crossings through $\a$ when starting from the basepoint and following its orientation. Note that changing the basepoint of $\a$ changes the cyclic order of the iterated coproduct, and this produces a different tensor when $\a$ is closed and $b\neq 1$ (see condition (\ref{eq: cointegral cotrace property}) of Definition \ref{def: compatible integral cointegral}). Similarly, for each $\b\in\bbb$ there is an element $\b^*
\in H_1(M)$ and hence $\psi(\b^*)\in G(B^*)$. To each $\b$ we assign the following tensor:

\begin{figure}[H]
\centering

\begin{pspicture}(0,-0.6091925)(9.759192,0.6091925)
\psline[linecolor=black, linewidth=0.018, arrowsize=0.05291667cm 2.0,arrowlength=0.8,arrowinset=0.2]{->}(0.009192505,-0.6)(0.4091925,-0.2)
\psline[linecolor=black, linewidth=0.018, arrowsize=0.05291667cm 2.0,arrowlength=0.8,arrowinset=0.2]{->}(0.009192505,0.6)(0.4091925,0.2)
\rput[bl](0.06919251,-0.24){$\vdots$}
\rput[bl](0.5191925,-0.07){$m_H$}
\psline[linecolor=black, linewidth=0.018, arrowsize=0.05291667cm 2.0,arrowlength=0.8,arrowinset=0.2]{->}(1.2091925,0.0)(1.6091925,0.0)
\psline[linecolor=black, linewidth=0.018, arrowsize=0.05291667cm 2.0,arrowlength=0.8,arrowinset=0.2]{->}(0.009192505,0.3)(0.4091925,0.1)
\rput[bl](1.7991925,-0.13){$\mu$}
\rput[bl](2.7591925,-0.17){$\psi(\beta^*)$}
\psline[linecolor=black, linewidth=0.018, arrowsize=0.05291667cm 2.0,arrowlength=0.8,arrowinset=0.2]{->}(2.2091925,0.0)(2.6091926,0.0)
\psline[linecolor=black, linewidth=0.018, arrowsize=0.05291667cm 2.0,arrowlength=0.8,arrowinset=0.2]{->}(5.8091927,-0.6)(6.2091923,-0.2)
\psline[linecolor=black, linewidth=0.018, arrowsize=0.05291667cm 2.0,arrowlength=0.8,arrowinset=0.2]{->}(5.8091927,0.6)(6.2091923,0.2)
\rput[bl](5.8691926,-0.24){$\vdots$}
\rput[bl](6.3191924,-0.07){$m_H$}
\psline[linecolor=black, linewidth=0.018, arrowsize=0.05291667cm 2.0,arrowlength=0.8,arrowinset=0.2]{->}(7.0091925,0.0)(7.4091926,0.0)
\psline[linecolor=black, linewidth=0.018, arrowsize=0.05291667cm 2.0,arrowlength=0.8,arrowinset=0.2]{->}(5.8091927,0.3)(6.2091923,0.1)
\rput[bl](7.4891925,-0.11){$\pi_B$}
\rput[bl](8.559193,-0.17){$\psi(\beta^*)$}
\psline[linecolor=black, linewidth=0.018, arrowsize=0.05291667cm 2.0,arrowlength=0.8,arrowinset=0.2]{->}(8.009192,0.0)(8.409192,0.0)
\rput[bl](4.6091924,-0.1){\text{or}}
%\rput[bl](9.709192,-0.5){.}
\end{pspicture}
\end{figure}

\noindent depending on whether $\b$ is closed (left) or $\b$ is an arc (right). As before, the legs from top to bottom of the product correspond to the crossings of $\b$ starting from its basepoint and following its orientation.
\medskip

Now let $x$ be a crossing of $\HH$. Then $x$ corresponds to a unique outcoming leg (resp. incoming leg) of a tensor associated to some $\a\in\aaa$ (resp. $\b\in\bbb$). If $\e_x\in\{0,1\}$ is defined as above, then we associate to $x$ the tensor
\begin{figure}[H]
\centering
\begin{pspicture}(0,-0.13)(1.6,0.13)
\psline[linecolor=black, linewidth=0.018, arrowsize=0.05291667cm 2.0,arrowlength=0.8,arrowinset=0.2]{->}(0.0,0.02)(0.4,0.02)
\rput[bl](0.49,-0.11){$S_H^{\epsilon_x}$}
\psline[linecolor=black, linewidth=0.018, arrowsize=0.05291667cm 2.0,arrowlength=0.8,arrowinset=0.2]{->}(1.1,0.02)(1.5,0.02)
\rput[bl](1.55,-0.13){.}
\end{pspicture}
\end{figure}

\noindent We now take the tensor product of the tensors corresponding to the $\a$'s for all $\a\in\aaa$. We assume the tensor product is ordered according to the order of $\aaa$. Similarly, we take the tensor product of the tensors corresponding to the $\b$'s for each $\b\in\bbb$, following the order of $\bbb$. These two tensor products are then contracted using the tensors associated to the crossings as above. The result of this contraction is a scalar in the base field of $H$.

\begin{definition}
\label{def: ZHrho and ZHrho with orientation}
Let $\HH$ be an ordered, oriented, based, extended Heegaard diagram. We denote by $Z_H^{\rho}(\HH)$ the scalar obtained by contracting all the tensors above. Given an orientation $\o$ of $H_*(M,R_-(\c);\R)$ we also define $$Z_H^{\rho}(\HH,\o)\eq \dHH Z_H^{\rho}(\HH)$$
where $\dHH$ is a sign defined by $$\dHH\eq (\pm 1)^{\deg(\mu)}=(\pm 1)^{\deg(\iota)}$$ where the $\pm 1$ sign is determined by $o(\HH)=\pm\o$ and $o$ is the map of Subsection \ref{subsection: homology orientation}. 
\end{definition}

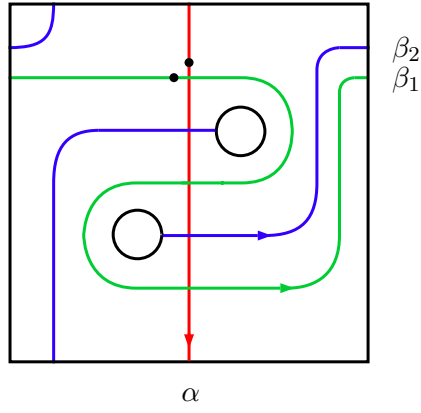
\begin{figure}[h]
\centering
\begin{pspicture}(0,-2.65)(5.41,2.65)
\definecolor{colour0}{rgb}{0.0,0.8,0.2}
\definecolor{colour1}{rgb}{0.2,0.0,1.0}
\pscircle[linecolor=black, linewidth=0.04, dimen=outer](3.0857143,0.9357143){0.34285715}
\pscircle[linecolor=black, linewidth=0.04, dimen=outer](1.7142857,-0.43571427){0.34285715}
\psline[linecolor=red, linewidth=0.04, arrowsize=0.05291667cm 2.0,arrowlength=1.4,arrowinset=0.0]{->}(2.4,2.65)(2.4,-1.95)
\psdots[linecolor=black, dotsize=0.12](2.4,1.85)
\psbezier[linecolor=colour0, linewidth=0.04](3.7,-1.15)(4.1,-1.15)(4.4,-0.95)(4.4,-0.45)(4.4,0.05)(4.4,1.25)(4.4,1.45)(4.4,1.65)(4.6,1.65)(4.6,1.65)
\rput[bl](2.3,-2.65){$\alpha$}
\rput[bl](5.1,1.85){$\beta_2$}
\rput[bl](5.1,1.45){$\beta_1$}
\psbezier[linecolor=colour1, linewidth=0.04](3.4,-0.45)(3.9,-0.45)(4.1,-0.23)(4.1,0.25)(4.1,0.73)(4.1,1.45)(4.1,1.75)(4.1,2.05)(4.4,2.05)(4.4,2.05)
\psline[linecolor=colour1, linewidth=0.04, arrowsize=0.05291667cm 2.0,arrowlength=1.4,arrowinset=0.0]{->}(2.05,-0.45)(3.5,-0.45)
\psbezier[linecolor=colour0, linewidth=0.04](0.0,1.65)(0.8,1.65)(2.4,1.65)(3.1,1.65)(3.8,1.65)(3.7714286,0.9357143)(3.7714286,0.9357143)(3.7714286,0.9357143)(3.8,0.25)(3.1,0.25)(2.4,0.25)(3.4,0.25)(2.3,0.25)
\psdots[linecolor=black, dotsize=0.12](2.2,1.65)
\psbezier[linecolor=colour0, linewidth=0.04, arrowsize=0.05291667cm 2.0,arrowlength=1.4,arrowinset=0.0]{<-}(3.8,-1.15)(3.2,-1.15)(2.4,-1.15)(1.7,-1.15)(1.0,-1.15)(1.0,-0.45)(1.0,-0.45)(1.0,-0.45)(1.0,0.25)(1.7,0.25)(2.4,0.25)(2.0,0.25)(2.4,0.25)
\psbezier[linecolor=colour1, linewidth=0.04](1.2,0.95)(0.8,0.95)(0.6,0.75)(0.6,0.25)(0.6,-0.25)(0.6,-0.15)(0.6,-0.85)(0.6,-1.55)(0.6,-1.55)(0.6,-2.15)
\psline[linecolor=colour1, linewidth=0.04](2.75,0.95)(1.2,0.95)
\psbezier[linecolor=colour1, linewidth=0.04](0.6,2.65)(0.6,2.2)(0.5,2.05)(0.0,2.05)
\psline[linecolor=colour1, linewidth=0.04](4.4,2.05)(4.8,2.05)
\psline[linecolor=colour0, linewidth=0.04](4.6,1.65)(4.766667,1.65)
\psframe[linecolor=black, linewidth=0.04, dimen=outer](4.8,2.65)(0.0,-2.15)
\psline[linecolor=red, linewidth=0.04](2.4,-1.75)(2.4,-2.1166666)
\end{pspicture}
\caption{A zoom in of the extended Heegaard diagram of the left trefoil, with orientations and basepoints. The arc $\a_2$ is not drawn as it plays no role in $Z_{H_n}^{\rho}$ since $A=\C$}
\label{figure: Z invariant of trefoil}
\end{figure}

\begin{example}
\label{example: computation of Z of left trefoil}
Let $H=H_n$ ($n\geq 1$) be the Hopf algebra over $\kk=\C$ of Example \ref{example: Hopf algebra Hn finite dim quotient} with the relative integral and cointegral given there, so $A=\C$ and $B\cong \C[\Z/n\Z]$. We compute $Z_{H}^{\rho}(\HH)$ where $\HH$ is the extended Heegaard diagram of the left trefoil complement on the right of Figure \ref{figure: cut systems for left trefoil}, with orientations and basepoints as in Figure \ref{figure: Z invariant of trefoil} below. We let $\rho:H_1(M)\to G(B^*)$ be given by $\rho(\beta^*_2)(K)=q$ where $q=e^{\frac{2\pi i}{n}}$ (note that $\b^*_2$ is a meridian of the knot). Note that with the notation above, $\v=1$ and $\rho=\psi$ since $G(A)=\{1\}$. Moreover, since $i_A$ is the unit of $H$, the tensors corresponding to the $\a$-arcs do not contribute to $Z_H$, hence we do not consider them. The tensor corresponding to $\HH=(\S,\{\a\},\{\b_1,\b_2\})$ is thus given as follows:

%Since $H_*(M,R_-;\R)=0$, we orient $\a,\b$ so that $\a\cdot\b=+1$ (see Remark \ref{remark: canonical homology orientation when H is zero}).  Here $H_1(M)$ is canonically isomorphic to $\Z$, once a positive meridian of the knot is chosen (which is equivalent to an orientation of $K$). If the arc $b$ is oriented as in Figure \ref{figure: Z invariant of trefoil}, then we take $b^*$ to be the canonical generator of $H_1(M)$. Let $\rho:H_1(M)\to \Z_n\cong G(B^*)$ be mod $n$ reduction, so that $\rho(b^*)(K)=q$ where $q=e^{\frac{2\pi i}{n}}$. Let's compute $Z^{\rho}(\HH,\x)$ where the orientations and basepoints of $\HH$ are given in Figure \ref{figure: Z invariant of trefoil}.

\begin{figure}[H]
\centering
\begin{pspicture}(0,-1.3780358)(8.17,1.3780358)
\rput[bl](0.96,-0.23303574){$\Delta_H$}
\rput[bl](5.06,0.56696427){$m_H$}
\rput[bl](5.06,-1.0330358){$m_H$}
\rput[bl](2.76,-0.23303574){$S_H$}
\psline[linecolor=black, linewidth=0.018, arrowsize=0.05291667cm 2.0,arrowlength=0.8,arrowinset=0.2]{->}(1.66,-0.13303573)(2.66,-0.13303573)
\psbezier[linecolor=black, linewidth=0.018, arrowsize=0.05291667cm 2.0,arrowlength=0.8,arrowinset=0.2]{->}(3.36,-0.13303573)(4.3242855,-0.13303573)(4.11,0.6669643)(4.86,0.6669642639160156)
\psbezier[linecolor=black, linewidth=0.018, arrowsize=0.05291667cm 2.0,arrowlength=0.8,arrowinset=0.2]{<-}(4.86,0.8669643)(4.06,1.5669643)(2.36,1.3669642)(1.66,0.26696426391601563)
\psbezier[linecolor=black, linewidth=0.018, arrowsize=0.05291667cm 2.0,arrowlength=0.8,arrowinset=0.2]{->}(1.66,0.06696426)(2.06,0.36696425)(2.76,0.6669643)(3.26,0.46696426391601564)(3.76,0.26696426)(3.76,-0.033035737)(3.96,-0.8330357)(4.16,-1.6330358)(4.46,-1.3330357)(4.86,-1.1330358)
\psbezier[linecolor=black, linewidth=0.018, arrowsize=0.05291667cm 2.0,arrowlength=0.8,arrowinset=0.2]{->}(1.66,-0.33303574)(2.36,-0.8330357)(4.36,-0.23303574)(4.86,-0.7330357360839844)
\psbezier[linecolor=black, linewidth=0.018, arrowsize=0.05291667cm 2.0,arrowlength=0.8,arrowinset=0.2]{->}(1.66,-0.53303576)(2.16,-1.1330358)(3.56,-1.8330357)(4.86,0.46696426391601564)
\psline[linecolor=black, linewidth=0.018, arrowsize=0.05291667cm 2.0,arrowlength=0.8,arrowinset=0.2]{->}(5.76,0.6669643)(6.16,0.6669643)
\psline[linecolor=black, linewidth=0.018, arrowsize=0.05291667cm 2.0,arrowlength=0.8,arrowinset=0.2]{->}(5.76,-0.93303573)(6.16,-0.93303573)
\psline[linecolor=black, linewidth=0.018, arrowsize=0.05291667cm 2.0,arrowlength=0.8,arrowinset=0.2]{->}(0.36,-0.13303573)(0.76,-0.13303573)
\rput[bl](6.36,0.5469643){$\mu$}
\rput[bl](6.26,-1.0330358){$\pi_B$}
\rput[bl](0.0,-0.20303574){$\iota$}
\rput[bl](7.32,0.49696428){$\psi(\beta_1^*)$}
\rput[bl](7.3,-1.0930357){$\psi(\beta_2^*)$}
\rput[bl](8.6,-0.13){.}
\psline[linecolor=black, linewidth=0.018, arrowsize=0.05291667cm 2.0,arrowlength=0.8,arrowinset=0.2]{->}(6.76,0.6669643)(7.16,0.6669643)
\psline[linecolor=black, linewidth=0.018, arrowsize=0.05291667cm 2.0,arrowlength=0.8,arrowinset=0.2]{->}(6.76,-0.93303573)(7.16,-0.93303573)
\end{pspicture}

\end{figure}

Recall that $\iota=\frac{1}{n}(\sum_{i=0}^{n-1} K^iX)$. To compute the above tensor, we rely on Lemma \ref{lemma: multiply by b on lower handlebody does not affect result} below: each $K^i$ accompanying $X$ is group-like and central in $H$, hence it has the effect of multiplying the whole tensor by $\lb \psi(\ov{\a}),K\rb^i$. Since $\ov{\a}=1$ in $H_1(M)$, it follows that each $K^i$ does nothing, hence it suffices to compute the above tensor with $X$ in place of the cointegral $\iota$. Now, we have
\begin{align*}
\De_H^{(5)}(X)=X\ot 1\ot\dots \ot 1+K\ot X\ot 1\ot\dots \ot 1+\dots + K\ot\dots \ot K\ot X.
\end{align*}
where there are five terms in the sum. Since $\mu$ vanishes on $B\subset H$, the only terms of this coproduct that contribute to $Z$ are those for which $X$ lies over $\b$ and there are three of them (the first, third and fifth in the above expansion of $\De_H^{(5)}(X)$). Hence, one computes
\begin{align*}
Z_H^{\rho}(\HH)&=\psi(\b_1^*)\ot\psi(\b_2^*)(\mu(X)\ot 1+\mu(KS(X))\ot K+\mu(KK^{-1}X)\ot K^2)\\
&=1-\psi(\b_2^*)(K)+\psi(\b_2^*)(K^2)\\
&= 1-q+q^2.
\end{align*}
Note that varying $n$ this defines a polynomial in a variable $q$, namely, $1-q+q^2$ which is the Alexander polynomial of the knot up to a unit. See Corollary \ref{corollary: Z0 recovers multivariable Alexander polynomial} below.
\end{example}

\begin{remark}
The above example anticipates the proof of Theorem \ref{thm: intro Z recovers Reidemeister torsion} (or rather, the Fox calculus formula of Theorem \ref{thm: intro Z via Fox calculus}). Indeed, with the orientations and basepoint as above, one has $$ \ov{\a}=\b_1^*\b_2^*(\b_1^*)^{-1}\b_2^*\b_1^*.$$
%Since $\ov{\a}=1$ in $H_1(M)$ this gives $\b^*b^*=1$ in $H_1M$. 
The Fox derivative (see Subsection \ref{subs: Fox calculus}) is
\begin{align*}
\frac{\p\ov{\a}}{\p \b_1^*}&=1-\b_1^*\b_2^*(\b_1^*)^{-1}+\b_1^*\b_2^*(\b_1^*)^{-1}\b_2^*=1-\b_2^*+(\b_2^*)^{2} \\
\end{align*}
when considered as an element of $\Z[H_1(M)]$. Evaluating this on the character $\rho:H_1(M)\to \C^{\t}$ defined by $\rho(\b_2^*)=q$ gives $Z_{H_n}^{\rho}(\HH)$ as computed above therefore satisfying the conclusion of Theorem \ref{thm: intro Z via Fox calculus}.
\end{remark}

\subsection[Basepoints]{Basepoints}
\label{subsection: basepoints}

Whenever $b\neq 1_B$ or $a^*\neq\e_A$, the scalar $Z_H^{\rho}(\HH)$ defined above depends on the basepoints on $\HH$ (see Proposition \ref{prop: changing basepoints and Z} for a precise statement), thus, it is not necessarily a topological invariant of $(M,\c)$. To correct this, we explain now how to put basepoints on the curves in $\aa\cup\bb$ in a coherent way.

\begin{definition}
\label{def: basepoints from multipoint}
Let $\HH$ be an oriented extended Heegaard diagram. Let $\x\in\Tab$ be a multipoint, say, $\x=\{ x_1,\dots,x_d \}$ with $x_i\in\a_i\cap\b_i$ for each $i$ (where $d=|\aa|=|\bb|$). We define basepoints $p_i=p_i(\x)$ on each $\a_i$ and $q_i=q_i(\x)$ on each $\b_i$ as follows: if the crossing $x_i$ is positive (resp. negative), then we put $p_i$ and $q_i$ right before $x_i$ (resp. after) when following the orientation of $\a_i$ and $\b_i$, see Figure \ref{fig: basepoint convention}.
\end{definition}

%Our invariant will only depend on $\x$ up to the $\Spinc$ structure $s_z(\x)$ it defines (note that this is why we took our Heegaard diagram to be based). 

\begin{figure}[h]
\centering
\begin{pspicture}(0,-1.4334792)(6.94,1.52)
\definecolor{colour0}{rgb}{0.0,0.8,0.2}
\psline[linecolor=colour0, linewidth=0.04, arrowsize=0.05291667cm 2.0,arrowlength=1.4,arrowinset=0.0]{<-}(1.2,1.5265208)(1.2,-0.8734791)
\psline[linecolor=red, linewidth=0.04, arrowsize=0.05291667cm 2.0,arrowlength=1.4,arrowinset=0.0]{->}(0.0,0.32652083)(2.4,0.32652083)
\psdots[linecolor=black, fillstyle=solid, dotstyle=o, dotsize=0.2, fillcolor=white](1.2,0.32652083)
\psdots[linecolor=black, dotsize=0.2](0.8,0.32652083)
\rput[bl](1.08,-1.4334792){$\beta_i$}
\psdots[linecolor=black, dotsize=0.2](1.2,-0.07347915)
\rput[bl](1.6,-0.29347914){$q_i$}
\rput[bl](0.46,0.64652085){$p_i$}
\rput[bl](2.64,0.22652084){$\alpha_i$}
\psline[linecolor=colour0, linewidth=0.04, arrowsize=0.05291667cm 2.0,arrowlength=1.4,arrowinset=0.0]{<-}(5.2,1.5265208)(5.2,-0.8734791)
\psline[linecolor=red, linewidth=0.04, arrowsize=0.05291667cm 2.0,arrowlength=1.4,arrowinset=0.0]{<-}(4.0,0.32652083)(6.4,0.32652083)
\psdots[linecolor=black, dotstyle=o, dotsize=0.2, fillcolor=white](5.2,0.32652083)
\psdots[linecolor=black, dotsize=0.2](4.8,0.32652083)
\rput[bl](5.06,-1.4334792){$\beta_i$}
\rput[bl](5.64,0.6265209){$q_i$}
\rput[bl](4.54,-0.23347916){$p_i$}
\rput[bl](6.64,0.22652084){$\alpha_i$}
\psdots[linecolor=black, dotsize=0.2](5.2,0.72652084)
\end{pspicture}

\caption{Basepoints on $\aa\cup\bb$ coming from $\x\in\Tab$. The point $x_i\in\a_i\cap\b_i$ is represented by a white dot, and the basepoints $p_i,q_i$ by black dots. We have depicted both types of crossings, a positive one on the left and a negative one on the right. The surface is oriented as $\S=\p U_{\a}$ (i.e. the normal vector points to the reader)}
\label{fig: basepoint convention}
\label{figure: reversing orientation of alpha changes basepoints}
\end{figure}
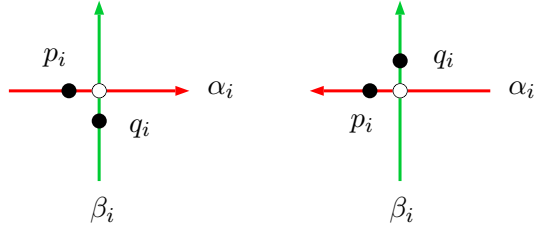

\begin{remark}
\label{remark: basepoint convention under or reversal}
%Let $\HH$ be an oriented diagram with the basepoints on $\aa\cup\bb$ coming from $\x$ as described above. If the crossing $x_i$ is positive, then $x_i$ is the first crossing of $\a_i$ (resp. $\b_i$) in the order of $\Iaaa$ (resp. $\Ibbb$). Now, if the crossing $x_i$ is negative, $x_i$ becomes the {\em last} crossing of $\a_i$ (resp. $\b_i$) in $\Iaaa$ (resp. $\Ibbb$), see Figure \ref{fig: basepoint convention}. Thus, reversing the orientation of a curve also changes a basepoint. 
This rule for the basepoints is set to mimic the condition $$\mu(bx)=\pm S_B\circ\mu\circ S_H(x)$$
of Definition \ref{def: compatible integral cointegral}. Indeed, moving the basepoints affects the Kuperberg tensors by multiplication by $b$ while reversing the orientation of a curve involves the antipode. This will be essential for the proof of Lemma \ref{lemma: invariance of orientations}.
\end{remark}

%The reason to put one basepoint on the left side and the other on the right side (instead of both on the same side) is the following: if $\HD$ is a Heegaard diagram of $(M,\c)$, then $(-\S,\bb,\aa)$ is a Heegaard diagram of $(M,-\c)$ (here $(M,-\c)$ represents the sutured manifold obtained from $(M,\c)$ by reversing the orientation of $s(\c)$). Then if $\x\in\Tab$ is a multipoint, the associated basepoints on both diagrams are the same since the left side of $\b$ on $\S$ becomes the right side of $\b$ in $-\S$.

%The reason to put basepoints this way is that we can use the dual homology classes of Definition \ref{def: dual homology classes} to restate Lemma \ref{lemma: spinc and exy}. \medskip

Now let $\x,\y\in\Tab$ be two multipoints, say, $x_i\in \a_i\cap\b_i$ and $y_i\in\a_{i}\cap\b_{\s(i)}$ for $i=1,\dots,d$. There are two sets of basepoints $\{p_1(\x),\dots,p_d(\x)\}$ and $\{p_1(\y),\dots, p_d(\y)\}$ on the $\a$ curves, where $p_i(\x),p_i(\y)\in \a_i$ for each $i$. Consider the oriented arc $c'_i$ from $p_i(\x)$ to $p_i(\y)$ along $\a_i$ in the direction of its orientation. As in Definition \ref{notation: waaa and wbbb for immersed curve in Sigma} we get an element $\ov{c'_i}\in H_1(M)$ for each $i=1,\dots, d$, that is, $\ov{c'_i}$ is the product of the dual homology classes of the curves and arcs in $\bbb$ encountered as one follows $c'_i$, with appropriate signs. Similarly, there are basepoints $q_i(\x),q_i(\y)\in\b_i$ for each $i=1,\dots,d$, let $d'_i$ be the oriented arc from $q_i(\x)$ to $q_i(\y)$ along $\b_i$. There is an element $\ov{d'_i}\in H_1(M)$ which is a product of dual homology classes of $\aaa$.

\begin{lemma}
\label{lemma: exy change of basepoints}
For any $\x,\y\in\Tab$ one has
\begin{align*}
\e(\x,\y)=\prod_{i=1}^d\ov{c'_i}=\prod_{i=1}^d\ov{d'_i}^{-1}
\end{align*}
in $H_1(M)$, where $\e(\x,\y)$ is the homology class of Definition \ref{def: exy}.
\end{lemma}

\begin{proof} 
For each $i$ let $c_i$ be the oriented subarc of $\a_i$ starting at $x_i$ and ending at $y_i$ and let $d_i$ be the oriented subarc of $\b_i$ from $y_{\s^{-1}(i)}$ to $x_i$, so that $\cup_{i=1}^d(c_i\cup d_i)$ represents $\e(\x,\y)$ (see Definition \ref{def: exy}). Let $e_i$ be an arc parallel to $d_i$ on the left side of $\b_i$ from $p_{\s^{-1}(i)}(\y)$ to $p_i(\x)$. It is clear that $\cup_{i=1}^d(c_i\cup d_i)$ is isotopic to $\cup_{i=1}^d(c'_i\cup e_i)$ (as oriented 1-submanifolds of $\S$): each $c_i\cup d_i\cup c_{\s^{-1}(i)}$ can be pushed off to the left side of $\b_i$ to match $c'_i\cup e_i\cup c'_{\s^{-1}(i)}$. Thus $\e(\x,\y)$ is represented by the 1-submanifold $\bigcup_{i=1}^d(c_i'\cup e_i)$, which is transversal to $\bbb$.  By Remark \ref{remark: ovl equals homology class of l if cycle}
\begin{align*}
\e(\x,\y)&=\ov{\cup_{i=1}^d(c_i'\cup e_i)}\\
&= \ov{\cup_{i=1}^d c_i'}\\
&= \prod_{i=1}^d\ov{c'_i}
\end{align*}
where we used $\ov{e_i}=1$ since the $e_i$'s are disjoint from $\bbb$. The expression using the dual homology classes in $\aaa$ is proved similarly noting that $d'_i$ and $d_i$ have opposite orientations.

\end{proof}

\subsection[Construction of $I_H$]{Construction of $I_H$}
\label{subs: Construction of Z_H}

Let $(M,\c)$, $\o$ and $\rho$ be as before and let $\HH$ be an ordered, oriented, extended Heegaard diagram of $(M,\c)$. Pick a multipoint $\x$ of $\HH$. Whenever $\HH$ is based according to $\x$, we denote the tensor of Definition \ref{def: ZHrho and ZHrho with orientation} by $Z_H^{\rho}(\HH,\x)$ (resp. $Z_H^{\rho}(\HH,\x,\o)$ if the orientation $\o$ is taken into account). 
\medskip

Now let $\ss\in\Spinc(M,\c)$. Recall that $\Spinc(M,\c)$ is an affine space over $H^2(M,\p M)$ (Proposition \ref{prop: spinc is affine space over Htwo}) so there is an homology class $\ss-s(\x)\in H^2(M,\p M)$ where $s$ is the map of Definition \ref{def: sx of OSz}. Let $h_{\ss,\x}$ be the class in $H_1(M)$ defined by $$h_{\ss,\x}\eq PD[\ss-s(\x)]$$ where $PD:H^2(M,\p M)\to H_1(M)$ is Poincar\'e duality. 

\begin{definition}
\label{def: hx and scalar zeta ss x}
We let
\begin{align*}
\zeta_{\ss,\x}\eq \lb \rho(h_{\ss,\x}),(a^*)^{-1}\ot b\rb=\lb (a^*)^{-1},\v(h_{\ss,\x})\rb\lb \psi(h_{\ss,\x}),b\rb\in\kk^{\t}
\end{align*}
where $\rho=(\v,\psi)$ under the isomorphism $G(A\ot B^*)\cong G(A)\ot G(B^*)$. Here $\lb \, , \, \rb$ denotes the corresponding evaluation pairings.
\end{definition}

\medskip

\begin{theorem}
\label{thm: Z is an invariant}
Let $(M,\c)$ be a balanced sutured manifold with connected $R_-(\c)$ endowed with $\ss\in\Spinc(M,\c)$, an orientation $\o$ of $H_*(M,R_-(\c);\R)$ and a group homomorphism $\rho:H_1(M)\to G(A\ot B^*)$. Then the scalar
\begin{align*}
\zeta_{\ss,\x} Z_H^{\rho}(\HH,\x,\o)
\end{align*}
is independent of all the choices made in its definition and defines a topological invariant of the tuple $(M,\c,\ss,\o,\rho)$.
\end{theorem}

We will prove this theorem in Section \ref{section: proof of invariance}. The case when $R_-$ is disconnected is treated in Subsection \ref{subs: disconnected case} below. Theorem \ref{thm: intro Z is an invariant} is the combination of Theorem \ref{thm: Z is an invariant} together with Proposition \ref{prop: Z for disconnected case}.

\begin{definition}
Using the above notation and assuming $R_-$ connected, we define
\begin{align*}
I_H^{\rho}(M,\c,\ss,\o)\eq\zeta_{\ss,\x} Z_H^{\rho}(\HH,\x,\o)
\end{align*}
where $\HH$ is any ordered, oriented, extended Heegaard diagram of $(M,\c)$ (based according to $\x$). This is well-defined by the above theorem.
\end{definition}

%\begin{example}If $K$ is the left trefoil, the presentation of $\pi_1(S^3\sm K)$ associated to the Heegaard diagram of Example \ref{example: sutured HD for link complement} (with orientations and basepoints as in Example \ref{example: computation of Z of left trefoil}) is\begin{align*}\lb \b^*, b^* \ | \ \b^*b^*(\b^*)^{-1}b^*\b^*(b^*)^{-1}\rb.\end{align*}Note that this gives $\b^*b^*=1$ in $H_1M$. The Fox derivative is\begin{align*}\frac{\p\ov{\a}}{\p \b^*}&=1-\b^*b^*(\b^*)^{-1}+\b^*b^*(\b^*)^{-1}b^* \\\end{align*}so in $H_1(M)$ this becomes\begin{align*}\frac{\p\ov{\a}}{\p \b^*}&=1-b^*+(b^*)^{2}.\end{align*}Evaluating this on the character $\rho:H_1(M)\to \C^{\t}$ defined by $\rho(b^*)=e^{\frac{2\pi i}{n}}$ gives $Z(\HH,\x_0)$ as computed in Example \ref{example: computation of Z of left trefoil} therefore satisfying the conclusion of Theorem \ref{thm: intro Z via Fox calculus}\end{example}

\subsection[The disconnected case]{The disconnected case}
\label{subs: disconnected case} 

Assuming Theorem \ref{thm: Z is an invariant}, we show how to extend the definition of $I_H$ to arbitrary balanced sutured manifolds (i.e. with possibly disconnected $R_-$). 
\medskip

Let $(M,\c)$ be a balanced sutured manifold. Then we can construct a balanced sutured manifold $(M',\c')$ containing $M$ and with connected $R'=R_-(\c')$ as follows (see \cite[Section 3.6]{FJR11}). First, attach a 2-dimensional 1-handle $h$ along $s(\c)$. This handle can be thickened to a 3-dimensional 1-handle $h\t I$ attached to $\c=s(\c)\t I$. This produces a new balanced sutured manifold, and after sufficiently many handle attachments, we get $(M',\c')$ with connected $R'$. For each $\ss\in\Spinc(M,\c)$ we let $i(\ss)\in \Spinc(M',\c')$ be defined by $i(\ss)|_M=\ss$ and let $i(\ss)$ be the gradient of the height function $h\t I\to I$ on each 2-dimensional 1-handle $h$ attached to $R$. Note that the inclusion map induces an isomorphism
\begin{align*}
H_*(M,R;\R)\cong H_*(M',R';\R)
\end{align*}
so given an orientation $\o$ of $H_*(M,R;\R)$, we denote by $\o'$ the corresponding orientation of $H_*(M',R';\R)$. Now, given a representation 
$\rho:H_1(M)\to G(A\ot B^*)$ we let $\rho':H_1(M')\to G(A\ot B^*)$ be an arbitrary extension, i.e., we choose $\rho'$ such that $\rho'\circ i_*=\rho$ where $i_*:H_1(M)\to H_1(M')$ is the homomorphism induced by inclusion.

\begin{lemma}
If $R=R_-(\c)$ is connected, then
\begin{align*}
I_H^{\rho}(M,\c,\ss,\o)=I_H^{\rho'}(M',\c',i(\ss),\o')
\end{align*}
where $(M',\c')$ is any balanced sutured manifold constructed from $(M,\c)$ as above.
\end{lemma}
\begin{proof}
It suffices to suppose $(M',\c')$ is obtained by adding a single 3-dimensional 1-handle to $R_-\t I$. Then, an extended Heegaard diagram of $M'$ is obtained from an extended diagram $\HH=\HDD$ of $M$ by attaching a 2-dimensional 1-handle $h$ to $\S$ along $\p\S$ and letting $\bolda'=\bolda\cup\{a\}, \boldb'=\boldb\cup\{b\}$, where both $a$ and $b$ are arcs parallel to the cocore of the 1-handle attached (as usual we note $\aaa=\aa\cup\bolda$ and $\bbb=\bb\cup\boldb$). If $\S'$ denotes the surface $\S$ with $h$ attached, then $(\S',\aa,\bb,\bolda',\boldb')$ is an extended diagram of $(M',\c')$. Note that with such Heegaard diagrams, one has 
an identification $\Tab=\Tabprime$ and we note $\x'$ the multipoint in $\Tabprime$ corresponding to $\x\in\Tab$. Since the cocore of $h$ does not intersects any curve or arc of $\HH$ it follows from the definition that
\begin{align*}
Z^{\rho}(\HH,\x,\o)=Z^{\rho'}(\HH',\x',\o').
\end{align*}
Moreover, since the map $i:\Spinc(M,\c)\to \Spinc(M',\c')$ defined above is an affine map and satisfies $i(s(\x))=s'(\x')$ (see \cite[Section 3.6]{FJR11}), one has 
\begin{align*}
PD[h_{i(\ss),\x'}]&=i(\ss)-s'(\x')\\
&=i(\ss)-i(s(\x)\\
&= i(\ss-s(\x)))\\
&=i_*(PD[h_{\ss,\x}])
\end{align*}
which implies $\zeta_{\ss,\x}=\zeta_{i(\ss),\x'}$ since $\rho'\circ i_*=\rho$. Thus
\begin{align*}
I_H^{\rho}(M,\c,\ss,\o)=I_H^{\rho'}(M',\c',i(\ss),\o')
\end{align*}
as desired.
\end{proof}

\begin{proposition}
\label{prop: Z for disconnected case}
Let $(M,\c)$ be a balanced sutured manifold with possibly disconnected $R=R_-(\c)$. Let $(M',\c')$ be any sutured manifold with connected $R'=R_-(\c')$ constructed as above. The scalar $I_H^{\rho'}(M',\c',i(\ss),\o')$ is independent of how the one-handles are attached to $R$, provided $R'$ is connected.
\end{proposition}
\begin{proof}
Suppose $(M',\c')$ and $(M'',\cc'')$ are obtained from $(M,\c)$ by attaching 1-handles to $R$ as above and that $R',R''$ are connected. 
One can keep attaching 1-handles to find a sutured manifold $(M_0,\c_0)$ containing both $M'$ and $M''$. By the lemma above, it follows that 
$I(M')=I(M_0)=I(\HH')$ as desired.
\end{proof}

\begin{definition}
\label{def: Z for disconnected R}
Let $(M,\c)$ be a connected balanced sutured manifold with possibly disconnected $R=R_-(\c)$. We define
\begin{align*}
I_H^{\rho}(M,\c,\ss,\o)\eq I_H^{\rho'}(M',\c',i(\ss),\o')
\end{align*}
where $(M',\c')$ is any balanced sutured manifold obtained by adding 1-handles to $M$ as above, $\rho'$ is an arbitrary extension of $\rho$ to $H_1(M')$ and $\o'$ is the orientation of $H_*(M',R';\R)$ corresponding to $\o$ under the isomorphism induced by inclusion. If $M$ is disconnected with connected components $M_1,\dots,M_m$ then we let
\begin{align*}
I_H^{\rho}(M,\c,\ss,\o)\eq\prod_{i=1}^m I_H^{\rho_i}(M_i,\c_i,\ss_i,\o_i)
\end{align*}
where $\c_i=\c\cap M_i, \rho_i=\rho|_{H_1(M_i)}, \ss_i=\ss |_{M_i},\o_i=\o|_{H_*(M_i,R_-(\c_i);\R)}$ for each $i=1,\dots,m$.
\end{definition}

\section{Proof of invariance}
\label{section: proof of invariance}

In this section we prove Theorem \ref{thm: Z is an invariant}. The main difference between our proof and that of Kuperberg \cite{Kup1} is that the tensors we associate to the closed curves of a Heegaard diagram are not necessarily cyclic. Thus, $Z_H^{\rho}(\HH)$ depends on the basepoints of $\HH$, but we find a nice basepoint-changing formula as in Proposition \ref{prop: changing basepoints and Z} below. We will use the same notation as in the previous section. Thus, $H$ is an involutive Hopf superalgebra endowed with a compatible relative integral and cointegral with distinguished group-likes $b\in G(B)$ and $a^*\in G(A^*)$. We also let $\HH=\HDD$ be an ordered, oriented, based extended Heegaard diagram of a balanced sutured 3-manifold $(M,\c)$ (with connected $R_-(\c)$) and let $\rho=(\v,\psi):H_1(M)\to G(A\ot B^*)$ be a group homomorphism.

 %Recall that we note $m_b:H\to H, h\mapsto i_B(b)x$ and $\De_{a^*}:H\to H, h\mapsto a^*(\pi_A(h_{(1)}))h_{(2)}$.

\begin{lemma}
\label{lemma: multiply by b on lower handlebody does not affect result}
For any $\b\in\bbb$, composing an incoming leg of a $\b$-tensor with $m_{b^{\pm 1}}$ has the effect of multiplying that tensor by $\lb \psi((\b^*)^{\pm 1}),b\rb$. Graphically, this is
\begin{figure}[H]
\centering
\begin{pspicture}(0,-0.50498873)(11.68,0.50498873)
\rput[bl](0.5,-0.11){$m_{b^{\pm 1}}$}
\rput[bl](2.1,-0.09){$m_H$}
\psline[linecolor=black, linewidth=0.018, arrowsize=0.05291667cm 2.0,arrowlength=0.8,arrowinset=0.2]{->}(2.8,0.0)(3.1,0.0)
\psline[linecolor=black, linewidth=0.018, arrowsize=0.05291667cm 2.0,arrowlength=0.8,arrowinset=0.2]{->}(3.6,0.0)(3.9,0.0)
\rput[bl](3.24,-0.11){$\mu$}
\rput[bl](4.02,-0.18){$\psi(\beta^*)$}
\psline[linecolor=black, linewidth=0.018, arrowsize=0.05291667cm 2.0,arrowlength=0.8,arrowinset=0.2]{->}(0.0,0.0)(0.3,0.0)
\rput[bl](9.52,-0.18){$\cdot\langle\psi((\beta^*)^{\pm 1}),b\rangle$}
\rput[bl](5.23,-0.04){=}
\rput[bl](6.5,-0.09){$m_H$}
\psline[linecolor=black, linewidth=0.018, arrowsize=0.05291667cm 2.0,arrowlength=0.8,arrowinset=0.2]{->}(7.2,0.0)(7.5,0.0)
\psline[linecolor=black, linewidth=0.018, arrowsize=0.05291667cm 2.0,arrowlength=0.8,arrowinset=0.2]{->}(8.0,0.0)(8.3,0.0)
\rput[bl](7.64,-0.11){$\mu$}
\rput[bl](8.42,-0.18){$\psi(\beta^*)$}
\psline[linecolor=black, linewidth=0.018, arrowsize=0.05291667cm 2.0,arrowlength=0.8,arrowinset=0.2]{->}(1.4,0.0)(2.0,0.0)
\psline[linecolor=black, linewidth=0.018, arrowsize=0.05291667cm 2.0,arrowlength=0.8,arrowinset=0.2]{->}(1.8,-0.5)(2.0,-0.2)
\psline[linecolor=black, linewidth=0.018, arrowsize=0.05291667cm 2.0,arrowlength=0.8,arrowinset=0.2]{->}(1.8,0.5)(2.0,0.2)
\psarc[linecolor=black, linewidth=0.026, linestyle=dotted, dotsep=0.10583334cm, dimen=outer](2.15,-0.03){0.6}{135.0}{170.0}
\psarc[linecolor=black, linewidth=0.026, linestyle=dotted, dotsep=0.10583334cm, dimen=outer](2.1766667,-0.06){0.6}{-170.0}{-135.0}
\psline[linecolor=black, linewidth=0.018, arrowsize=0.05291667cm 2.0,arrowlength=0.8,arrowinset=0.2]{->}(5.8,0.0)(6.4,0.0)
\psline[linecolor=black, linewidth=0.018, arrowsize=0.05291667cm 2.0,arrowlength=0.8,arrowinset=0.2]{->}(6.2,-0.5)(6.4,-0.2)
\psline[linecolor=black, linewidth=0.018, arrowsize=0.05291667cm 2.0,arrowlength=0.8,arrowinset=0.2]{->}(6.2,0.5)(6.4,0.2)
\psarc[linecolor=black, linewidth=0.026, linestyle=dotted, dotsep=0.10583334cm, dimen=outer](6.55,-0.03){0.6}{135.0}{170.0}
\psarc[linecolor=black, linewidth=0.026, linestyle=dotted, dotsep=0.10583334cm, dimen=outer](6.576667,-0.06){0.6}{-170.0}{-135.0}
\end{pspicture}
\end{figure}
\noindent as well as with $\pi_B$ in place of $\mu$. Similarly, for any $\a\in\aaa$, composing an outcoming leg of an $\a$-tensor with $\De_{(a^*)^{\pm 1}}$ has the effect of multiplying that tensor by $\lb a^*,\v((\a^*)^{\pm 1})\rb$.
\end{lemma}
\begin{proof}
We think of $B$ as a central subalgebra of $H$ via $i_B$. Then, this is an immediate consequence of the fact that $b^{\pm 1}\in H$ is central, $\mu$ (or $\pi_B$) is $B$-linear and $\psi(\b^*)$ is group-like. The second assertion follows similarly.
\end{proof}

\begin{proposition}
\label{prop: changing basepoints and Z}
Let $\HH$ be an ordered, oriented, based, extended Heegaard diagram. Let $\a\in\aa$ be a closed curve and $p\in\a$ be its basepoint. Then changing the basepoint $p$ to another basepoint $p'\in\a$ has the effect of multiplying $Z_H^{\rho}(\HH)$ by $\lb \psi(\ov{c}),b\rb$ where $c$ is the arc from $p'$ to $p$ along $\a$ (following its orientation). Similarly, if $\b\in\bb$ is closed and $q\in\b$ is its basepoint, moving $q$ to another basepoint $q'$ multiplies $Z_H^{\rho}(\HH)$ by $\lb a^*,\v(\ov{c})\rb$ where $c$ is the arc from $q'$ to $q$ along $\b$.
\end{proposition}
\begin{proof}

We prove only the first assertion, the second is analogous. First we claim that, after moving the basepoint to $p'$, the tensor corresponding to $\a$ is obtained from the original one by applying $m_b$ to all the crossings along $c$. If $k$ is the number of crossings along $\a$ from $p$ to $p'$ and $l$ is the number of crossings along $c$, then this is shown as follows (we denote $\De=\De_H$):

\begin{figure}[H]
\centering
\begin{pspicture}(0,-2.435)(11.95,2.435)
\rput[bl](5.65,1.765){=}
\rput[bl](9.81,2.115){$\Delta$}
\psline[linecolor=black, linewidth=0.018, arrowsize=0.05291667cm 2.0,arrowlength=0.8,arrowinset=0.2]{->}(10.25,2.315)(10.55,2.415)
\psline[linecolor=black, linewidth=0.018, arrowsize=0.05291667cm 2.0,arrowlength=0.8,arrowinset=0.2]{->}(10.25,2.115)(10.55,2.015)
\rput[bl](9.81,1.315){$\Delta$}
\psline[linecolor=black, linewidth=0.018, arrowsize=0.05291667cm 2.0,arrowlength=0.8,arrowinset=0.2]{->}(10.25,1.515)(10.55,1.615)
\psline[linecolor=black, linewidth=0.018, arrowsize=0.05291667cm 2.0,arrowlength=0.8,arrowinset=0.2]{->}(10.25,1.315)(10.55,1.215)
\rput[bl](10.65,2.015){$\big\}$}
\rput[bl](10.65,1.215){$\big\}$}
\rput[bl](10.95,2.115){$k$}
\rput[bl](10.95,1.315){$l$}
\rput[bl](5.65,0.065){=}
\rput[bl](10.41,0.415){$\Delta$}
\psline[linecolor=black, linewidth=0.018, arrowsize=0.05291667cm 2.0,arrowlength=0.8,arrowinset=0.2]{->}(10.85,0.615)(11.15,0.715)
\psline[linecolor=black, linewidth=0.018, arrowsize=0.05291667cm 2.0,arrowlength=0.8,arrowinset=0.2]{->}(10.85,0.415)(11.15,0.315)
\rput[bl](11.25,0.315){$\big\}$}
\rput[bl](11.55,0.415){$k$}
\psline[linecolor=black, linewidth=0.018, arrowsize=0.05291667cm 2.0,arrowlength=0.8,arrowinset=0.2]{->}(9.05,0.215)(10.25,0.415)
\psline[linecolor=black, linewidth=0.018, arrowsize=0.05291667cm 2.0,arrowlength=0.8,arrowinset=0.2]{->}(9.05,0.015)(9.35,-0.085)
\psline[linecolor=black, linewidth=0.018, arrowsize=0.05291667cm 2.0,arrowlength=0.8,arrowinset=0.2]{->}(1.95,1.815)(2.25,1.815)
\rput[bl](2.41,1.715){$\Delta$}
\rput[bl](3.31,2.115){$\Delta$}
\rput[bl](3.31,1.315){$\Delta$}
\rput[bl](5.15,2.125){$k$}
\rput[bl](5.15,1.365){$l$}
\rput[bl](1.6,1.765){$\iota$}
\rput[bl](0.0,1.665){$\varphi(\alpha^*)$}
\psline[linecolor=black, linewidth=0.018, arrowsize=0.05291667cm 2.0,arrowlength=0.8,arrowinset=0.2]{->}(1.05,1.815)(1.35,1.815)
\psline[linecolor=black, linewidth=0.018, arrowsize=0.05291667cm 2.0,arrowlength=0.8,arrowinset=0.2]{->}(2.85,1.915)(3.15,2.115)
\psline[linecolor=black, linewidth=0.018, arrowsize=0.05291667cm 2.0,arrowlength=0.8,arrowinset=0.2]{->}(2.85,1.715)(3.15,1.515)
\rput[bl](4.85,2.025){$\}$}
\psline[linecolor=black, linewidth=0.018, arrowsize=0.05291667cm 2.0,arrowlength=0.8,arrowinset=0.2]{->}(3.75,2.115)(4.75,1.315)
\psline[linecolor=black, linewidth=0.018, arrowsize=0.05291667cm 2.0,arrowlength=0.8,arrowinset=0.2]{->}(3.75,2.315)(4.75,1.515)
\rput[bl](4.85,1.265){$\}$}
\rput[bl](9.45,-0.335){$m_b$}
\psline[linecolor=black, linewidth=0.018, arrowsize=0.05291667cm 2.0,arrowlength=0.8,arrowinset=0.2]{->}(10.05,-0.285)(10.35,-0.385)
\rput[bl](10.51,-0.485){$\Delta$}
\psline[linecolor=black, linewidth=0.018, arrowsize=0.05291667cm 2.0,arrowlength=0.8,arrowinset=0.2]{->}(10.95,-0.285)(11.25,-0.185)
\psline[linecolor=black, linewidth=0.018, arrowsize=0.05291667cm 2.0,arrowlength=0.8,arrowinset=0.2]{->}(10.95,-0.485)(11.25,-0.585)
\rput[bl](11.35,-0.585){$\big\}$}
\rput[bl](11.65,-0.485){$l$}
\rput[bl](5.65,-1.735){=}
\rput[bl](9.51,-1.385){$\Delta$}
\psline[linecolor=black, linewidth=0.018, arrowsize=0.05291667cm 2.0,arrowlength=0.8,arrowinset=0.2]{->}(9.95,-1.185)(10.25,-1.085)
\psline[linecolor=black, linewidth=0.018, arrowsize=0.05291667cm 2.0,arrowlength=0.8,arrowinset=0.2]{->}(9.95,-1.385)(10.25,-1.485)
\rput[bl](9.51,-2.185){$\Delta$}
\psline[linecolor=black, linewidth=0.018, arrowsize=0.05291667cm 2.0,arrowlength=0.8,arrowinset=0.2]{->}(9.95,-1.985)(10.25,-1.885)
\psline[linecolor=black, linewidth=0.018, arrowsize=0.05291667cm 2.0,arrowlength=0.8,arrowinset=0.2]{->}(9.95,-2.185)(10.25,-2.285)
\rput[bl](10.35,-1.485){$\big\}$}
\rput[bl](10.65,-1.385){$k$}
\psline[linecolor=black, linewidth=0.018, arrowsize=0.05291667cm 2.0,arrowlength=0.8,arrowinset=0.2]{->}(9.05,-1.585)(9.35,-1.385)
\psline[linecolor=black, linewidth=0.018, arrowsize=0.05291667cm 2.0,arrowlength=0.8,arrowinset=0.2]{->}(9.05,-1.785)(9.35,-1.985)
\rput[bl](10.45,-2.035){$m_b$}
\rput[bl](10.45,-2.435){$m_b$}
\psline[linecolor=black, linewidth=0.018, arrowsize=0.05291667cm 2.0,arrowlength=0.8,arrowinset=0.2]{->}(11.05,-1.885)(11.35,-1.885)
\psline[linecolor=black, linewidth=0.018, arrowsize=0.05291667cm 2.0,arrowlength=0.8,arrowinset=0.2]{->}(11.05,-2.285)(11.35,-2.285)
\rput[bl](11.45,-2.285){$\big\}$}
\rput[bl](11.75,-2.185){$l$}
\psline[linecolor=black, linewidth=0.018, arrowsize=0.05291667cm 2.0,arrowlength=0.8,arrowinset=0.2]{->}(8.15,1.815)(8.45,1.815)
\rput[bl](7.8,1.765){$\iota$}
\rput[bl](6.2,1.665){$\varphi(\alpha^*)$}
\psline[linecolor=black, linewidth=0.018, arrowsize=0.05291667cm 2.0,arrowlength=0.8,arrowinset=0.2]{->}(7.25,1.815)(7.55,1.815)
\psline[linecolor=black, linewidth=0.018, arrowsize=0.05291667cm 2.0,arrowlength=0.8,arrowinset=0.2]{->}(8.15,0.115)(8.45,0.115)
\rput[bl](8.61,0.015){$\Delta$}
\rput[bl](7.8,0.065){$\iota$}
\rput[bl](6.2,-0.035){$\varphi(\alpha^*)$}
\psline[linecolor=black, linewidth=0.018, arrowsize=0.05291667cm 2.0,arrowlength=0.8,arrowinset=0.2]{->}(7.25,0.115)(7.55,0.115)
\psline[linecolor=black, linewidth=0.018, arrowsize=0.05291667cm 2.0,arrowlength=0.8,arrowinset=0.2]{->}(8.15,-1.685)(8.45,-1.685)
\rput[bl](8.61,-1.785){$\Delta$}
\rput[bl](7.8,-1.735){$\iota$}
\rput[bl](6.2,-1.835){$\varphi(\alpha^*)$}
\psline[linecolor=black, linewidth=0.018, arrowsize=0.05291667cm 2.0,arrowlength=0.8,arrowinset=0.2]{->}(7.25,-1.685)(7.55,-1.685)
\rput[bl](11.9,-2.435){.}
\psline[linecolor=black, linewidth=0.018, arrowsize=0.05291667cm 2.0,arrowlength=0.8,arrowinset=0.2]{->}(3.75,1.515)(4.75,2.315)
\psline[linecolor=black, linewidth=0.018, arrowsize=0.05291667cm 2.0,arrowlength=0.8,arrowinset=0.2]{->}(3.75,1.315)(4.75,2.115)
\rput[bl](8.61,1.715){$\Delta^{\text{op}}$}
\psline[linecolor=black, linewidth=0.018, arrowsize=0.05291667cm 2.0,arrowlength=0.8,arrowinset=0.2]{->}(9.35,1.915)(9.65,2.115)
\psline[linecolor=black, linewidth=0.018, arrowsize=0.05291667cm 2.0,arrowlength=0.8,arrowinset=0.2]{->}(9.35,1.715)(9.65,1.515)
\end{pspicture}
\end{figure}
\medskip

\noindent Here in the second equality we used property (\ref{eq: cointegral cotrace property}) of Definition \ref{def: compatible integral cointegral} and we used that $b$ is group-like in the third equality. The leftmost tensor is the tensor corresponding to $\a$ after moving the basepoint, proving our claim. Now, let $x$ be a crossing through $c$ and denote by $\b_x$ the $\b$ curve or arc passing through $x$. We have

\begin{figure}[H]
\centering
\begin{pspicture}(0,-0.155)(7.6,0.155)
\psline[linecolor=black, linewidth=0.018, arrowsize=0.05291667cm 2.0,arrowlength=0.8,arrowinset=0.2]{->}(0.0,0.045)(0.4,0.045)
\rput[bl](0.65,-0.095){$m_b$}
\rput[bl](1.9,-0.085){$S^{\epsilon_x}$}
\rput[bl](6.15,-0.095){$m_{b^{\pm 1}}$}
\psline[linecolor=black, linewidth=0.018, arrowsize=0.05291667cm 2.0,arrowlength=0.8,arrowinset=0.2]{->}(1.3,0.045)(1.7,0.045)
\psline[linecolor=black, linewidth=0.018, arrowsize=0.05291667cm 2.0,arrowlength=0.8,arrowinset=0.2]{->}(2.6,0.045)(3.0,0.045)
\rput[bl](4.8,-0.085){$S^{\epsilon_x}$}
\psline[linecolor=black, linewidth=0.018, arrowsize=0.05291667cm 2.0,arrowlength=0.8,arrowinset=0.2]{->}(4.2,0.045)(4.6,0.045)
\psline[linecolor=black, linewidth=0.018, arrowsize=0.05291667cm 2.0,arrowlength=0.8,arrowinset=0.2]{->}(5.5,0.045)(5.9,0.045)
\psline[linecolor=black, linewidth=0.018, arrowsize=0.05291667cm 2.0,arrowlength=0.8,arrowinset=0.2]{->}(7.1,0.045)(7.5,0.045)
\rput[bl](3.5,-0.005){=}
%\rput[bl](7.55,-0.155){.}
\end{pspicture}
\end{figure}
\noindent where the $\pm 1$ on the right hand side is the sign of the crossing $x$. Thus, moving the basepoint $p\in\a$ to $p'$ has the effect of composing by $m_{b^{\pm 1}}$ every incoming leg of a $\b$-tensor that corresponds to a crossing of $c$. By Lemma \ref{lemma: multiply by b on lower handlebody does not affect result}, this has the effect of multiplying the scalar $Z_H^{\rho}(\HH)$ by
\begin{align*}
\prod_{x\in c}\lb \psi((\b_x^*)^{\pm 1}),b\rb
\end{align*}
where $\pm 1$ is the sign at $x$. This is exactly $\lb \psi(\ov{c}),b\rb$ as was to be shown.

\end{proof}

Now suppose $\HH$ is ordered, oriented and based according to a multipoint $\x\in\Tab$. As before, we denote by $Z_H^{\rho}(\HH,\x)$ the corresponding tensor (or $Z_H^{\rho}(\HH,\x,\o)$ if we take $\o$ into account). Given $\ss\in\Spinc(M,\c)$, recall that we defined $\zeta_{\ss,\x}\in\kk^{\t}$ by $$\zeta_{\ss,\x}\eq \lb (a^*)^{-1},\v(h_{\ss,\x})\rb\lb \psi(h_{\ss,\x}),b\rb$$ where $h_{\ss,\x}\eq PD[\ss-s(\x)]\in H_1(M)$.

\medskip

\begin{lemma}
\label{lemma: invariance of multipoint}
%Let $\HH$ be a diagram of $(M,\c)$ as above and let $\ss\in\Spinc(M,\c)$. 
Let $\ss\in\Spinc(M,\c)$ be fixed. Then the scalar $$\zeta_{\ss,\x}Z_H^{\rho}(\HH,\x)$$ is independent of the multipoint $\x\in\Tab$.
\end{lemma}

\begin{proof}
Let $\x,\y\in\Tab$, say $x_i\in\a_i\cap\b_i$ and $y_i\in\a_i\cap\b_{\s(i)}$ for each $i=1,\dots,d$, where $\s\in S_d$. We thus have basepoints $p_i(\x),p_i(\y)\in \a_i$ and $q_i(\x),q_i(\y)\in\b_i$ for each $i=1,\dots,d$. We let $c'_i$ be the arc from $p_i(\x)$ to $p_i(\y)$ along $\a_i$ and let $d'_i$ be the arc from $q_i(\x)$ to $q_i(\y)$ along $\b_i$. Let $\ov{\HH}$ be the based diagram with basepoints the $p_i(\y)$'s on the $\a$ curves and the $q_i(\x)$'s on the $\b$ curves (and let $\ov{\HH}$ have the same ordering and orientations of $\HH$). Using Proposition \ref{prop: changing basepoints and Z} over each alpha curve we get
\begin{align*}
Z_H^{\rho}(\HH,\x)=\left\langle \prod_{i=1}^d\psi(\ov{c}'_i),b\right\rangle Z_H^{\rho}(\ov{\HH})=\left\langle\psi(\e(\x,\y)),b\right\rangle Z_H^{\rho}(\ov{\HH})
\end{align*}
where we use Lemma \ref{lemma: exy change of basepoints} in the second equality. Similarly, we have
\begin{align*}
Z_H^{\rho}(\ov{\HH})=\left\langle a^*,\prod_{i=1}^d\v(\ov{d}'_i)\right\rangle Z_H^{\rho}(\HH,\y)=\left\langle a^*, \v(\e(\y,\x))\right\rangle Z_H^{\rho}(\HH,\y).
\end{align*}
Combining these, we get 
\begin{align*}
Z_H^{\rho}(\HH,\y)&=\left\langle a^*, \v(\e(\x,\y))\right\rangle\left\langle \psi(\e(\x,\y)),b^{-1}\right\rangle Z_H^{\rho}(\HH,\x)\\
&=\left\langle\rho(\e(\x,\y)),a^*\ot b^{-1} \right\rangle Z_H^{\rho}(\HH,\x).
\end{align*}
Noting that $h_{\ss,\x},h_{\ss,\y}\in H_1(M)$ are related by $h_{\ss,\y}=h_{\ss,\x}+\e(\x,\y)$, we find
\begin{align*}
\zeta_{\ss,\y}=\zeta_{\ss,\x}\lb \rho(\e(\x,\y)), (a^*)^{-1}\ot b\rb
\end{align*}
and we deduce
\begin{align*}
\zeta_{\ss,\y}Z_H^{\rho}(\HH,\y)=\zeta_{\ss,\x}Z_H^{\rho}(\HH,\x)
\end{align*}
as was to be shown.
\end{proof}

\begin{lemma}
\label{lemma: invariance of orientations}
For any $\x\in\Tab$, the scalar $Z_H^{\rho}(\HH,\x,\o)$ is independent of the orientations of $\aaa,\bbb$ chosen.
\end{lemma}
\begin{proof}
We will only prove this for $\aaa$. Suppose first we reverse the orientation of an arc $\a$. Then $\a^*$ gets inversed, the crossings through $\a$ reverse order and change sign. Since $H$ is involutive, the tensor corresponding to a crossing $x$ of $\a$ becomes $S_H\circ S_H^{\e_x}$, where $\e_x$ is the intersection sign before the orientation reversal. The following computation then shows that $Z^{\rho}_H$ is independent of the orientation of $\a$  (we denote $\De=\De_H$ for simplicity):

\begin{figure}[H]
\centering
\begin{pspicture}(0,-0.745)(11.119107,0.745)
\rput[bl](5.45,0.305){=}
\psline[linecolor=black, linewidth=0.018, arrowsize=0.05291667cm 2.0,arrowlength=0.8,arrowinset=0.2]{->}(9.95,0.355)(10.25,0.355)
\rput[bl](6.0,0.205){$\varphi(\alpha^*)$}
\psline[linecolor=black, linewidth=0.018, arrowsize=0.05291667cm 2.0,arrowlength=0.8,arrowinset=0.2]{->}(7.05,0.355)(7.35,0.355)
\rput[bl](5.45,-0.595){=}
\rput[bl](6.0,-0.695){$\varphi(\alpha^*)$}
\psline[linecolor=black, linewidth=0.018, arrowsize=0.05291667cm 2.0,arrowlength=0.8,arrowinset=0.2]{->}(2.65,0.355)(2.95,0.355)
\rput[bl](2.17,0.255){$i_A$}
\rput[bl](0.0,0.205){$\varphi((\alpha^*)^{-1})$}
\psline[linecolor=black, linewidth=0.018, arrowsize=0.05291667cm 2.0,arrowlength=0.8,arrowinset=0.2]{->}(1.75,0.355)(2.05,0.355)
\rput[bl](4.25,0.455){$S_H$}
\rput[bl](4.25,0.055){$S_H$}
\rput[bl](7.45,0.255){$S_A$}
\psline[linecolor=black, linewidth=0.018, arrowsize=0.05291667cm 2.0,arrowlength=0.8,arrowinset=0.2]{->}(8.05,0.355)(8.35,0.355)
\rput[bl](9.35,0.255){$S_H$}
\rput[bl](10.41,0.255){$\Delta$}
\psline[linecolor=black, linewidth=0.018, arrowsize=0.05291667cm 2.0,arrowlength=0.8,arrowinset=0.2]{->}(10.85,0.455)(11.15,0.555)
\psline[linecolor=black, linewidth=0.018, arrowsize=0.05291667cm 2.0,arrowlength=0.8,arrowinset=0.2]{->}(10.85,0.255)(11.15,0.155)
\rput[bl](8.41,-0.645){$\Delta$}
\psline[linecolor=black, linewidth=0.018, arrowsize=0.05291667cm 2.0,arrowlength=0.8,arrowinset=0.2]{->}(8.85,-0.445)(9.15,-0.345)
\psline[linecolor=black, linewidth=0.018, arrowsize=0.05291667cm 2.0,arrowlength=0.8,arrowinset=0.2]{->}(8.85,-0.645)(9.15,-0.745)
\psline[linecolor=black, linewidth=0.018, arrowsize=0.05291667cm 2.0,arrowlength=0.8,arrowinset=0.2]{->}(4.85,0.555)(5.15,0.555)
\psline[linecolor=black, linewidth=0.018, arrowsize=0.05291667cm 2.0,arrowlength=0.8,arrowinset=0.2]{->}(4.85,0.155)(5.15,0.155)
\psline[linecolor=black, linewidth=0.018, arrowsize=0.05291667cm 2.0,arrowlength=0.8,arrowinset=0.2]{->}(8.95,0.355)(9.25,0.355)
\rput[bl](8.47,0.255){$i_A$}
\psline[linecolor=black, linewidth=0.018, arrowsize=0.05291667cm 2.0,arrowlength=0.8,arrowinset=0.2]{->}(7.05,-0.545)(7.35,-0.545)
\psline[linecolor=black, linewidth=0.018, arrowsize=0.05291667cm 2.0,arrowlength=0.8,arrowinset=0.2]{->}(7.95,-0.545)(8.25,-0.545)
\rput[bl](7.47,-0.645){$i_A$}
\rput[bl](9.3,-0.745){.}
\rput[bl](3.11,0.255){$\Delta^{\text{op}}$}
\psline[linecolor=black, linewidth=0.018, arrowsize=0.05291667cm 2.0,arrowlength=0.8,arrowinset=0.2]{->}(3.85,0.455)(4.15,0.555)
\psline[linecolor=black, linewidth=0.018, arrowsize=0.05291667cm 2.0,arrowlength=0.8,arrowinset=0.2]{->}(3.85,0.255)(4.15,0.155)
\end{pspicture}

\end{figure}

\noindent Here the first equality uses that $S_H$ is a coalgebra anti-automorphism and the second uses that $i_A$ preserves antipodes. Now, suppose we reverse the orientation of a closed curve, say $\a_1$. Let $\x=\{x_1,\dots,x_d\}$ with $x_i\in\a_i\cap\b_i$ for $i=1,\dots, d$. As before $\a_1^*$ gets inversed and the crossings through $\a_1$ get reversed and change sign, but also the basepoint of $\b_1$ is crossed through $\a_1$. Let's see first how does the tensor corresponding to $\a_1$ changes:

\begin{figure}[H]
\centering
\begin{pspicture}(0,-1.195)(11.75,1.195)
\rput[bl](5.25,0.755){=}
\rput[bl](8.3,0.755){$\iota$}
\rput[bl](5.7,0.655){$\varphi(\alpha_1^*)$}
\psline[linecolor=black, linewidth=0.018, arrowsize=0.05291667cm 2.0,arrowlength=0.8,arrowinset=0.2]{->}(6.75,0.805)(7.05,0.805)
\rput[bl](5.25,-0.145){=}
\psline[linecolor=black, linewidth=0.018, arrowsize=0.05291667cm 2.0,arrowlength=0.8,arrowinset=0.2]{->}(9.95,-0.095)(10.25,-0.095)
\rput[bl](9.6,-0.145){$\iota$}
\rput[bl](6.8,-0.245){$\varphi(\alpha_1^*)$}
\psline[linecolor=black, linewidth=0.018, arrowsize=0.05291667cm 2.0,arrowlength=0.8,arrowinset=0.2]{->}(7.85,-0.095)(8.15,-0.095)
\rput[bl](5.25,-1.045){=}
\psline[linecolor=black, linewidth=0.018, arrowsize=0.05291667cm 2.0,arrowlength=0.8,arrowinset=0.2]{->}(10.45,-0.995)(10.75,-0.995)
\rput[bl](10.91,-1.095){$\Delta$}
\rput[bl](10.1,-1.045){$\iota$}
\rput[bl](8.5,-1.145){$\varphi(\alpha_1^*)$}
\psline[linecolor=black, linewidth=0.018, arrowsize=0.05291667cm 2.0,arrowlength=0.8,arrowinset=0.2]{->}(9.55,-0.995)(9.85,-0.995)
\psline[linecolor=black, linewidth=0.018, arrowsize=0.05291667cm 2.0,arrowlength=0.8,arrowinset=0.2]{->}(11.35,-0.895)(11.65,-0.795)
\psline[linecolor=black, linewidth=0.018, arrowsize=0.05291667cm 2.0,arrowlength=0.8,arrowinset=0.2]{->}(11.35,-1.095)(11.65,-1.195)
\rput[bl](7.15,0.705){$S_A$}
\psline[linecolor=black, linewidth=0.018, arrowsize=0.05291667cm 2.0,arrowlength=0.8,arrowinset=0.2]{->}(7.75,0.805)(8.05,0.805)
\psline[linecolor=black, linewidth=0.018, arrowsize=0.05291667cm 2.0,arrowlength=0.8,arrowinset=0.2]{->}(8.65,0.805)(8.95,0.805)
\rput[bl](8.35,-0.245){$\Delta_{a^*}$}
\psline[linecolor=black, linewidth=0.018, arrowsize=0.05291667cm 2.0,arrowlength=0.8,arrowinset=0.2]{->}(9.05,-0.095)(9.35,-0.095)
\rput[bl](10.41,-0.195){$\Delta$}
\psline[linecolor=black, linewidth=0.018, arrowsize=0.05291667cm 2.0,arrowlength=0.8,arrowinset=0.2]{->}(10.85,0.005)(11.15,0.105)
\psline[linecolor=black, linewidth=0.018, arrowsize=0.05291667cm 2.0,arrowlength=0.8,arrowinset=0.2]{->}(10.85,-0.195)(11.15,-0.295)
\rput[bl](5.7,-1.165){$(-1)^{|\iota|} \langle a^*,\varphi(\alpha_1^*)\rangle\cdot$}
\psline[linecolor=black, linewidth=0.018, arrowsize=0.05291667cm 2.0,arrowlength=0.8,arrowinset=0.2]{->}(4.75,1.005)(5.05,1.005)
\psline[linecolor=black, linewidth=0.018, arrowsize=0.05291667cm 2.0,arrowlength=0.8,arrowinset=0.2]{->}(4.75,0.605)(5.05,0.605)
\rput[bl](11.7,-1.195){.}
\rput[bl](5.65,-0.195){$(-1)^{|\iota|} \cdot$}
\rput[bl](0.0,0.655){$\varphi((\alpha_1^*)^{-1})$}
\rput[bl](4.15,0.905){$S_H$}
\rput[bl](4.15,0.505){$S_H$}
\psline[linecolor=black, linewidth=0.018, arrowsize=0.05291667cm 2.0,arrowlength=0.8,arrowinset=0.2]{->}(9.65,0.805)(9.95,0.805)
\rput[bl](9.05,0.705){$S_H$}
\rput[bl](10.11,0.705){$\Delta$}
\psline[linecolor=black, linewidth=0.018, arrowsize=0.05291667cm 2.0,arrowlength=0.8,arrowinset=0.2]{->}(10.55,0.905)(10.85,1.005)
\psline[linecolor=black, linewidth=0.018, arrowsize=0.05291667cm 2.0,arrowlength=0.8,arrowinset=0.2]{->}(10.55,0.705)(10.85,0.605)
\rput[bl](2.2,0.755){$\iota$}
\psline[linecolor=black, linewidth=0.018, arrowsize=0.05291667cm 2.0,arrowlength=0.8,arrowinset=0.2]{->}(1.65,0.805)(1.95,0.805)
\psline[linecolor=black, linewidth=0.018, arrowsize=0.05291667cm 2.0,arrowlength=0.8,arrowinset=0.2]{->}(2.55,0.805)(2.85,0.805)
\rput[bl](3.01,0.705){$\Delta^{\text{op}}$}
\psline[linecolor=black, linewidth=0.018, arrowsize=0.05291667cm 2.0,arrowlength=0.8,arrowinset=0.2]{->}(3.75,0.905)(4.05,1.005)
\psline[linecolor=black, linewidth=0.018, arrowsize=0.05291667cm 2.0,arrowlength=0.8,arrowinset=0.2]{->}(3.75,0.705)(4.05,0.605)
\end{pspicture}
\end{figure}
\noindent The leftmost tensor is that corresponding to $\a_1$ after the orientation reversal. Here the first equality follows as in the previous case, the second uses condition (\ref{eq: cointegral inv of orientation}) of Definition \ref{def: compatible integral cointegral} and the third is by definition of $\De_{a^*}$. The last tensor is that corresponding to $\a_1$ before reversing the orientation (multiplied by a factor). However, since the basepoint of $\b_1$ traverses $\a_1$, the whole tensor $Z_H^{\rho}(\HH,\x)$ is further multiplied by $\lb a^*,\v(\a_1^*)^{-1}\rb$ by Proposition \ref{prop: changing basepoints and Z}. This cancels the factor appearing in the above computation, and shows that $Z_H^{\rho}(\HH,\x)$ is only multiplied by the sign $(-1)^{|\iota|}$. But $\dHH$ is multiplied by the same sign when the orientation of a closed curve is reversed, hence $Z_H^{\rho}(\HH,\x,\o)$ is unchanged.

\end{proof}

Finally, we have the following easy lemma.

\begin{lemma}
\label{lemma: invariance of ordering}
The scalar $Z_H^{\rho}(\HH,\x,\o)$ is independent of the ordering of $\HH$ chosen.
\end{lemma}
\begin{proof}
Indeed, whenever the relative integral has degree one, the tensor $Z_H^{\rho}(\HH)$ changes sign when two closed curves in $\aaa$ (or $\bbb$) are interchanged (thus changing the ordering of $\HH$). The factor $\dHH$ changes sign in the same way, and therefore the product $Z_H^{\rho}(\HH,\o)\eq\dHH Z_H^{\rho}(\HH)$ is unchanged.
\end{proof}

In the following proof, we will denote the tensors of $H$ simply by $m,\De,S$, only the tensors of $A$ or $B$ will have subscripts, such as $i_A,m_A,\De_A$, etc.

\begin{proof}[Proof of theorem \ref{thm: Z is an invariant}] We have to show that $\zeta_{\ss,\x}Z_H^{\rho}(\HH,\x,\o)$ only depends on the underlying $(M,\c,\ss,\o,\rho)$ and that $$I_H^{\rho}(M,\c,\ss,\o)=I_H^{\rho'}(M',\c',\ss',\o')$$ whenever $d:(M,\c)\to (M',\c')$ is a sutured manifold diffeomorphism and $\rho'\circ d_*=\rho, \ss'=d_*(\ss), \o'=d_*(\o)$. The latter follows from the former since $d(\HH)=(d(\S),d(\aaa),d(\bbb))$ is an extended diagram of $(M',\c')$ for any extended diagram $\HH=\HDD$ of $(M,\c)$ and it is obvious that $$\zeta_{\ss,\x}Z_H^{\rho}(\HH,\x,\o)=\zeta_{\ss',d(\x)}Z_H^{\rho'}(d(\HH),d(\x),\o').$$
By Lemmas \ref{lemma: invariance of multipoint}, \ref{lemma: invariance of orientations}, \ref{lemma: invariance of ordering}, in order to show that $\zeta_{\ss,\x}Z_H^{\rho}(\HH,\x,\o)$ depends only on $(M,\c,\ss,\o,\rho)$, it suffices to show invariance under the extended Heegaard moves of Proposition \ref{prop: extended RS thm}. Note that since all the properties of the integral are dual to that of the cointegral, it suffices to prove isotopy and handlesliding invariance for the attaching curves of just one handlebody, say, the lower one. Each time we change a curve or arc $\a\in\aaa$ by one of such moves, we will denote by $\aa'$ or $\bolda'$ the new set of ordered oriented curves or arcs thus obtained and by $\HH'=(\S',\aaap,\bbbp)$ the new Heegaard diagram. %We also let $N'$ be the number of crossings of $\HH'$ and let $\tau'$ be the permutation of $S_{N'}$ associated to $\HH'$. 
We suppose $\HH$ is based according to $\x\in\Tab$, with $\x=\{x_1,\dots, x_d\}$ and $x_i\in\a_i\cap\b_i$ for $i=1,\dots, d$, while $\HH'$ is based according to $j(\x)\in\Tabprime$ as in Proposition \ref{prop: Spinc map is preserved under Heegaard moves} (so in the case of isotopy, we suppose $\HH'$ is obtained from $\HH$ by adding two new intersection points). Note that since $s'(j(\x))=s(\x)$, we get $h'_{\ss,j(\x)}=h_{\ss,\x}$ in $H_1(M)$ and so $\zeta'_{\ss,j(\x)}=\zeta_{\ss,\x}$ in $\kk^{\t}$. Hence, it suffices to show that
\begin{align*}
Z_H^{\rho}(\HH',j(\x),\o)=Z_H^{\rho}(\HH,\x,\o)
\end{align*}
for each extended Heegaard move.

\begin{enumerate}

\item Isotopy: suppose first that the isotopy of a curve or arc takes place inside $\inte(\S)$. Then the isotopy adds two new consecutive crossings, and after an orientation reversal if necessary, we can suppose the first is negative and the second is positive. The tensor corresponding to the isotopy region is the following:
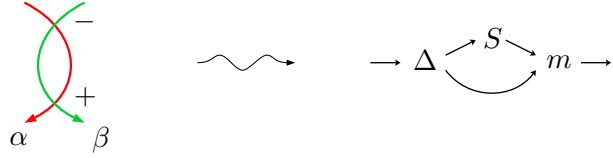
\begin{figure}[H]
\centering
\begin{pspicture}(0,-1.0067096)(7.9644337,1.0067096)
\definecolor{colour0}{rgb}{0.0,0.8,0.2}
\psbezier[linecolor=red, linewidth=0.03, arrowsize=0.05291667cm 2.0,arrowlength=1.4,arrowinset=0.0]{->}(0.2,0.9932905)(1.0,0.5932905)(1.0,-0.20670952)(0.2,-0.6067095184326172)
\psbezier[linecolor=colour0, linewidth=0.03, arrowsize=0.05291667cm 2.0,arrowlength=1.4,arrowinset=0.0]{->}(1.0,0.9932905)(0.2,0.5932905)(0.2,-0.20670952)(1.0,-0.6067095184326172)
\rput[bl](0.86,-0.39670953){$+$}
\rput[bl](0.85,0.6132905){$-$}
\rput[bl](5.37,0.09329048){$\Delta$}
\psline[linecolor=black, linewidth=0.018, arrowsize=0.05291667cm 2.0,arrowlength=0.8,arrowinset=0.2]{->}(5.8,0.2932905)(6.2,0.49329048)
\rput[bl](6.3,0.3932905){$S$}
\psbezier[linecolor=black, linewidth=0.018, arrowsize=0.05291667cm 2.0,arrowlength=0.8,arrowinset=0.2]{->}(5.8,0.09329048)(6.1,-0.30670953)(6.7,-0.30670953)(7.0,0.09329048156738282)
\psline[linecolor=black, linewidth=0.018, arrowsize=0.05291667cm 2.0,arrowlength=0.8,arrowinset=0.2]{->}(6.6,0.49329048)(7.0,0.2932905)
\rput[bl](7.14,0.12329048){$m$}
\psline[linecolor=black, linewidth=0.018, arrowsize=0.05291667cm 2.0,arrowlength=0.8,arrowinset=0.2]{->}(4.8,0.19329049)(5.2,0.19329049)
\psline[linecolor=black, linewidth=0.018, arrowsize=0.05291667cm 2.0,arrowlength=0.8,arrowinset=0.2]{->}(7.6,0.19329049)(8.0,0.19329049)
\rput[bl](0.0,-0.9067095){$\alpha$}
\rput[bl](1.1,-1.0067096){$\beta$}
\psbezier[linecolor=black, linewidth=0.014](2.5,0.19329049)(2.7,0.19329049)(2.7076924,0.2932905)(2.8,0.2932904815673828)(2.8923078,0.2932905)(3.0076923,0.09329048)(3.1,0.09329048)(3.1923077,0.09329048)(3.3307693,0.2932905)(3.4230769,0.2932905)(3.5153847,0.2932905)(3.5153847,0.19329049)(3.7,0.19329049)
\psline[linecolor=black, linewidth=0.014, arrowsize=0.05291667cm 2.0,arrowlength=1.4,arrowinset=0.0]{->}(3.7,0.19329049)(3.8,0.19329049)
\end{pspicture}
\end{figure}
\noindent That this equals the tensor before the isotopy follows immediately from the antipode axiom. Now suppose we isotope an arc $\a\in\bolda$ past an arc $\b\in\boldb$ along $\p\S$, thus adding a unique new crossing. Such a crossing adds the following tensor to the picture:
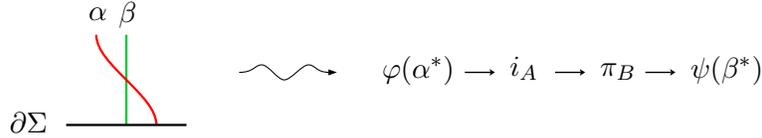
\begin{figure}[H]
\centering
\begin{pspicture}(0,-0.855)(9.9,0.855)
\definecolor{colour0}{rgb}{0.0,0.8,0.2}
\psline[linecolor=colour0, linewidth=0.03](1.55,0.445)(1.55,-0.755)
\psbezier[linecolor=red, linewidth=0.03](1.15,0.445)(1.15,0.045)(1.95,-0.355)(1.95,-0.755)
\psline[linecolor=black, linewidth=0.03](0.75,-0.755)(2.35,-0.755)
\psline[linecolor=black, linewidth=0.018, arrowsize=0.05291667cm 2.0,arrowlength=0.8,arrowinset=0.2]{->}(6.05,-0.055)(6.45,-0.055)
\rput[bl](6.65,-0.155){$i_A$}
\psline[linecolor=black, linewidth=0.018, arrowsize=0.05291667cm 2.0,arrowlength=0.8,arrowinset=0.2]{->}(7.25,-0.055)(7.65,-0.055)
\psline[linecolor=black, linewidth=0.018, arrowsize=0.05291667cm 2.0,arrowlength=0.8,arrowinset=0.2]{->}(8.45,-0.055)(8.85,-0.055)
\rput[bl](7.85,-0.155){$\pi_B$}
\rput[bl](9.05,-0.215){$\psi(\beta^*)$}
\rput[bl](4.95,-0.215){$\varphi(\alpha^*)$}
\rput[bl](1.05,0.645){$\alpha$}
\rput[bl](1.45,0.545){$\beta$}
\rput[bl](0.0,-0.855){$\partial\Sigma$}
\psbezier[linecolor=black, linewidth=0.014](3.05,-0.055)(3.25,-0.055)(3.2576923,0.045)(3.35,0.045)(3.4423077,0.045)(3.5576923,-0.155)(3.65,-0.155)(3.7423077,-0.155)(3.8807693,0.045)(3.9730768,0.045)(4.0653844,0.045)(4.0653844,-0.055)(4.25,-0.055)
\psline[linecolor=black, linewidth=0.014, arrowsize=0.05291667cm 2.0,arrowlength=1.4,arrowinset=0.0]{->}(4.25,-0.055)(4.35,-0.055)
\end{pspicture}
\end{figure}
\noindent That the scalar on the right hand side is trivial follows from condition (\ref{eq: invariance under bdry isotopy of arc}) of Definition \ref{def: compatible integral cointegral}. This proves what we want.
\medskip

\item Diffeomorphisms isotopic to the identity: let $d:\HH_1\to \HH_2$ be a diffeomorphism as in Definition \ref{def: diffeo isotopic to identity}. By the isotopy condition, the dual homology classes $\a^*$ and $\b^*$ of both diagrams are the same in $H_1(M)$. Since both diagrams are diffeomorphic, this immediately implies that $Z(\HH_1)=Z(\HH_2)$.
\medskip

\item Handlesliding: suppose we handleslide $\a_j$ over $\a_i$ over an arc $\d$, where $\a_j,\a_i\in \aaa$ and $(\a_j,\a_i)\notin \aa\t\bolda$. We denote by $\a'_j$ the curve obtained after handlesliding. Since we already proved invariance of ordering and orientation (and isotopy), we make the following assumptions:

\begin{enumerate}
\item $\a_i<\a_j$ in the order of $\aaa$.
\item $\a_i,\a_j,\a'_j$ are oriented so that $\p P=\a_i\cup \a_j\cup -\a'_j$ as oriented 1-manifolds, where $P$ is the handlesliding region (see Remark \ref{remark: handlesliding alpha curves}).
\item If $\a_i\in\aa$ we orient $\b_i$ so that the crossing $x_i$ is positive. If $\a_j\in\aa$ we orient $\b_j$ so that $x_j$ is negative. 
\item When both $\a_i,\a_j\in\aa$ we suppose the handlesliding arc $\d$ connects $p_i(\x)$ to $p_j(\x)$.
\end{enumerate}

Thus, the handlesliding region looks as follows:

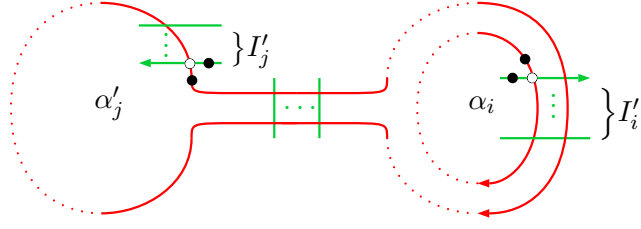
\begin{figure}[H]
\centering
\begin{pspicture}(0,-1.4200084)(8.309883,1.4200084)
\definecolor{colour0}{rgb}{0.0,0.8,0.2}
\psbezier[linecolor=red, linewidth=0.03](2.4198835,0.39999154)(2.4198835,0.99999154)(1.8198835,1.3999915)(1.2198834,1.3999915313720703)
\psbezier[linecolor=red, linewidth=0.03, arrowsize=0.05291667cm 2.0,arrowlength=1.4,arrowinset=0.0]{->}(6.2198834,0.99999154)(6.6198835,0.99999154)(7.0198836,0.6999915)(7.0198836,0.0)(7.0198836,-0.70000845)(6.6198835,-1.0000085)(6.2198834,-1.0000085)
\psbezier[linecolor=red, linewidth=0.03, linestyle=dotted, dotsep=0.10583334cm](6.2198834,-1.0000085)(5.8198833,-1.0000085)(5.4198833,-0.70000845)(5.4198833,0.0)(5.4198833,0.6999915)(5.8198833,0.99999154)(6.2198834,0.99999154)
\psline[linecolor=colour0, linewidth=0.03, arrowsize=0.05291667cm 2.0,arrowlength=1.4,arrowinset=0.0]{->}(6.5198836,0.39999154)(7.7198834,0.39999154)
\psline[linecolor=colour0, linewidth=0.03](6.5198836,-0.40000847)(7.7198834,-0.40000847)
\rput[bl](7.1798835,-0.13000847){\textcolor{colour0}{$\vdots$}}
\psdots[linecolor=black, fillstyle=solid, dotstyle=o, dotsize=0.14, fillcolor=white](6.9498835,0.39999154)
\psdots[linecolor=black, dotsize=0.14](6.8532166,0.6499915)
\psdots[linecolor=black, dotsize=0.14](6.689883,0.39999154)
\psbezier[linecolor=red, linewidth=0.03, linestyle=dotted, dotsep=0.10583334cm](5.0198836,0.39999154)(5.0198836,0.99999154)(5.6198835,1.3999915)(6.2198834,1.3999915313720703)
\psbezier[linecolor=red, linewidth=0.03, linestyle=dotted, dotsep=0.10583334cm](5.0198836,-0.40000847)(5.0198836,-1.0000085)(5.6198835,-1.4000084)(6.2198834,-1.4000084686279297)
\psbezier[linecolor=red, linewidth=0.03, arrowsize=0.05291667cm 2.0,arrowlength=1.4,arrowinset=0.0]{->}(6.2198834,1.3999915)(6.8198833,1.3999915)(7.4198833,1.0999916)(7.4198833,0.0)(7.4198833,-1.1000085)(6.8198833,-1.4000084)(6.2198834,-1.4000084)
\psline[linecolor=colour0, linewidth=0.03](1.7198834,1.0999916)(2.8198833,1.0999916)
\psline[linecolor=colour0, linewidth=0.03, arrowsize=0.05291667cm 2.0,arrowlength=1.4,arrowinset=0.0]{<-}(1.7198834,0.59999156)(2.8198833,0.59999156)
\rput[bl](2.0298834,0.67999154){\textcolor{colour0}{$\vdots$}}
\psdots[linecolor=black, dotsize=0.14](2.6498835,0.59999156)
\psdots[linecolor=black, fillstyle=solid, dotstyle=o, dotsize=0.14, fillcolor=white](2.3998835,0.5899915)
\rput[bl](1.1298834,-0.16000848){$\alpha'_j$}
\rput[bl](6.0998836,-0.07000847){$\alpha_i$}
\psline[linecolor=colour0, linewidth=0.03](3.5198834,0.39999154)(3.5198834,-0.40000847)
\psline[linecolor=colour0, linewidth=0.03](4.1198835,0.39999154)(4.1198835,-0.40000847)
\rput[bl](3.6398835,-0.010008468){\textcolor{colour0}{$\dots$}}
\psbezier[linecolor=red, linewidth=0.03](2.4198835,-0.40000847)(2.4198835,-1.0000085)(1.8198835,-1.4000084)(1.2198834,-1.4000084686279297)
\psbezier[linecolor=red, linewidth=0.03, linestyle=dotted, dotsep=0.10583334cm](1.2198834,1.3999915)(0.6198834,1.3999915)(0.019883422,0.99999154)(0.019883422,0.0)(0.019883422,-1.0000085)(0.6198834,-1.4000084)(1.2198834,-1.4000084)
\psbezier[linecolor=red, linewidth=0.03](2.4198835,0.39999154)(2.4198835,0.19999152)(2.4198835,0.19999152)(3.0198834,0.19999153137207032)(3.6198835,0.19999152)(3.5198834,0.19999152)(3.8198833,0.19999152)
\psbezier[linecolor=red, linewidth=0.03](5.0198836,0.39999154)(5.0198836,0.19999152)(5.0198836,0.19999152)(4.4198833,0.19999153137207032)(3.8198833,0.19999152)(4.0198836,0.19999152)(3.8198833,0.19999152)
\psbezier[linecolor=red, linewidth=0.03](2.4198835,-0.40000847)(2.4198835,-0.20000847)(2.4198835,-0.20000847)(3.0198834,-0.20000846862792968)(3.6198835,-0.20000847)(3.6198835,-0.20000847)(3.8198833,-0.20000847)
\psbezier[linecolor=red, linewidth=0.03](5.0198836,-0.40000847)(5.0198836,-0.20000847)(5.0198836,-0.20000847)(4.4198833,-0.20000846862792968)(3.8198833,-0.20000847)(3.8198833,-0.20000847)(3.6198835,-0.20000847)
\psdots[linecolor=black, dotsize=0.14](2.4265501,0.36999154)
\rput[bl](2.9298835,0.5299915){$\Big\}I'_j $}
\rput[bl](7.8298836,-0.42000848){$\bigg\}I'_i$}
\end{pspicture}
\label{fig: Proof of Thm, handlesliding alphas}
\end{figure}

\noindent Note that by (b), one has $o(\HH)=o(\HH')$ so also $\dHH=\dHHp$. Let $I_{\d}\eq \d\cap\bbb$. For the moment, we further suppose that $I_{\d}=\emptyset$. Thus, we can write
\begin{align*}
\a'_j\cap\bbb=I'_i\cup I'_j
\end{align*}
where $I'_n$ is the set of crossings through $\a'_j$ that sit next to the crossings of $\a_n$ for $n=i,j$. By the conventions above (including (c) in the case $\a_j,\a_i\in\aa$) we see that when following $\a'_j$ starting from its basepoint, the crossings in $I'_i$ appear first and then come those of $I'_j$, as in the above figure. Thus, $I'_i,I'_j$ inherit the order coming from $\a_i,\a_j$ and $x<y$ for all $x\in I'_i, y\in I'_j$. We can suppose the same in the case that $\a_j'$ is an arc, since $i\v((\a'_j)^*)$ is group-like in $H$. Denoting $g_i=\v(\a^*_i)$ and $g_j=\v(\a^*_j)$, it follows that the tensor corresponding to $\a_i\cup\a_j'$ has the following form:

\begin{figure}[H]
\centering
\begin{pspicture}(0,-0.8861223)(9.901425,0.8861223)
\rput[bl](5.25,-0.06387772){=}
\rput[bl](2.61,0.48612228){$\Delta$}
\rput[bl](0.5,0.4361223){$g_ig_j^{-1}$}
\psline[linecolor=black, linewidth=0.018, arrowsize=0.05291667cm 2.0,arrowlength=0.8,arrowinset=0.2]{->}(1.35,0.5861223)(1.65,0.5861223)
\rput[bl](1.71,-0.41387773){$\Delta$}
\rput[bl](0.0,-0.4638777){$g_j$}
\psline[linecolor=black, linewidth=0.018, arrowsize=0.05291667cm 2.0,arrowlength=0.8,arrowinset=0.2]{->}(0.45,-0.3138777)(0.75,-0.3138777)
\psline[linecolor=black, linewidth=0.018, arrowsize=0.05291667cm 2.0,arrowlength=0.8,arrowinset=0.2]{->}(2.15,-0.21387772)(2.45,-0.11387771)
\psline[linecolor=black, linewidth=0.018, arrowsize=0.05291667cm 2.0,arrowlength=0.8,arrowinset=0.2]{->}(2.15,-0.41387773)(2.45,-0.5138777)
\rput[bl](2.61,-0.11387771){$\Delta$}
\rput[bl](2.61,-0.71387774){$\Delta$}
\psline[linecolor=black, linewidth=0.018, arrowsize=0.05291667cm 2.0,arrowlength=0.8,arrowinset=0.2]{->}(3.05,-0.5138777)(3.35,-0.3138777)
\psline[linecolor=black, linewidth=0.018, arrowsize=0.05291667cm 2.0,arrowlength=0.8,arrowinset=0.2]{->}(3.05,-0.71387774)(3.35,-0.9138777)
\psline[linecolor=black, linewidth=0.018, arrowsize=0.05291667cm 2.0,arrowlength=0.8,arrowinset=0.2]{->}(3.05,-0.6138777)(3.35,-0.6138777)
\psline[linecolor=black, linewidth=0.018, arrowsize=0.05291667cm 2.0,arrowlength=0.8,arrowinset=0.2]{->}(3.05,-0.013877716)(3.65,-0.013877716)
\psline[linecolor=black, linewidth=0.018, arrowsize=0.05291667cm 2.0,arrowlength=0.8,arrowinset=0.2]{->}(3.05,0.5861223)(3.65,0.5861223)
\psline[linecolor=black, linewidth=0.018, arrowsize=0.05291667cm 2.0,arrowlength=0.8,arrowinset=0.2]{->}(3.05,0.08612228)(3.65,0.48612228)
\psline[linecolor=black, linewidth=0.018, arrowsize=0.05291667cm 2.0,arrowlength=0.8,arrowinset=0.2]{->}(3.05,0.48612228)(3.65,0.08612228)
\rput[bl](3.8,0.48612228){$m$}
\rput[bl](3.8,-0.11387771){$m$}
\psline[linecolor=black, linewidth=0.018, arrowsize=0.05291667cm 2.0,arrowlength=0.8,arrowinset=0.2]{->}(4.25,0.5861223)(4.55,0.5861223)
\psline[linecolor=black, linewidth=0.018, arrowsize=0.05291667cm 2.0,arrowlength=0.8,arrowinset=0.2]{->}(4.25,-0.013877716)(4.55,-0.013877716)
\rput[bl](6.2,0.13612229){$g_ig_j^{-1}$}
\psline[linecolor=black, linewidth=0.018, arrowsize=0.05291667cm 2.0,arrowlength=0.8,arrowinset=0.2]{->}(7.05,0.2861223)(7.35,0.2861223)
\rput[bl](6.6,-0.6638777){$g_j$}
\psline[linecolor=black, linewidth=0.018, arrowsize=0.05291667cm 2.0,arrowlength=0.8,arrowinset=0.2]{->}(7.05,-0.5138777)(7.35,-0.5138777)
\rput[bl](8.3,0.18612228){$m$}
\psline[linecolor=black, linewidth=0.018, arrowsize=0.05291667cm 2.0,arrowlength=0.8,arrowinset=0.2]{->}(8.75,0.2861223)(9.05,0.2861223)
\rput[bl](8.31,-0.6138777){$\Delta$}
\psline[linecolor=black, linewidth=0.018, arrowsize=0.05291667cm 2.0,arrowlength=0.8,arrowinset=0.2]{->}(8.45,-0.21387772)(8.45,0.08612228)
\psline[linecolor=black, linewidth=0.018, arrowsize=0.05291667cm 2.0,arrowlength=0.8,arrowinset=0.2]{->}(8.75,-0.5138777)(9.05,-0.5138777)
\rput[bl](9.21,-0.6138777){$\Delta$}
\psline[linecolor=black, linewidth=0.018, arrowsize=0.05291667cm 2.0,arrowlength=0.8,arrowinset=0.2]{->}(9.65,-0.41387773)(9.95,-0.21387772)
\psline[linecolor=black, linewidth=0.018, arrowsize=0.05291667cm 2.0,arrowlength=0.8,arrowinset=0.2]{->}(9.65,-0.6138777)(9.95,-0.8138777)
\psline[linecolor=black, linewidth=0.018, arrowsize=0.05291667cm 2.0,arrowlength=0.8,arrowinset=0.2]{->}(9.65,-0.5138777)(9.95,-0.5138777)
\rput[bl](9.21,0.18612228){$\Delta$}
\psline[linecolor=black, linewidth=0.018, arrowsize=0.05291667cm 2.0,arrowlength=0.8,arrowinset=0.2]{->}(9.65,0.3861223)(9.95,0.48612228)
\psline[linecolor=black, linewidth=0.018, arrowsize=0.05291667cm 2.0,arrowlength=0.8,arrowinset=0.2]{->}(9.65,0.18612228)(9.95,0.08612228)
\psline[linecolor=black, linewidth=0.018, arrowsize=0.05291667cm 2.0,arrowlength=0.8,arrowinset=0.2]{->}(7.85,-0.5138777)(8.15,-0.5138777)
\rput[bl](7.45,-0.6738777){$f_j$}
\psline[linecolor=black, linewidth=0.018, arrowsize=0.05291667cm 2.0,arrowlength=0.8,arrowinset=0.2]{->}(7.85,0.2861223)(8.15,0.2861223)
\rput[bl](7.46,0.15612228){$f_i$}
\psline[linecolor=black, linewidth=0.018, arrowsize=0.05291667cm 2.0,arrowlength=0.8,arrowinset=0.2]{->}(1.25,-0.3138777)(1.55,-0.3138777)
\rput[bl](0.85,-0.47387773){$f_j$}
\psline[linecolor=black, linewidth=0.018, arrowsize=0.05291667cm 2.0,arrowlength=0.8,arrowinset=0.2]{->}(2.15,0.5861223)(2.45,0.5861223)
\rput[bl](1.76,0.45612228){$f_i$}
\rput[bl](10.2,-0.9){.}
\end{pspicture}
\end{figure}

\noindent Here we denote $(f_i,f_j)=(\iota,\iota), (\iota, i_A)$ or $(i_A,i_A)$ depending on whether $(\a_i,\a_j)\in \aa\t\aa, \aa\t\bolda,\bolda\t\bolda$. On the right hand side we used that $\De$ is an algebra morphism. That this tensor equals the tensor before handlesliding is shown in the following computation:
\begin{figure}[H]
\centering

\begin{pspicture}(-0.5,-0.63)(12.5,0.63)
\rput[bl](0.0,0.18){$g_ig_j^{-1}$}
\psline[linecolor=black, linewidth=0.018, arrowsize=0.05291667cm 2.0,arrowlength=0.8,arrowinset=0.2]{->}(0.85,0.33)(1.15,0.33)
\psline[linecolor=black, linewidth=0.018, arrowsize=0.05291667cm 2.0,arrowlength=0.8,arrowinset=0.2]{->}(1.65,-0.47)(1.95,-0.47)
\rput[bl](1.25,-0.63){$f_j$}
\rput[bl](0.4,-0.63){$g_j$}
\psline[linecolor=black, linewidth=0.018, arrowsize=0.05291667cm 2.0,arrowlength=0.8,arrowinset=0.2]{->}(0.85,-0.47)(1.15,-0.47)
\psline[linecolor=black, linewidth=0.018, arrowsize=0.05291667cm 2.0,arrowlength=0.8,arrowinset=0.2]{->}(1.65,0.33)(1.95,0.33)
\rput[bl](1.26,0.2){$f_i$}
\rput[bl](2.1,0.23){$m$}
\psline[linecolor=black, linewidth=0.018, arrowsize=0.05291667cm 2.0,arrowlength=0.8,arrowinset=0.2]{->}(2.55,0.33)(2.85,0.33)
\rput[bl](2.11,-0.57){$\Delta$}
\psline[linecolor=black, linewidth=0.018, arrowsize=0.05291667cm 2.0,arrowlength=0.8,arrowinset=0.2]{->}(2.25,-0.17)(2.25,0.13)
\psline[linecolor=black, linewidth=0.018, arrowsize=0.05291667cm 2.0,arrowlength=0.8,arrowinset=0.2]{->}(2.55,-0.47)(2.85,-0.47)
\psline[linecolor=black, linewidth=0.018, arrowsize=0.05291667cm 2.0,arrowlength=0.8,arrowinset=0.2]{->}(4.45,0.33)(4.75,0.33)
\psline[linecolor=black, linewidth=0.018, arrowsize=0.05291667cm 2.0,arrowlength=0.8,arrowinset=0.2]{->}(4.45,-0.47)(4.75,-0.47)
\rput[bl](4.9,0.23){$m_A$}
\psline[linecolor=black, linewidth=0.018, arrowsize=0.05291667cm 2.0,arrowlength=0.8,arrowinset=0.2]{->}(5.55,0.33)(5.85,0.33)
\rput[bl](4.91,-0.57){$\Delta_A$}
\psline[linecolor=black, linewidth=0.018, arrowsize=0.05291667cm 2.0,arrowlength=0.8,arrowinset=0.2]{->}(5.05,-0.17)(5.05,0.13)
\psline[linecolor=black, linewidth=0.018, arrowsize=0.05291667cm 2.0,arrowlength=0.8,arrowinset=0.2]{->}(5.55,-0.47)(5.85,-0.47)
\psline[linecolor=black, linewidth=0.018, arrowsize=0.05291667cm 2.0,arrowlength=0.8,arrowinset=0.2]{->}(8.25,0.33)(8.55,0.33)
\rput[bl](8.7,0.23){$m_A$}
\psline[linecolor=black, linewidth=0.018, arrowsize=0.05291667cm 2.0,arrowlength=0.8,arrowinset=0.2]{->}(9.35,0.33)(9.65,0.33)
\psline[linecolor=black, linewidth=0.018, arrowsize=0.05291667cm 2.0,arrowlength=0.8,arrowinset=0.2]{->}(8.05,-0.27)(8.648252,0.13345931)
\psline[linecolor=black, linewidth=0.018, arrowsize=0.05291667cm 2.0,arrowlength=0.8,arrowinset=0.2]{->}(9.35,-0.47)(9.65,-0.47)
\rput[bl](8.9,-0.62){$g_j$}
\psline[linecolor=black, linewidth=0.018, arrowsize=0.05291667cm 2.0,arrowlength=0.8,arrowinset=0.2]{->}(11.75,0.33)(12.05,0.33)
\psline[linecolor=black, linewidth=0.018, arrowsize=0.05291667cm 2.0,arrowlength=0.8,arrowinset=0.2]{->}(11.75,-0.47)(12.05,-0.47)
\rput[bl](3.15,-0.07){=}
\rput[bl](10.75,-0.07){=}
\rput[bl](6.95,-0.07){=}
\rput[bl](3.6,0.18){$g_ig_j^{-1}$}
\rput[bl](4.0,-0.62){$g_j$}
\rput[bl](7.4,0.18){$g_ig_j^{-1}$}
\rput[bl](7.7,-0.62){$g_j$}
\rput[bl](11.3,0.18){$g_i$}
\rput[bl](11.3,-0.62){$g_j$}
\psline[linecolor=black, linewidth=0.018, arrowsize=0.05291667cm 2.0,arrowlength=0.8,arrowinset=0.2]{->}(6.35,-0.47)(6.65,-0.47)
\rput[bl](5.95,-0.63){$f_j$}
\psline[linecolor=black, linewidth=0.018, arrowsize=0.05291667cm 2.0,arrowlength=0.8,arrowinset=0.2]{->}(6.35,0.33)(6.65,0.33)
\rput[bl](5.96,0.2){$f_i$}
\psline[linecolor=black, linewidth=0.018, arrowsize=0.05291667cm 2.0,arrowlength=0.8,arrowinset=0.2]{->}(10.15,-0.47)(10.45,-0.47)
\rput[bl](9.75,-0.63){$f_j$}
\psline[linecolor=black, linewidth=0.018, arrowsize=0.05291667cm 2.0,arrowlength=0.8,arrowinset=0.2]{->}(10.15,0.33)(10.45,0.33)
\rput[bl](9.76,0.2){$f_i$}
\psline[linecolor=black, linewidth=0.018, arrowsize=0.05291667cm 2.0,arrowlength=0.8,arrowinset=0.2]{->}(12.45,-0.47)(12.75,-0.47)
\rput[bl](12.05,-0.63){$f_j$}
\psline[linecolor=black, linewidth=0.018, arrowsize=0.05291667cm 2.0,arrowlength=0.8,arrowinset=0.2]{->}(12.45,0.33)(12.75,0.33)
\rput[bl](12.06,0.2){$f_i$}
\rput[bl](12.95,-0.62){.}
\end{pspicture}
\end{figure}

\noindent In the first equality we have used Proposition \ref{prop: hsliding property for cointegral arc-arc, arc-curve, curve-curve}, in the second that $g_j$ is group-like and in the third we multiplied the corresponding group-likes. The last tensor is that corresponding to $\HH$, that is, before handlesliding, proving what we want. Now suppose that $I_{\d}=\d\cap\bbb\neq \emptyset$. Isotope $\a_j$ through $\d$ pass through all the points of $I_{\d}$, denote by $\wt{\a_j}$ the curve thus obtained and let $\d'$ the portion of $\d$ left, so that $\d'\cap\bbb=\emptyset$. Then $\a'_j$ is obtained by handlesliding $\wt{\a}_j$ over $\a_i$ through $\d'$ with $\d'\cap\bbb=\emptyset$ so the general case follows follows from the case $I_{\d}=\emptyset$ above in the case that $\a_j$ is an arc. When $\a_j\in\aa$ we further need to take care of basepoints. In this case, we first move $p_j(\x')\in \a'_j$ to $p_0=\wt{\a}_j\cap\d'$, this multiplies $Z$ by some factor $\la$ as in Proposition \ref{prop: changing basepoints and Z}. Then one can apply the argument above for $I_{\d}=\emptyset$ and finally move the basepoint back to $p_j(\x')$, which has the effect of cancelling the initial $\la$. This shows that $Z_H^{\rho}(\HH)=Z_H^{\rho}(\HH',j(\x))$ in all situations.

\medskip

\item Stabilization: this follows from the axiom $\e_B\circ\mu\circ\iota\circ\eta_A=1$ from the compatibility condition.

\end{enumerate}
\end{proof}

\section[The torsion for sutured manifolds]{The torsion for sutured manifolds}
\label{section: The torsion for sutured manifolds}

%\begin{enumerate}\item \cite{FJR11} requires $M$ irreducible for $\tau$ be computed via Fox calculus. Why? The theorem should be written in the following way: the  invariant $Z_{H_n}$ can be computed as a Fox calculus determinant. No need to mention torsion or that this determinant is an invariant. \item Some notation for arcs needs to be changed. Use $c'_i$'s for arcs along alphas and $d$'s for the beta curves.\end{enumerate}

In this section we prove Theorem \ref{thm: intro Z recovers Reidemeister torsion}. The main ingredient is Theorem \ref{thm: intro Z via Fox calculus} of Subsection \ref{subs: computation of ZHn} in which we give a Fox calculus expression for $Z_{H_n}^{\rho}$ (hence also for $I_{H_n}^{\rho}$). In Subsection \ref{subs: The case n=0} we explain how to define $Z_{H_0}^{\rho}$ and $I_{H_0}^{\rho}$ even if $H_0$ has no cointegral. The complete proof of Theorem \ref{thm: intro Z recovers Reidemeister torsion} is devoted to Subsection \ref{subs: the sutured Reid torsion}. Finally, in Subsection \ref{subs: corollaries} we deduce some immediate corollaries from this theorem.

\subsection[Fox calculus]{Fox calculus}
\label{subs: Fox calculus}\def\bstar{x} We begin by recalling the Fox calculus (see e.g. \cite[Sect. 16.2]{Turaev:BOOK2}). Let $F$ be a finitely generated free group, say, with free generators $\bstar_1,\dots,\bstar_{d+l}$. Let $\Z[F]$ denote the group ring of $F$. There is an augmentation map $\aug:\Z[F]\to\Z$ defined by $\aug(\bstar_i)=1$ for all $i=1,\dots,d+l$ and extended linearly to $\Z[F]$. Fox calculus says that there are abelian group homomorphisms $\p/\p \bstar_i:\Z[F]\to \Z[F]$ satisfying
\begin{align*}
\frac{\p (uv)}{\p \bstar_i}=\frac{\p u}{\p \bstar_i}\aug(v)+u\frac{\p v}{\p \bstar_i}
\end{align*}
for all $u,v\in\Z[F]$ and
\begin{align*}
\frac{\p \bstar_j}{\p \bstar_i}=\d_{ij}
\end{align*}
for all $i,j=1,\dots,d+l$. As a consequence of these properties, one has\begin{align*}
\frac{\p \bstar_i^{-1}}{\p \bstar_i}=-\bstar_i^{-1}.
\end{align*}
If $w\in F$ is a word in the generators $x_1,\dots,x_{d+l}$, the Fox derivative $\p w/\p\bstar_i$ is computed as follows: suppose there are $n$ appearances of $\bstar_i^{\pm 1}$ in $w$, say $$w=w_1\bstar_i^{m_1}w_2\dots w_n\bstar_i^{m_n}w_{n+1}$$ where each word $w_k$ contains no $\bstar_i^{\pm 1}$ and $m_j=\pm 1$ for each $j$. Then 
\begin{align}
\label{eq: fox derivative of a word w}
\frac{\p w}{\p x_i}=\sum_{j=1}^n m_j(A_j x_i^{-\e_j})
\end{align}
where  $A_j$ is the subword of $w$ given by $A_j=w_1x_i^{m_1}w_2\dots x_i^{m_{j-1}} w_j$ and $\e_j$ is defined by $\e_j=0$ if $m_j=1$ and $\e_j=1$ if $m_j=-1$.

%\begin{example}If $K$ is the left trefoil, the presentation of $\pi_1(S^3\sm K)$ associated to the Heegaard diagram of Example \ref{example: sutured HD for link complement} (with orientations and basepoints as in Example \ref{example: computation of Z of left trefoil}) is\begin{align*}\lb \b^*, b^* \ | \ \b^*b^*(\b^*)^{-1}b^*\b^*(b^*)^{-1}\rb.\end{align*}Note that this gives $\b^*b^*=1$ in $H_1M$. The Fox derivative is\begin{align*}\frac{\p\ov{\a}}{\p \b^*}&=1-\b^*b^*(\b^*)^{-1}+\b^*b^*(\b^*)^{-1}b^* \\\end{align*}so in $H_1(M)$ this becomes\begin{align*}\frac{\p\ov{\a}}{\p \b^*}&=1-b^*+(b^*)^{2}.\end{align*}Evaluating this on the character $\rho:H_1(M)\to \C^{\t}$ defined by $\rho(b^*)=e^{\frac{2\pi i}{n}}$ gives $Z(\HH,\x_0)$ as computed in Example \ref{example: computation of Z of left trefoil} therefore satisfying the conclusion of Theorem \ref{thm: intro Z via Fox calculus}\end{example}

\subsection[Computation of $Z^{\rho}_{H_n}$]{Computation of $Z^{\rho}_{H_n}$}\label{subs: computation of ZHn} Consider the algebra $H_n$ of Definition \ref{def: Hopf algebra Hn} at $n\geq 1$ and $\kk=\C$, and take the relative integral and cointegral defined in Example \ref{example: Hopf algebra Hn finite dim quotient}. Since $A=\C$ and $B=\C\lb K^{\pm 1} \ | \ K^n=1\rb$ we can think of $G(A\ot B^*)=G(B^*)$ as the group of $n$-th roots of unity. Note that an homomorphism $\rho:G\to G(B^*)$ is uniquely determined by the homomorphism $G\to\C^{\t}, g\mapsto \rho(g)(K)$. We will denote the latter by $\rho_K$.
\medskip

Now let $(M,\c)$ be a balanced sutured 3-manifold with connected $R_-(\c)$ and let $\HH=\HDD$ be an ordered, oriented, based, extended Heegaard diagram of $(M,\c)$. Since the relative cointegral of $H_n$ has $A=\C$, the $\a$-arcs do not play any role in $Z_{H_n}^{\rho}$, hence we do not consider them. As usual, we note $d=|\aa|=|\bb|$, $l=|\boldb|$ and $\bb=\{\b_1,\dots,\b_d\}$ and $\boldb=\{\b_{d+1},\dots,\b_{d+l}\}$. Since all curves and arcs are oriented and the $\a$ curves have basepoints, we can write a presentation
\begin{align*}
\pi_1(M,*)=\lb \b^*_1,\dots, \b^*_{d+l}\ | \ \ov{\a}_1,\dots,\ov{\a}_d \rb
\end{align*}
where $*$ is some basepoint. Here the $\ov{\a}_1,\dots,\ov{\a}_d$ are words in the $\b^*_1,\dots,\b^*_{d+l}$ constructed as in Definition \ref{def: ov l for l arc}, but without taking homology. Note that to define the $\b^*_i$'s as elements of the fundamental group one needs to be more precise on the basepoint $*$, but it doesn't matter for our purposes, since we will evaluate on an abelian representation $\rho$. Since we only treat abelian representations, and for simplicity of notation, we will consider the Fox derivative $\p\ov{\a_i}/\p \b^*_j$ as an element of $\Z[H_1(M)]$ (strictly speaking, it belongs to $\Z[F]$ where $F$ is the free group generated by $\b^*_1,\dots,\b^*_{d+l}$).

%Thus, a representation $\rho:H_1(M)\to G(B^*)$ can be seen as a group homomorphism $H_1(M)\to \kk^{\t}$ (which we still denote by $\rho$) taking values in the subgroup of $n$-th roots of unity, and viceversa. We can further extend $\rho$ to a ring homomorphism $\Z[H_1(M)]\to \kk$.

\begin{theorem}
\label{thm: intro Z via Fox calculus}
Let $H_n$ be the Borel subalgebra of $\Uqgl11$ at a root of unity of order $n\geq 1$. Let $(M,\c)$ be a balanced sutured manifold with connected $R_-(\c)$ and let $\HH=(\S,\aa,\bbb)$ be an ordered, oriented, based extended Heegaard diagram of $(M,\c)$. If $\rho:H_1(M)\to G(B^*)$ is a group homomorphism, then
\begin{align*}
Z_{H_n}^{\rho}(\HH)=\det\left(\rho_K\left(\frac{\p\ov{\a_i}}{\p\b^*_j}\right)\right)_{i,j=1,\dots,d}
\end{align*}
where the right hand side is computed with the same ordering, orientation and basepoints as $\HH$ and $\rho_K$ (given by $\rho_K(h)=\rho(h)(K)$ for all $h\in H_1(M)$) has been extended to $\Z[H_1(M)]$.
\end{theorem}

We will prove first a preliminary lemma. For any $x\in\a\cap\bb$, let $p(x)\in\a$ be a basepoint defined as in Subsection \ref{subsection: basepoints}, so if the intersection at $x$ is positive (resp. negative), $p(x)$ lies just before $x$ (resp. after $x$) along the orientation of $\a$. If $p$ is the basepoint of $\a$ we denote
\begin{align*}
\ov{\a}_{x}=\prod_{y\in[p,p(x)]\cap\bbb}\b^*_y\in H_1(M)
\end{align*}
where $\b_y$ is the $\b$-curve through the crossing $y$ and $[p,p(x)]$ is the subarc of $\a$ that goes from $p$ to $p(x)$. Thus, these elements satisfy
\begin{align}
\label{eq: Fox derivative of alpha relation}
\frac{\p \ov{\a}}{\p\b^*}=\sum_{x\in\a\cap\b}m_x\ov{\a}_x\in\Z[H_1(M)]
\end{align}
for any $\a\in\aa$ and $\b\in\bbb$, where $m_x$ is the intersection sign at $x$. If $\x=\{x_1,\dots, x_d\}$ is a multipoint of $\HH$,  say $x_i\in\a_i\cap\b_{\s(i)}$ for each $i=1,\dots, d$, we further denote
\begin{align*}
\ov{\a}(\x)=\prod_{i=1}^d\ov{\a}_{x_i}\in H_1(M)
\end{align*}
and
\begin{align*}
m(\x)=\sgn(\s)\prod_{i=1}^d m_{x_i}.
\end{align*}

\begin{lemma}
\label{lemma: Fox calculus formula is a sum over multipoints}
Let $\HH$ be an ordered, oriented, based, extended Heegaard diagram. Then 
\begin{equation*}
\det\left(\frac{\p\ov{\a_i}}{\p\b_j^*}\right)_{i,j=1,\dots, d}=\sum_{\x\in\Tab}m(\x)\ov{\a}(\x)\in\Z[H_1(M)].
\end{equation*}
\end{lemma}
\begin{proof}
This is essentially \cite[Proposition 4.2]{FJR11}. We give a proof here for completeness. Let $Z$ be the Fox calculus matrix of the left hand side. Expand the determinant of $Z$ using the formula
\begin{align*}
\det(Z)=\sum_{\s\in S_d}\sign(\s)Z_{1\s(1)}\dots Z_{d\s(d)}
\end{align*}
without ever cancelling terms. Now, by Equation (\ref{eq: Fox derivative of alpha relation}) each $Z_{i\s(i)}$ is a sum of terms indexed by the points in $\a_i\cap\b_{\s(i)}$. If we put this expression into the above formula for $\det(Z)$, we get an expansion of $\det(Z)$ as a sum indexed by multipoints $\x\in\Tab$. More precisely, if $\x=\{x_1,\dots,x_d\}$ with $x_i\in\a_i\cap\b_{\s(i)}$ for each $i$, then the term corresponding to this multipoint is 
\begin{align*}
\sign(\s)\prod_{i=1}^dm_{x_i}\ov{\a}_{x_i}.
\end{align*}
This is exactly $m(\x)\ov{\a}(\x)$, as desired.
\end{proof}

\begin{proof}[Proof of Theorem \ref{thm: intro Z via Fox calculus}]
Recall that the cointegral of $H_n$ is $\iota=\frac{1}{n}(\sum_{i=0}^{n-1}K^iX)$. Thus, the tensor corresponding to each $\a\in\aa$ is
\begin{figure}[H]

\centering
\begin{pspicture}(0,-0.605)(10.5,0.605)
\rput[bl](0.0,-0.005){$\iota$}
\rput[bl](3.4,-0.005){$K^iX$}
\rput[bl](7.4,-0.005){$X$}
\rput[bl](9.3,0.395){$m_{K^i}$}
\psline[linecolor=black, linewidth=0.018, arrowsize=0.05291667cm 2.0,arrowlength=0.8,arrowinset=0.2]{->}(0.3,0.095)(0.6,0.095)
\rput[bl](0.77,-0.015){$\Delta$}
\rput[bl](1.8,0.045){=}
\psline[linecolor=black, linewidth=0.018, arrowsize=0.05291667cm 2.0,arrowlength=0.8,arrowinset=0.2]{->}(4.3,0.095)(4.6,0.095)
\rput[bl](4.77,-0.015){$\Delta$}
\rput[bl](2.35,-0.455){$\displaystyle\frac{1}{n}\sum_{i=0}^{n-1}$}
\psline[linecolor=black, linewidth=0.018, arrowsize=0.05291667cm 2.0,arrowlength=0.8,arrowinset=0.2]{->}(7.9,0.095)(8.2,0.095)
\rput[bl](8.37,-0.015){$\Delta$}
\psline[linecolor=black, linewidth=0.018, arrowsize=0.05291667cm 2.0,arrowlength=0.8,arrowinset=0.2]{->}(10.1,0.495)(10.4,0.495)
\rput[bl](9.3,-0.405){$m_{K^i}$}
\psline[linecolor=black, linewidth=0.018, arrowsize=0.05291667cm 2.0,arrowlength=0.8,arrowinset=0.2]{->}(10.1,-0.305)(10.4,-0.305)
\rput[bl](6.35,-0.455){$\displaystyle\frac{1}{n}\sum_{i=0}^{n-1}$}
\rput[bl](5.8,0.045){=}
\rput[bl](10.45,-0.605){.}
\psline[linecolor=black, linewidth=0.018, arrowsize=0.05291667cm 2.0,arrowlength=0.8,arrowinset=0.2]{->}(1.2,0.195)(1.6,0.495)
\psline[linecolor=black, linewidth=0.018, arrowsize=0.05291667cm 2.0,arrowlength=0.8,arrowinset=0.2]{->}(1.2,-0.005)(1.6,-0.305)
\rput[bl](1.5,-0.035){$\vdots$}
\psline[linecolor=black, linewidth=0.018, arrowsize=0.05291667cm 2.0,arrowlength=0.8,arrowinset=0.2]{->}(5.2,0.195)(5.6,0.495)
\psline[linecolor=black, linewidth=0.018, arrowsize=0.05291667cm 2.0,arrowlength=0.8,arrowinset=0.2]{->}(5.2,-0.005)(5.6,-0.305)
\rput[bl](5.5,-0.035){$\vdots$}
\psline[linecolor=black, linewidth=0.018, arrowsize=0.05291667cm 2.0,arrowlength=0.8,arrowinset=0.2]{->}(8.8,0.195)(9.2,0.495)
\psline[linecolor=black, linewidth=0.018, arrowsize=0.05291667cm 2.0,arrowlength=0.8,arrowinset=0.2]{->}(8.8,-0.005)(9.2,-0.305)
\rput[bl](9.1,-0.035){$\vdots$}
\end{pspicture}
\end{figure}
\noindent In the last equation we used that $K^i$ is group-like. Now, since $K^i\in B$, by Lemma \ref{lemma: multiply by b on lower handlebody does not affect result}, each $K^i$ has the effect of multiplying the whole tensor by $\lb \rho(\ov{\a}),K^i\rb$ and since $\ov{\a}=1$ in $H_1(M)$, it follows that each $K^i$ does nothing. Hence, we can equally compute $Z_{H_n}^{\rho}(\HH)$ by associating the tensor

\begin{figure}[H]
\centering
\begin{pspicture}(0,-0.3858641)(1.876941,0.3858641)
\rput[bl](0.0,-0.099999845){$X$}
\psline[linecolor=black, linewidth=0.018, arrowsize=0.05291667cm 2.0,arrowlength=0.8,arrowinset=0.2]{->}(0.5,0.0)(0.9,0.0)
\psline[linecolor=black, linewidth=0.018, arrowsize=0.05291667cm 2.0,arrowlength=0.8,arrowinset=0.2]{->}(1.5,0.10000015)(1.9,0.40000015)
\psline[linecolor=black, linewidth=0.018, arrowsize=0.05291667cm 2.0,arrowlength=0.8,arrowinset=0.2]{->}(1.5,-0.099999845)(1.9,-0.39999986)
\rput[bl](1.07,-0.10999985){$\Delta$}
\rput[bl](1.8,-0.12999985){$\vdots$}
\end{pspicture}
\end{figure}

\noindent to each $\a\in\aa$ and contracting with the same $\b$-tensors as before (and antipodes on negative crossings). Let $\a\in\aa$ and $p\in\a$ be its basepoint. By definition of $\De(X)$, the above $\a$-tensor is a sum of as many terms as crossings through $\a$. More precisely, for each $x\in\a\cap\bbb$, this tensor has a term with an $X$ over $x$, a $K$ over each crossing from $p$ to $x$ and the unit 1 over each crossing from $x$ to $p$ as in the following figure:

\begin{figure}[H]
\centering
\begin{pspicture}(0,-0.45)(5.4,0.45)
\definecolor{colour0}{rgb}{0.0,0.8,0.2}
\psline[linecolor=red, linewidth=0.03, arrowsize=0.05291667cm 2.0,arrowlength=0.8,arrowinset=0.2]{->}(0.0,0.2)(4.8,0.2)
\psdots[linecolor=black, dotsize=0.1](0.3,0.2)
\rput[bl](0.25,-0.15){$p$}
\rput[bl](5.2,0.05){$\alpha$}
\rput[bl](1.75,-0.45){$K$}
\rput[bl](2.35,-0.45){$X$}
\rput[bl](3.1,-0.44){$1$}
\rput[bl](4.4,-0.44){$1$}
\rput[bl](2.65,0.3){$x$}
\psline[linecolor=colour0, linewidth=0.03](0.7,0.4)(0.7,0.0)
\psline[linecolor=colour0, linewidth=0.03](1.9,0.4)(1.9,0.0)
\psline[linecolor=colour0, linewidth=0.03](2.5,0.4)(2.5,0.0)
\psline[linecolor=colour0, linewidth=0.03](3.1,0.4)(3.1,0.0)
\psline[linecolor=colour0, linewidth=0.03](4.4,0.4)(4.4,0.0)
\psdots[linecolor=black, dotstyle=o, dotsize=0.1, fillcolor=white](2.5,0.2)
\rput[bl](1.1,-0.4){$\dots$}
\rput[bl](3.6,-0.4){$\dots$}
\rput[bl](0.55,-0.45){$K$}
\rput[bl](5.55,-0.45){.}
\end{pspicture}

\end{figure}

\noindent Now, since $\pi_B(X)=0$, each term having an $X$ over an arc of $\boldb$ vanishes after contracting with the $\b$-tensors, hence it does not contributes to $Z^{\rho}_{H_n}$. Similarly, if two $X$'s corresponding to different $\a$'s lie over a same $\b\in\bb$, then they will vanish after multiplying along $\b$, since $X^2=0$. Thus, the only contributions to $Z^{\rho}_{H_n}$ come from those tensors having an $X$ over a crossing $x_i\in\a_i\cap\b_{\s(i)}$ for each $i=1,\dots, d$ where $\s$ is some permutation of $S_d$. In other words, $Z^{\rho}_{H_n}$ is a sum over multipoints:
\begin{align}
\label{eqproof: Z is a sum over multipoints}
Z^{\rho}_{H_n}(\HH)=\sum_{\x\in\Tab}A'(\x)
\end{align}
Here $A'(\x)$ is obtained by putting, for each $i=1,\dots, d$, an $X$ over $x_i$, a $K$ over each crossing from $p$ to $x_i$ along $\a_i$ and the unit $1$ over each crossing from $x_i$ to $p$ along $\a_i$, and then contracting with antipodes (whenever a crossing is negative) and the $\b$-tensors for $\b\in\bbb$. In view of Lemma \ref{lemma: Fox calculus formula is a sum over multipoints}, the theorem is a consequence of the following:

\begin{claim*}
We have $A'(\x)=m(\x)\rho_K(\ov{\a}(\x))\in\C$.
\end{claim*}

Indeed, by Lemma \ref{lemma: multiply by b on lower handlebody does not affect result} again, the $K$'s before an $x_i$ along $\a_i$ can be turned into 1's at the cost of multiplication by some scalars. More precisely, if $K$ lies over $x\in\a\cap\b_x$, then it multiplies the tensor by $\lb\rho((\b^*_x)^{m_x}),K\rb$. Note that if a crossing $x_i$ is negative, then one has to apply an antipode which is $S(X)=-K^{-1}X$. Thus, there is an extra $K^{-1}$ lying over each negative $x_i$. Thus, the $K$'s lying over $\a_i$ can be turned into 1's at the cost of multiplying the new tensor by $\lb \rho(\b^*_x),K\rb$ for all crossings $x$ from the basepoint of $\a_i$ to $x_i$, including $x_i$ itself whenever it is negative. This is exactly $\lb \rho(\ov{\a}_{x_i}), K\rb$ and multiplying these for $i=1,\dots, d$ gives $\lb\rho(\ov{\a}(\x)),K\rb=\rho_K(\ov{\a}(\x))$. Thus, we have $$A'(\x)=\d \rho_K(\ov{\a}(\x))$$ for some scalar $\d$, which is obtained by turning all the $K$'s of $A'(\x)$ into 1's and contracting. We show now that $\d=m(\x)$. Indeed, $\d$ is computed by putting an $X$ over each $x_i$, but if $x_i$ is negative, a sign appears from $S(X)=-K^{-1}X$ (and only a sign, since we already took care of the $K^{-1}$). Some extra signs appear when permuting $X$'s, since they have degree one. If $x_i\in\a_i\cap\b_{\s(i)}$ for each $i$, then the $X$'s are permuted according to $\s$, so the sign $\sign(\s)$ appears. Thus, $\d$ is the product of the sign
\begin{align*}
\sign(\s)\prod_{i=1}^dm_{x_i}=m(\x)
\end{align*}
together with the contraction of the $\b$-tensors. But since each $\b$-curve has a single $X$ and $\mu(X)=1$ it follows that $\d$ coincides with the above sign. This proves the claim and hence the theorem.

\end{proof}

\subsection{The case $n=0$}\label{subs: The case n=0} When $n=0$, $H_0$ is infinite dimensional hence it has no cointegral. However, in the proof above we showed that for each $n\geq 1$, $Z_{H_n}^{\rho}(\HH)$ can be computed with the element $X\in H_n$ in place of the cointegral $\iota=\frac{1}{n}(\sum_{i=0}^{n-1}K^iX)$ involved in the $\a$-tensors. Hence we can {\em define} $Z_{H_0}^{\rho}(\HH)$ (and $I_{H_0}^{\rho}(M,\c,\ss,\o)$) for an arbitrary $\rho:H_1(M)\to\C^{\t}$ by comultiplying $X$ along the $\a$'s and contracting with the $\b$-tensors. Equivalently, we can consider the expression
\begin{align*}
F(\HH,\ss,\x_0,\o)\eq \d_{\o}(\HH)h_{\ss,\x_0}\det\left(\frac{\p \ov{\a_i}}{\p \b^*_j}\right)_{i,j=1,\dots,d}\in\Z[H_1(M)]
\end{align*}
where $\HH$ is an ordered, oriented, extended Heegaard diagram of $(M,\c)$, based according to a multipoint $\x_0\in\Tab$, $\ss\in\Spinc(M,\c)$, $\o$ is an orientation of $H_*(M,R_-(\c);\R)$ and $h_{\ss,\x_0}=PD[\ss-s(\x_0)]$. Note that the basepoints of $\HH$ are necessary to write down the words $\ov{\a}_i$. Theorem \ref{thm: intro Z via Fox calculus} implies that 
\begin{align}
\label{eq: specialization of ZH0}
\rho(F(\HH,\ss,\x_0,\o))=I^{\rho}_{H_n}(M,\c,\ss,\o)
\end{align}
for any $\rho:H_1(M)\to\C$ having values in the subgroup of $n$-th roots of unity. Since this is true for any such $\rho$ and any $n\geq 1$, it follows that $F(\HH,\ss,\x_0,\o)$ is an invariant of $(M,\c,\ss,\o)$.

\begin{definition}
\label{def: ZH0}
We denote $I_{H_0}(M,\c,\ss,\o)\eq F(\HH,\ss,\x_0,\o)\in\Z[H_1(M)]$. If $R_-(\c)$ is disconnected, then we set $I_{H_0}(M,\c,\ss,\o)\eq I_{H_0}(M',\c',i(\ss),\o')$ where $(M',\c')$ is constructed from $(M,\c)$ by adding 1-handles to $R_-\t I$ so that $R(\c')$ is connected. For an arbitrary $\rho:H_1(M)\to\C^{\t}$ we set $I_{H_0}^{\rho}(M,\c,\ss,\o)=\rho(I_{H_0}(M,\c,\ss,\o))\in\C$.
\end{definition}

Note that whenever $R_-$ is disconnected, $I_{H_0}(M',\c',\ss,\o)$ a priori belongs to $\Z[H_1(M')]$, but Theorem \ref{thm: intro Z recovers Reidemeister torsion} implies that indeed $I_{H_0}\in \Z[H_1(M)]$. %We denote this invariant by $I_{H_0}$ since it specializes to the invariants $I_{H_n}^{\rho}$ for any $n\geq 1$ (by Equation (\ref{eq: specialization of ZH0})) in the same way as the Hopf superalgebra $H_0$ specializes to $H_n$ by setting $K^n=1$.

%The reason why we denote this invariant by $Z_{H_0}$ is the following. Recall that in $H_0$ the generator $K$ has infinite order. Thus, $H_0$ has no cointegral and it does not satisfies the properties we required in Theorem \ref{thm: Z is an invariant} to define sutured manifold invariants. However,

\subsection[Recovering Reidemeister torsion]{Recovering Reidemeister torsion} \label{subs: the sutured Reid torsion}\def\Hq{H_0} We now relate the invariant $I_{H_0}$ of Definition \ref{def: ZH0} to the relative Reidemeister torsion $\tau(M,R_-)$. The torsion $\tau(M,R_-)$ is an element of $\Z[H_1(M)]$ defined up to multiplication by an element of $\pm H_1(M)$ (see \cite[Section 6.2]{Turaev:BOOK2}). The indeterminacy can be removed, but we will not need this. We denote by $\dot{=}$ equality up to multiplication by an element of $\pm H_1(M)$. Since we will mostly consider equalities up to multiplication by $\pm H_1(M)$, we will drop the $\Spinc$ structure and homology orientation from the notation. Thus, we note $$I_{H_0}(M,\c)\in \Z[H_1(M)]/\pm H_1(M).$$ Theorem \ref{thm: intro Z recovers Reidemeister torsion} follows from the following facts (which in turn are the sutured versions of some standard properties of Reidemeister torsion, cf. \cite[Theorem 16.5]{Turaev:BOOK2}). 
\begin{enumerate}
\item \label{FJR11 prop 4.1} \cite[Prop. 4.1]{FJR11}: When $R_-$ is connected, then
\begin{align*}
\tau(M,R_-)\dot{=}\det(\p\ov{\a_i}/\p\b^*_j)
\end{align*}
where we consider the Fox derivatives as elements of $\Z[H_1(M)]$.
\medskip

\item \label{FJR11 lemma 3.20}\cite[Lemma 3.20]{FJR11}: If $(M',\c')$ is constructed by adding 1-handles to $(M,\c)$ as in Subsection \ref{subs: disconnected case}, then $i_*(\tau(M,R_-))=\tau(M',R'_-)$ in $\Z[H_1(M')]$, where $i_*:H_1(M)\to H_1(M')$ is induced by the inclusion.
\end{enumerate}

Note that in \cite{FJR11} the above properties are stated for the {\em sutured torsion} $\tau(M,\c,\ss,\o)$, which is a normalization of $\tau(M,R_-)$ and belongs to $\Z[H_1(M)]$ if $(M,\c)$ is balanced. Since we consider equalities up to multiplication by $\pm H_1(M)$, we can (and do) state these properties for $\tau(M,R_-)$ instead.

\begin{proof}[Proof of Theorem \ref{thm: intro Z recovers Reidemeister torsion}]
If $(M',\c')$ is obtained by adding a (possibly empty) collection of 1-handles to $R_-$ in such a way that $R'_-$ is connected, then
\begin{align*}
\label{eq: Z from tau of FJR}
I_{H_0}(M,\c)\eq I_{H_0}(M',\c')&\dot{=}\det(\p\ov{\a_i}/\p\b^*_j)\\
&\dot{=}\tau(M',R'_-)=i_*(\tau(M,R_-)).
\end{align*}
The second equality is by definition of $I_{H_0}$ (which in turn relies on Theorem \ref{thm: intro Z via Fox calculus}). In the third and fourth equalities we used properties (\ref{FJR11 prop 4.1}) and (\ref{FJR11 lemma 3.20}) above respectively. This proves the theorem.
\end{proof}

\subsection{Corollaries}
\label{subs: corollaries}
We now compute $I_{H_0}(M,\c)$ for the sutured manifolds associated to closed 3-manifolds and link complements. Let $Y(1)$ be the sutured manifold associated to a (pointed) closed 3-manifold $Y$, that is, $Y(1)$ is the complement of an open 3-ball $B$ in $Y$ with a single suture in $\p B$ (see Example \ref{example: sutured from closed}).

\begin{corollary}
\label{corollary: Z of closed 3-mfld}
Let $\rho:H_1(M)\to \C^{\t}$ be a group homomorphism. Then $I_{H_0}^{\rho}(Y(1),\ss,\o)$ is zero if $\rho\nequiv 1$ and equals $\pm |H_1(Y;\Z)|$ if $\rho\equiv 1$ (which is defined to be zero if $b_1(Y)>0$).
\end{corollary}
\begin{proof}
Denote $M=Y\sm B$, the underlying manifold of $Y(1)$ and let $\HH=\HD$ be a Heegaard diagram of $M$. This specifies a cell decomposition $X$ of $M$ with one 0-cell and an equal number $g$ of 1-cells and 2-cells (see Remark \ref{remark: Heegaard diagram gives handle decomposition}). Let $\p_i^{\rho}:C_i^{\rho}(X)\to C_{i-1}^{\rho}(X)$ be the boundary operator of the complex $C_*^{\rho}(X)\eq C_*(\wh{X})\ot_{Z[H_1(M)]}\C$ where $\wh{X}$ is the maximal abelian cover of $X$ and $\C$ is considered as a $\Z[H_1(M)]$-module via $\rho$. For an appropriate choice of basis in $C_*^{\rho}(X)$, $\p_2^{\rho}$ is represented by the matrix $(\rho(\p \ov{\a}_i/\p \b^*_j))$ (see e.g. \cite[Claim 16.6]{Turaev:BOOK2}) while $\p_1^{\rho}$ is given by $(\rho(\b^*_1)-1,\dots,\rho(\b^*_g)-1)$. Note that $\rho\equiv 1$ if and only if $\p_1^{\rho}=0$. By Theorem \ref{thm: intro Z via Fox calculus} it follows that
$$I_{H_0}^{\rho}(M,\c,\ss,\o)=\pm\zeta_{\ss,\x_0}\det(\p_2^{\rho}).$$
But if $\rho\nequiv 1$, then $\p_1^{\rho}\neq 0$ and so $\p_2^{\rho}$ is non-surjective. Hence $I_{H_0}^{\rho}=0$ whenever $\rho\nequiv 1$. If $\rho\equiv 1$ then $C_*^{\rho}(M)=C_*(M)$ and since $\p_1=0$ we get $$I_{H_0}^{\rho\equiv 1}(M,\c,\ss,\o)=\det(\p_2)=\pm |H_1(M)|=\pm |H_1(Y)|$$ as desired.
\end{proof}

Now let $S^3(L)$ be the sutured manifold complementary to an $m$-component link $L\subset S^3$, that is, $S^3(L)=S^3\sm N(L)$ with two oppositely oriented meridians as sutures over each boundary component of a tubular neighborhood $N(L)$ of $L$. The sutured torsion of $S^3(L)$ is given as follows \cite[Lemma 6.3]{FJR11}:
\[
\tau(S^3(L),R_-)\dot{=}\begin{cases}
\De_K & \text{ if } m=1,\\
\prod_{i=1}^m(t_i-1)\De_L & \text{ if } m>1
\end{cases}
\]
where $\De_L$ denotes the multivariable Alexander polynomial of a link. Note that $\De_L$ is an element of $\Z[t_1^{\pm 1},\dots, t_m^{\pm 1}]$ %defined up to multiplication by $\pm t_1^{k_1}\dots t_m^{k_m}$ 
and requires $L$ to be ordered and oriented to be defined (see \cite[Definition 15.2]{Turaev:BOOK2}). The ordering and orientation of $L$ defines a canonical isomorphism $\Z[H_1(S^3(L))]=\Z[t_1^{\pm 1},\dots, t_m^{\pm 1}]$ and the above equality has to be understood under this isomorphism.

\begin{corollary}
\label{corollary: Z0 recovers multivariable Alexander polynomial}
If $(S^3(L),\c)$ is the sutured complement of an $m$-component link $L\subset S^3$, then
\[
I_{H_0}(S^3(L),\c)\dot{=}\begin{cases}
\De_K & \text{ if } m=1,\\
\prod_{i=1}^m(t_i-1)\De_L & \text{ if } m>1.
\end{cases}
\]
\end{corollary}
\begin{proof}
This follows from Theorem \ref{thm: intro Z recovers Reidemeister torsion} together with the above computation of $\tau(S^3(L),R_-)$.
\end{proof}

\bibliographystyle{amsplain}
\bibliography{/Users/daniel/Desktop/TEX/bib/referencesabr}

\providecommand{\bysame}{\leavevmode\hbox to3em{\hrulefill}\thinspace}
\providecommand{\MR}{\relax\ifhmode\unskip\space\fi MR }
% \MRhref is called by the amsart/book/proc definition of \MR.
\providecommand{\MRhref}[2]{%
  \href{http://www.ams.org/mathscinet-getitem?mr=#1}{#2}
}
\providecommand{\href}[2]{#2}
\begin{thebibliography}{10}

\bibitem{Altman:Seifert}
I.~Altman, \emph{Sutured {F}loer homology distinguishes between {S}eifert
  surfaces}, Topology Appl. \textbf{159} (2012), no.~14, 3143--3155.

\bibitem{BCGP}
C.~Blanchet, F.~Costantino, N.~Geer, and B.~Patureau-Mirand,
  \emph{Non-semi-simple {TQFT}s, {R}eidemeister torsion and {K}ashaev's
  invariants}, Adv. Math. \textbf{301} (2016), 1--78.

\bibitem{CC:ontwoinvariants}
L.~Chang and S.~X. Cui, \emph{On two invariants of three manifolds from {H}opf
  algebras}, Adv. Math. \textbf{351} (2019), 621--652.

\bibitem{CKS:relation-WRT-Henn}
Q.~Chen, S.~Kuppum, and P.~Srinivasan, \emph{On the relation between the {WRT}
  invariant and the {H}ennings invariant}, Math. Proc. Cambridge Philos. Soc.
  \textbf{146} (2009), no.~1, 151--163.

\bibitem{FJR11}
S.~Friedl, A.~Juh\'{a}sz, and J.~Rasmussen, \emph{The decategorification of
  sutured {F}loer homology}, J. Topol. \textbf{4} (2011), no.~2, 431--478.

\bibitem{Gabai:foliations}
D.~Gabai, \emph{Foliations and the topology of {$3$}-manifolds}, J.
  Differential Geom. \textbf{18} (1983), no.~3, 445--503.

\bibitem{Juhasz:holomorphic}
A.~Juh\'{a}sz, \emph{Holomorphic discs and sutured manifolds}, Algebr. Geom.
  Topol. \textbf{6} (2006), 1429--1457.

\bibitem{Juhasz:polytope}
\bysame, \emph{The sutured {F}loer homology polytope}, Geom. Topol. \textbf{14}
  (2010), no.~3, 1303--1354.

\bibitem{JTZ:naturality}
A.~Juh\'{a}sz, D.~P. Thurston, and I.~Zemke, \emph{Naturality and mapping class
  groups in {H}eegaard {F}loer homology}, arXiv:1210.4996 (2012).

\bibitem{Kup1}
G.~Kuperberg, \emph{Involutory {H}opf algebras and {$3$}-manifold invariants},
  Internat. J. Math. \textbf{2} (1991), no.~1, 41--66.

\bibitem{Kup2}
\bysame, \emph{Non-involutory {H}opf algebras and {$3$}-manifold invariants},
  Duke Math. J. \textbf{84} (1996), no.~1, 83--129.

\bibitem{LR:cosemisimplechar0}
R.~G. Larson and D.~E. Radford, \emph{Finite-dimensional cosemisimple {H}opf
  algebras in characteristic {$0$} are semisimple}, J. Algebra \textbf{117}
  (1988), no.~2, 267--289.

\bibitem{Murakami:Alexander}
J.~Murakami, \emph{The multi-variable {A}lexander polynomial and a
  one-parameter family of representations of {$U_q(\mathfrak{sl}(2,\C))$} at
  {$q^2=-1$}}, Quantum groups ({L}eningrad, 1990), Lecture Notes in Math., vol.
  1510, Springer, Berlin, 1992, pp.~350--353.

\bibitem{Murakami:state}
\bysame, \emph{A state model for the multivariable {A}lexander polynomial},
  Pacific J. Math. \textbf{157} (1993), no.~1, 109--135.

\bibitem{OS:knot}
P.~Ozsv\'{a}th and Z.~Szab\'{o}, \emph{Holomorphic disks and knot invariants},
  Adv. Math. \textbf{186} (2004), no.~1, 58--116.

\bibitem{OS1}
\bysame, \emph{Holomorphic disks and topological invariants for closed
  three-manifolds}, Ann. of Math. (2) \textbf{159} (2004), no.~3, 1027--1158.

\bibitem{Radford:BOOK}
D.~E. Radford, \emph{Hopf algebras}, Series on Knots and Everything, vol.~49,
  World Scientific Publishing Co. Pte. Ltd., Hackensack, NJ, 2012.

\bibitem{Reshetikhin:supergroup}
N.~Reshetikhin, \emph{Quantum supergroups}, Quantum field theory, statistical
  mechanics, quantum groups and topology ({C}oral {G}ables, {FL}, 1991), World
  Sci. Publ., River Edge, NJ, 1992, pp.~264--282.

\bibitem{RT2}
N.~Reshetikhin and V.~G. Turaev, \emph{Invariants of {$3$}-manifolds via link
  polynomials and quantum groups}, Invent. Math. \textbf{103} (1991), no.~3,
  547--597.

\bibitem{RS:Alexander}
L.~Rozansky and H.~Saleur, \emph{Quantum field theory for the multi-variable
  {A}lexander-{C}onway polynomial}, Nuclear Phys. B \textbf{376} (1992), no.~3,
  461--509.

\bibitem{Sartori:Alexander}
A.~Sartori, \emph{The {A}lexander polynomial as quantum invariant of links},
  Ark. Mat. \textbf{53} (2015), no.~1, 177--202.

\bibitem{Turaev:Spinc}
V.~G. Turaev, \emph{Torsion invariants of {${\rm Spin}^c$}-structures on
  {$3$}-manifolds}, Math. Res. Lett. \textbf{4} (1997), no.~5, 679--695.

\bibitem{Turaev:BOOK2}
\bysame, \emph{Introduction to combinatorial torsions}, Lectures in Mathematics
  ETH Z\"{u}rich, Birkh\"{a}user Verlag, Basel, 2001, Notes taken by Felix
  Schlenk.

\bibitem{Virelizier:flat}
A.~Virelizier, \emph{Involutory {H}opf group-coalgebras and flat bundles over
  3-manifolds}, Fund. Math. \textbf{188} (2005), 241--270.

\bibitem{Viro:Alexander}
O.~Ya. Viro, \emph{Quantum relatives of the {A}lexander polynomial}, Algebra i
  Analiz \textbf{18} (2006), no.~3, 63--157.

\end{thebibliography}

\end{document}